\documentclass[10pt]{amsart}
\usepackage{times}
\usepackage{amsmath}
\usepackage{amsfonts}
\usepackage{amsthm}
\usepackage{amssymb}
\usepackage{accents}

\allowdisplaybreaks
\numberwithin{equation}{section}
\newtheorem{theorem}{Theorem}
\numberwithin{theorem}{section}
\newtheorem*{theorem*}{Theorem}

\newtheorem*{acknowledgement*}{Acknowledgment}
\newtheorem*{definition*}{Definition}
\newtheorem*{notation*}{Notation}

\newtheorem{corollary}[theorem]{Corollary}
\newtheorem*{corollary*}{Corollary}

\newtheorem{definition}[theorem]{Definition}

\newtheorem{lemma}[theorem]{Lemma}
\newtheorem{notation}[theorem]{Notation}

\newtheorem{proposition}[theorem]{Proposition}
\newtheorem{remark}[theorem]{Remark}

\newcommand{\RR}[0]{\mathbb{R}}

\newcommand{\CP}[0]{\mathbb{CP}}

\newcommand{\pd}[2]{\frac{\partial #1}{\partial#2}}

\newcommand{\pdt}[0]{\frac{\partial}{\partial t}}

\newcommand{\gt}[0]{\tilde{g}}
\newcommand{\ft}[0]{\tilde{f}}
\newcommand{\gb}[0]{\overline{g}}
\newcommand{\delb}[0]{\overline{\nabla}}
\newcommand{\delt}[0]{\widetilde{\nabla}}

\newcommand{\Gamt}[0]{\widetilde{\Gamma}}

\newcommand{\ve}[1]{\mathbf{#1}}

\newcommand{\Rc}[0]{\operatorname{Rc}}
\newcommand{\Rm}[0]{\operatorname{Rm}}

\newcommand{\Xt}[0]{\widetilde{X}}
\newcommand{\Rmt}[0]{\widetilde{\operatorname{Rm}}}

\newcommand{\dfn}[0]{\doteqdot}

\newcommand{\ZZ}[0]{\mathbb{Z}}

\newcommand{\Ac}[0]{\mathcal{A}}
\newcommand{\Cc}[0]{\mathcal{C}}
\newcommand{\Sc}[0]{\mathcal{S}}
\newcommand{\Lc}[0]{\mathcal{L}}

\newcommand{\pdtau}[0]{\pd{}{\tau}}

\newcommand{\Qc}[0]{\mathcal{Q}}
\newcommand{\Xc}[0]{\mathcal{X}}
\newcommand{\Yc}[0]{\mathcal{Y}}
\newcommand{\Zc}[0]{\mathcal{Z}}

\newcommand{\cl}[1]{\overline{#1}}
\newcommand{\gc}[0]{\mathring{g}}
\newcommand{\Ps}[0]{\mathring{P}}
\newcommand{\Pb}[0]{\bar{P}}

\newcommand{\Deltat}[0]{\tilde{\Delta}}
\newcommand{\oinf}[0]{o(\infty)}

\newcommand{\gh}[0]{\hat{g}}
\newcommand{\Rt}[0]{\tilde{R}}
\newcommand{\Bt}[0]{\tilde{B}}
\newcommand{\Mt}[0]{\tilde{M}}

\newcommand{\nablat}[0]{\widetilde{\nabla}}
\newcommand{\Hc}[0]{\mathcal{H}}
\newcommand{\Kc}[0]{\mathcal{K}}
\newcommand{\Rct}[0]{\widetilde{\Rc}}
\newcommand{\Gt}[0]{\tilde{G}}
\newcommand{\Dc}[0]{\mathcal{D}}
\newcommand{\Wc}[0]{\mathcal{W}}
\newcommand{\dX}[0]{d\mathfrak{m}}
\newcommand{\ub}[1]{\underaccent{\bar}{#1}}
\newcommand{\cb}[0]{b}
\newcommand{\cnst}[0]{c}

\newcommand{\RP}[0]{\mathbb{RP}}
\newcommand{\Ss}[0]{\mathbb{S}}

\title[Asymptotically cylindrical shrinking solitons]{A uniqueness theorem for asymptotically cylindrical shrinking Ricci solitons}

\author{Brett Kotschwar}
\address{Arizona State University, Tempe, AZ, USA}
\email{kotschwar@asu.edu}

\author{Lu Wang}
\address{California Institute of Technology, Pasadena, CA, USA}
\email{drluwang@caltech.edu}

\thanks{The first author was supported in part by Simons Foundation grant \#359335. The second author was supported in part by NSF grants DMS-2018221 (formerly DMS-1406240) and DMS-2018220, an Alfred P. Sloan research fellowship, and the office of the Vice Chancellor for
Research and Graduate Education at the University of Wisconsin-Madison with funding from the Wisconsin Alumni Research Foundation.}

\begin{document}

\begin{abstract}
We prove that a shrinking gradient Ricci soliton which agrees to infinite order at spatial infinity with one of the standard cylinders $\Ss^k\times \RR^{n-k}$ for $k\geq 2$
along some end must be isometric to the cylinder on that end. When the underlying manifold is complete, it must be globally isometric either to the cylinder or (when $k=n-1$) to its  $\ZZ_2$-quotient.
\end{abstract}

\maketitle

\section{Introduction}
A \emph{shrinking Ricci soliton} is a Riemannian manifold $(M, g)$ for which
\begin{equation}\label{eq:soliton}
 2\Rc(g) + \mathcal{L}_{X}g = g
\end{equation}
for some smooth vector field $X$ on $M$.
The soliton is \emph{gradient} if $X = \nabla f$ for some $f\in C^{\infty}(M)$. When a shrinking soliton is complete and of bounded curvature, it is always possible to find $f$ such that $X - \nabla f$ is Killing \cite{Naber4D, Perelman1}, and so, for most applications,
there is no loss of generality
in considering only gradient shrinking solitons. (By contrast, there are \emph{expanding} Ricci solitons of bounded geometry which are nongradient in an essential way: see, e.g., \cite{BairdDanielo, Lott}.)
Below, we will assume that all shrinking solitons (or, simply, \emph{shrinkers}) are gradient and are normalized to satisfy the equations
\begin{align}
 \label{eq:grs}
  &\Rc(g)+ \nabla\nabla f = \frac{g}{2}, \quad R+ |\nabla f|^2 = f,
\end{align}
on $M$. The contracted second Bianchi identity implies that 
\[
\nabla (R + |\nabla f|^2 - f) \equiv 0
\]
whenever the first equation is satisfied, so 
it is always possible to achieve the normalization in the second equation
by adding a constant to $f$ on each connected component of $M$. 

We will use either $(M, g, X)$ or $(M, g, f)$ to denote a soliton structure, depending on our emphasis. When there is no ambiguity about which particular structure is meant, we will refer to $(M, g)$, or simply $M$, as ``the'' soliton.

Shrinking solitons are of some intrinsic geometric interest both as generalizations of positive Einstein manifolds and as models in the theory of smooth metric measure spaces. 
We are interested here in their connection
to the Ricci flow
\begin{equation}\label{eq:rf}
 \pdt g = -2\Rc(g),
\end{equation}
where they correspond to \emph{shrinking self-similar solutions}: generalized fixed points of the equation which move only under the natural actions of $\RR_+$ and $\operatorname{Diff}(M)$ on the space of metrics.
When $(M, g)$ is complete, the vector field $\nabla f$ is complete \cite{ZhangCompleteness} and the system
\[
 \left\{\begin{array}{rl}
	    \pd{\phi}{t} &= -\frac{1}{t}\nabla f\circ \phi\\
	     \phi_{-1} &= \operatorname{Id}
        \end{array}\right.
\]
may be solved to obtain a family of diffeomorphisms $\phi_t: M\to M$ defined for $t\in (-\infty, 0)$. The family of rescaled pull-backs $g(t) = -t \phi^*_t g$ of the original metric then solve \eqref{eq:rf} 
on $M\times (-\infty, 0)$.

The study of shrinking solitons is central to the analysis of the singular behavior of solutions to the Ricci flow. 
Solutions which develop a singularity at a finite time $T$ are expected ``generically'' to satisfy a so-called Type-I curvature bound $\sup_{M\times[0, T)} (T-t)|\Rm| < \infty$.
From the work of Hamilton \cite{HamiltonSingularities}, Perelman \cite{Perelman1},
\v{S}e\v{s}um \cite{SesumTypeI}, Naber \cite{Naber4D}, and Enders, M\"uller (Buzano), and Topping \cite{EndersMuellerTopping}, it is now known  that, about any point in the high-curvature region of such a Type-I singular solution,
one can extract a sequence of blow-ups converging to a complete nontrivial shrinking gradient Ricci soliton. In this sense, shrinkers represent potential models for the geometry of a solution
in the neighborhood of a developing singularity. It is a fundamental problem to understand what possible forms these model geometries may take.

\subsection{The classification problem for shrinking Ricci solitons}
Shrinking solitons are completely classified in dimensions two and three. In dimension two, Hamilton \cite{HamiltonSurfaces} proved 
that the only complete shrinkers are the flat plane $\RR^2$ with the Gaussian soliton structure and the standard round metrics on $\mathbb{S}^2$ 
and $\RR\mathbb{P}^2$.  (Alternative proofs that compact 2-D shrinkers have constant positive curvature were later given in
\cite{ChowKnopf} and \cite{ChenLuTian}; the latter, being independent of the Uniformization Theorem, can be used with the convergence
results in \cite{Chow2D, HamiltonSurfaces} to show that the Ricci flow uniformizes compact surfaces.) In three dimensions, the combined results of 
Hamilton \cite{HamiltonSingularities}, Ivey \cite{Ivey3DSolitons}, Perelman \cite{Perelman1}, Ni-Wallach \cite{NiWallach},
and Cao-Chen-Zhu \cite{CaoChenZhu} show that the only complete shrinkers are
the Gaussian soliton on $\RR^3$ and finite quotients of the round sphere $\Ss^3$ and standard round cylinder $\Ss^2\times \RR$. 

These classifications are made possible by some additional
a priori structure peculiar to two and three dimensions: in dimension two, orientable gradient solitons are necessarily rotationally symmetric 
(the rotation of the potential vector field $\nabla f$ by the complex structure is a Killing vector field) 
and in dimension three, complete shrinkers are necessarily
of nonnegative sectional curvature (on account of the Hamilton-Ivey estimate \cite{HamiltonSingularities, Ivey3DSolitons}).
In higher dimensions, the class of shrinkers--which includes all Einstein manifolds with positive scalar curvature--is simply too large to expect an exhaustive classification. 

The three-dimensional classification has nevertheless been extended in a variety of directions within limited classes.
For example, the work of Cao-Wang-Zhang \cite{CaoWangZhang}, Eminenti-LaNave-Mantegazza \cite{EminentiLaNaveMantegazza},
Fern\'andez-L\'opez and Garc\'ia-R\'io\cite{FernandezLopezGarciaRio}, Munteanu-Sesum \cite{MunteanuSesum}, Ni-Wallach \cite{NiWallach}, Petersen-Wylie
\cite{PetersenWylie}, and Zhang \cite{ZhangWeyl}, has shown that the only complete shrinkers with vanishing (even harmonic) Weyl tensor
are either the Gaussian soliton $\RR^n$ or finite quotients of $\Ss^n$ or $\Ss^{n-1}\times \RR$. Other results in this direction include the classification
of four-dimensional half-conformally flat shrinking solitons due to  X.-X. Chen and Y. Wang \cite{ChenWang}, later generalized by  H.-D. Cao and Q. Chen \cite{CaoChen} to shrinkers with vanishing Bach tensor in all higher dimensions. 

Some classifications have also been established for shrinkers satisfying additional curvature positivity conditions.
By a theorem of B.-L. Chen \cite{ChenStrongUniqueness} (cf. \cite{ChowLuYang}), every complete shrinker must at least have nonnegative scalar curvature, however, beginning in dimension four,
there are examples that have Ricci curvatures of mixed sign \cite{FeldmanIlmanenKnopf}. 
As a corollary of the work of B\"ohm-Wilking \cite{BohmWilking}, Brendle \cite{BrendleGenRFConvergence}, and Brendle-Schoen \cite{BrendleSchoen}, 
it is known that any compact shrinker whose curvature operator is $2$-positive or which satisfies the so-called $\mathrm{PIC}1$ condition must be a quotient of the round sphere. In four dimensions, 
X. Li, L. Ni, and K. Wang  \cite{LiNiWang} have shown recently that a 
complete gradient shrinker
with positive isotropic curvature must be a quotient of the standard sphere $\Ss^4$ or standard cylinder $\Ss^3\times \RR$. 
In another direction, Munteanu and J. Wang \cite{MunteanuWangPositiveCurvature} 
(generalizing results of Perelman \cite{Perelman2} and Naber \cite{Naber4D} in dimensions three and four) have shown that any complete shrinker with positive sectional curvature must be compact.

The body of literature on shrinking Ricci solitons is too large to adequately summarize here,
and our discussion has left out many important recent results. As entry points to further related work, we refer the reader to \cite{CaoZhou}, \cite{CarrilloNi}, \cite{HaslhoferMueller1}, \cite{HaslhoferMueller2}, \cite{LiWang}, \cite{MunteanuWang1}, and 
the references therein.

 \subsection{Complete noncompact shrinking solitons}  Given the formal similarity of \eqref{eq:grs} to the condition of nonnegative Ricci curvature, the geometry of a noncompact shrinker near infinity might be expected to be comparatively inflexible, constrained by strong and opposing tendencies toward incompleteness and reducibility. A growing body of evidence now appears to support this heuristic, and to suggest that the structural possibilities for the asymptotic geometry of a complete noncompact shrinker may indeed be few.

Every nontrivial complete noncompact shrinking soliton currently known either splits locally as a product or has a single end smoothly asymptotic to a cone. So far,
examples of the latter type are scarce. The construction of Feldman-Ilmanen-Knopf \cite{FeldmanIlmanenKnopf} produces a family of complete $\mathbb{U}(n)$-invariant asymptotically conical K\"ahler shrinkers on the tautological line bundle of $\CP^{n-1}$ with Ricci curvatures of mixed sign. This construction was later generalized by
Dancer-Wang \cite{DancerWang} and Yang \cite{Yang} to line bundles over products of 
K\"ahler-Einstein metrics with positive scalar curvature. These examples, too, have quadratic curvature decay and a single asymptotically conical end.  

In four dimensions, it is conjectured that any complete shrinker must fit one of the two above descriptions, at least asymptotically. The recent work of Munteanu-Wang \cite{MunteanuWang2, MunteanuWang3, MunteanuWang4} allows for a neat phrasing of this proposed dichotomy in terms of the scalar curvature. On one hand, in \cite{MunteanuWang2, MunteanuWang3}, Munteanu and Wang show that if the scalar curvature tends to zero at spatial infinity, then every end of $(M^4, g)$ must be smoothly asymptotic to a cone. On the other hand,
in \cite{MunteanuWang4}, they show that if the scalar curvature is bounded below by a positive constant, then either every end of $(M^4, g)$
is smoothly asymptotic to a quotient of $\Ss^3\times \RR$, or, for any sequence of points $x_i$ going to infinity along an integral curve of $\nabla f$, the sequence of pointed manifolds $(M^4, g, x_i)$
will subconverge in the smooth Cheeger-Gromov sense to a quotient of $\Ss^2\times \RR^2$.
(See also \cite{ChowLu} for a general splitting criterion for limits of pointed sequences of shrinkers.) When $(M^4, g)$ is K\"ahler and the scalar curvature is bounded, it is proven in \cite{MunteanuWang4} that these are the only two alternatives for the scalar curvature.

What connects this proposed dichotomy to a potential classification of complete noncompact four-dimensional solitons -- and what motivates the present paper -- is a question of uniqueness of interest in all dimensions: {\it to what extent is 
a shrinker determined by its asymptotic
geometry?} 

The authors have previously addressed this question for conical asymptotic geometries. In \cite{KotschwarWangConical}, it is shown that if two shrinkers are 
$C^2$-asymptotic to the same cone on some ends of each, then the shrinkers must themselves be isometric to each other near infinity
on those ends. This is an analog of a theorem of the second author for asymptotically conical self-shrinkers to the mean curvature flow \cite{WangConical}, and it reduces the classification 
of asymptotically conical shrinking solitons to that of the potential asymptotic cones.  

At present, there are few restrictions known to hold
on the cones which admit an asymptotic shrinker. Lott-Wilson \cite{LottWilson} have shown that there are at least no formal obstructions to the existence 
of a shrinker or an expander asymptotic to any regular cone,
and it is a consequence of the uniqueness result in \cite{KotschwarWangConical} that any isometry of the cone
must correspond to an isometry of the shrinker. The first author has also shown in \cite{KotschwarKaehlerShrinker} that if the cone is K\"ahler the shrinker must also be K\"ahler.

In this paper, we revisit the above question of uniqueness in the case of asymptotically cylindrical geometries.
\subsection{Asymptotically cylindrical shrinking Ricci solitons}
Let us establish the notation we need to state our main result.
For each $k\geq 2$, we will write $\Cc^k = \Ss^k\times \RR^{n-k}$  and let
\[
 g_k = (2(k-1)\gc)\oplus \bar{g}, \quad f_k(\theta, z) = \frac{|z|^2}{4} + \frac{k}{2},
\]
where $\gc$ is the round metric on $\Ss^k$ of constant sectional curvature $1$ and $\bar{g}$ is the Euclidean metric on $\RR^{n-k}$.
We will call the soliton structure $(\Cc^k, g_k, f_k)$ the \emph{standard cylinder}; the constants in the definitions for $g_k$ and $f_k$ have been chosen so that the normalizations in \eqref{eq:grs} are met.

For each $r > 0$, let $\mathcal{C}^k_r$ denote the end of the cylinder given by
\[
  \mathcal{C}^k_r = \left\{ 
  \begin{array}{rl} 
    \Ss^k \times(\RR^{n-k}\setminus \cl{B_r(0)}) & 2 \leq k < n-1,\\
    \Ss^{n-1}\times (r, \infty) & k = n-1.
    \end{array}\right.
\]
More generally, by an \emph{end} of a Riemannian manifold $(M, g)$, we will mean an unbounded connected component of the complement of a compact set in $M$.

The following definition makes precise the sense in which we mean that a metric ``agrees to infinite order'' with the cylinder at infinity.
\begin{definition}
\label{def:acyl}  
Let $r > 0$.
 We will say that $(\Cc^k_{r}, \gt)$ is \emph{strongly asymptotic to $(\Cc^k, g_k)$}
if, for all $l$, $m\geq 0$,
\begin{equation}\label{eq:ac1}
\sup_{\Cc_{r}^k} \left(|z|^l|\nabla^{(m)}_{g_k}(\tilde{g}-g_k)|_{g_k}(\theta, z)\right) < \infty.
\end{equation} 
We will say that $(\tilde{M}, \tilde{g})$ is \emph{strongly asymptotic to $(\Cc^k, g_k)$ along the end $V \subset (\tilde{M}, \tilde{g})$}
if there exists $r > 0$ and a diffeomorphism $\Psi: \Cc_{r}^k\to V$ such that $(\Cc_{r}, \Psi^*\gt)$ is strongly asymptotic to $(\Cc^k, g_k)$. 
\end{definition}

The main result of this paper is the following local uniqueness result.
\begin{theorem} 
\label{thm:main} 
Suppose $(\tilde{M}, \tilde{g}, \tilde{f})$ is a shrinking gradient Ricci soliton for which $(\tilde{M}, \tilde{g})$ 
is strongly asymptotic to $(\Cc^k, g_k)$ along the end $V\subset (\tilde{M}, \tilde{g})$ for some $k\geq 2$.
Then $(V, \tilde{g}|_{V})$ is isometric to $(\Cc^k_r, g_k|_{\Cc^k_r})$ for some $r > 0$. 
\end{theorem}

The local nature of the statement should be kept in mind when evaluating the strength of the hypothesis of infinite order decay. Only the geometry of the shrinker near infinity on the end $V$ is involved, and $(\tilde{M}, \tilde{g})$ is required neither to be complete nor to satisfy any a priori restriction on the number of its topological ends. In this generality,
there are heuristic reasons to believe the infinite order decay of $\gt - g_k$ may actually be necessary. The theorem is an analog of an earlier result of the second author \cite{WangCylindrical} for the mean curvature flow, which shows that an embedded self-shrinker which is asymptotic of infinite order to one of the standard cylinders must actually coincide with the cylinder. In this case, the assumption of infinite order decay is known to be effectively optimal-- the paper \cite{WangCylindrical} includes the construction of
a family of self-shrinkers on $\Ss^{n-1}\times (a, \infty)\hookrightarrow \RR^{n+1}$
 which are not themselves rotationally symmetric but which nevertheless decay to the cylinder at arbitrarily high fixed polynomial rates. 

When the underlying manifold $(\tilde{M}, \tilde{g})$ 
is complete, however, one expects to be able to say more;  in this case, Theorem \ref{thm:main} implies that $(\tilde{M}, \tilde{g})$ must be globally isometric to a quotient of $(\Cc^k, g_k)$. 
\begin{corollary}
\label{cor:maincor}
 Suppose that, in addition to the assumptions in Theorem \ref{thm:main}, the manifold $(\tilde{M}, \tilde{g})$ is complete.
 Then, either $(\tilde{M}, \tilde{g})$ is isometric to $(\Cc^k, g_k)$, or $k = n-1$ and $(\tilde{M}, \tilde{g})$ is isometric to the quotient $(\Cc^{n-1}, g_{n-1})\slash \Gamma$ where $\Gamma = \{\operatorname{Id}, \gamma\}$
 and $\gamma(\theta, z) = (-\theta, -z)$.
\end{corollary}
The techniques of this paper are rather specialized to address the local problem of uniqueness in Theorem \ref{thm:main}. We expect that when $(\tilde{M}, \tilde{g})$ is complete, it should be possible to weaken (or eliminate entirely) the assumption on the rate of convergence to the cylinder. In fact, even here we haven't fully optimized the formulation of condition \eqref{eq:ac1}, nor do we really require its full strength  to obtain the conclusion of Theorem \ref{thm:main}. For example,
using an interpolation argument, it isn't hard to see that a metric $\gt$ on $\Cc^k_r$ is strongly asymptotic to $(\Cc^k, g_k)$
provided only that
\begin{equation}\label{eq:ac12}
\sup_{\Cc_{r}^k} |z|^l|(\tilde{g}-g_k)|_{g_k}(\theta, z) < \infty, \quad \sup_{\Cc_{r}^k} |\nabla^{(m)}_{g_k}(\tilde{g}-g_k)|_{g_k}(\theta, z) < \infty,
\end{equation}
for all $l \geq 0$ and $m\geq 1$.  An inspection of the proof shows, moreover, that our argument actually only requires
that the pull-back of the metric $\gt$ satisfy \eqref{eq:ac12} for $m$ less than some universal constant $m_0$.

\subsection{Overview of the proof}
As in \cite{KotschwarWangConical}, \cite{WangCylindrical}, our basic strategy is to use the correspondence between shrinkers and self-similar solutions to transform Theorem \ref{thm:main} into an equivalent problem of parabolic unique continuation for solutions to the Ricci flow,
which we ultimately treat with the method of Carleman inequalities.  However, the resulting problem of unique continuation we face here -- for a nonlinear, weakly parabolic system \emph{at the singular time} -- is more complicated than those addressed in either \cite{KotschwarWangConical} or \cite{WangCylindrical}. (In \cite{KotschwarWangConical}, by contrast, the 
problem is fundamentally nonsingular since the solutions extend smoothly to the terminal time slice-- in that case, the end of the common asymptotic cone.  In \cite{WangCylindrical}, the problem, though singular in a similar way, reduces to the analysis of a solution to scalar parabolic inequality.) Our implementation of this strategy involves a number of new ingredients needed to overcome obstacles not present in these related problems.  We summarize the major steps in our proof now.

For the remainder of this section, we will assume that $k\geq 2$ is fixed and write simply
$\Cc= \Cc^k$, $\Cc_r= \Cc_r^k$, $g = g_k$, and $f = f_k$, using $|\cdot| = |\cdot|_{g_k}$ and $\nabla = \nabla_{g_k}$ to denote the norms and connections induced by $g$ and its Levi-Civita connection
on tensor bundles over $\Cc$. 

\subsubsection{Normalizing the soliton structure}
It is sufficient to prove Theorem \ref{thm:main} in the case that $\gt$ and $\ft$ are actually defined on $\Cc_{r_0}$ for some $r_0 > 0$, 
that is, when $(\Cc_{r_0}, \gt)$ is strongly asymptotic to $(\Cc, g)$. Taking this as our starting point, our first concern is to put the entire soliton structure $(\Cc_{r_0}, \gt, \ft)$ into a canonical form. The hypotheses of Theorem \ref{thm:main} only explicitly constrain the asymptotic behavior of $\gt$, and, by themselves, do not even guarantee that the difference of $\Xt$ and $X = \nabla f$ tends to zero at infinity.  

In Proposition \ref{prop:xnorm}, we first show that we can arrange for $\Xt - X$ to vanish to infinite order at infinity by pulling back $\gt$ and $\Xt$ by an appropriate translation on the Euclidean factor. We then show in Theorem \ref{thm:vfnormalization} that it is possible to construct a further injective diffeomorphism 
$\Phi: \Cc_{r_1}\to \Cc_{r_0}$ for some $r_1 > r_0$ such that $\Phi^*\Xt = X$  and for which $(\Cc_{r_1}, \Phi^*\gt)$ is still strongly
asymptotic to $(\Cc, g)$. To ensure the latter property requires that we construct $\Phi$ with some care. We postpone the details of this construction (which are independent of the rest of the paper) to Appendix \ref{app:normalization}. 

\subsubsection{Reducing to a problem of backward uniqueness}
Having reduced Theorem \ref{thm:main} to the case that $\Xt$ and $X$ coincide
on $\Cc_{r_1}$ for some $r_1 > 0$, our next step is to recast it as a problem of parabolic unique continuation for solutions
to the Ricci flow. The family of diffeomorphisms $\Psi: \Cc_{r_1}\times (0, 1]\to \Cc_{r_1}$ given by $\Psi_{\tau}(\theta, z) = (\theta, z/\sqrt{\tau})$
solve
\[
  \pd{\Psi}{\tau} = -\frac{1}{\tau}X \circ \Psi, \quad \Psi_{1} = \operatorname{Id},
\]
and (since $X = \nabla f = \delt \ft)$, we may use them to construct
from $\gt$ and $g$ smooth self-similar families of metrics
\[
\gt(\tau) = \tau \Psi_{\tau}^*\gt, \quad g(\tau) = \tau \Psi_{\tau}^*g = (2(k-1)\tau \gc) \oplus \bar{g},
\]
which solve the \emph{backward Ricci flow}
\begin{equation}
 \label{eq:brf}
 \pd{g}{\tau} = 2\Rc(g)
\end{equation}
on $\Cc_{r_1}$ for $\tau \in (0, 1]$. The normalizations we have performed to this point ensure that the difference $h(\tau) = (\gt - g)(\tau) = \tau\Psi_{\tau}^*(h(1))$ of these solutions is itself self-similar. This will be critical to us in Section \ref{sec:backwarduniqueness}. 
Moreover, since $(\Cc_{r_1}, \gt)$ is strongly asymptotic to $(\Cc, g)$, the tensor $h$ will vanish  to infinite order as $|z|\to \infty$ and $\tau \searrow 0$ in the sense that
\[
 \sup_{\Cc_{r_1}\times (0, 1]} \frac{|z|^{2l}}{\tau^l}|\nabla^{(m)}h|(\theta, z, \tau) <\infty
\]
for all $l$, $m\geq 0$.  Here and below, we write $|\cdot| = |\cdot|_{g(\tau)}$ and $\nabla = \nabla_{g(\tau)}$ (in fact, the connection $\nabla_{g(\tau)}$ of the evolving cylinder is independent of time).

To prove Theorem \ref{thm:main}, then, it is enough to show that $h(\tau_0) \equiv 0$ on $\Cc_{r}$ for some $\tau_0$ and $r > 0$. For, if so, $\tilde{g}(1) - g(1) = h(1) = \tau_0^{-1}(\Psi_{\tau_0}^{-1})^*h(\tau_0)$
vanishes on $\Cc_{r^{\prime}}$ for $r^{\prime} = r/\sqrt{\tau_0}$, and it follows from a continuation argument that $\gt$ and $g$ are isometric on $\Cc_{r_0}$. We give this parabolic restatement in Theorem \ref{thm:mainp} and verify that it indeed implies Theorem \ref{thm:main}
at the end of Section \ref{sec:reduction}.  

\subsubsection{Prolonging the system} To prove Theorem \ref{thm:mainp}, we must first address the lack of strict parabolicity
 of equation \eqref{eq:rf}. The degeneracy of the equation, a consequence of its diffeomorphism invariance, is not rectifiable here by the use of DeTurck's trick 
 as it is in the problem of forward uniqueness of solutions to the Ricci flow:
 the diffeomorphisms needed to pass to a problem of backward uniqueness for the strictly parabolic Ricci-DeTurck flow are  naturally solutions to a ill-posed \emph{terminal-value} problem 
 for a  harmonic map-type heat flow. See, e.g., \cite{KotschwarBackwardsUniqueness} for  a discussion of these and related issues.

To work around the degeneracy of \eqref{eq:rf}, we instead employ a device used by the first author in \cite{KotschwarBackwardsUniqueness} which encodes the vanishing of $h$ in terms of the vanishing of solutions to a prolonged ``PDE-ODE'' system of differential inequalities. The implementation of this device, however, is rather more involved than in \cite{KotschwarBackwardsUniqueness} and \cite{KotschwarWangConical} since the system used in these references turns out to be slightly too coarse to track by itself the blow-up which here occurs anisotropically at the singular time. We instead make use of two prolonged systems: a ``basic'' system which, on account of its relative simplicity, we use to frame and prove the backward uniqueness theorem which implies the vanishing of $h$, and
a ``refined'' system whose higher granularity allows us to track the blow-up rate of individual components of $\nablat \Rmt$.

The basic system is equivalent to those considered in 
\cite{KotschwarBackwardsUniqueness, KotschwarWangConical}, and consists of the families of sections
\[
\ve{X} = \delt \Rmt = \delt\Rmt - \nabla \Rm, \quad \ve{Y} = (h, \nabla h, \nabla \nabla h),
\]
of $\mathcal{X} = T^{(5, 0)}\Cc_{r_1}$ and $\mathcal{Y} = T^{(2, 0)}\Cc_{r_1}\oplus T^{(3, 0)}\Cc_{r_1}\oplus T^{(4, 0)}\Cc_{r_1}$, respectively.
These
sections satisfy a system of inequalities of the form
\begin{align*}
\left|\left(D_{\tau} + \Delta\right) \ve{X}\right| &\leq \frac{B}{\tau}|\ve{X}| + B|\ve{Y}|, \quad
 \left|D_{\tau} \ve{Y}\right| \leq B\left(|\ve{X}| + |\nabla \ve{X}|\right) + \frac{B}{\tau}|\ve{Y}|,
\end{align*}
for some constant $B$ on $\Cc_{r_1}\times (0, 1]$. Here, $\Rm = \Rm(g(\tau))$, $\Rmt = \Rm(\gt(\tau))$, $\nablat = \nabla_{\gt(\tau)}$, and $\Delta = \Delta_{g(\tau)}$, and $D_{\tau}$ indicates
a derivative taken relative to evolving $g(\tau)$-orthogonal frames. We describe this system in greater detail and derive the above equations in Section \ref{sec:pdeode}.

However, our basic system is inadequate for what is perhaps the most important step in the proof of Theorem \ref{thm:mainp}: to parlay the infinite order decay that we assume on $h$ and its derivatives (and hence on $\ve{X}$ and $\ve{Y}$) into an exponential-quadratic rate of decay for $\ve{X}$ and $\ve{Y}$ (and hence on $h$ and its derivatives). The Carleman estimate \eqref{eq:pdecarleman2} we use for this purpose cannot directly absorb the coefficient of $\tau^{-1}$ which appears on the right side of the equation 
for $\ve{X}$. 

In Section \ref{sec:pdeode2}, we will
replace the parabolic component $\ve{X}$ of our basic system with a more elaborate choice $\ve{W} = (W^0, W^1, \ldots, W^5)$ in an attempt to address this issue. The components $W^i$ consist of collections of components of $\delt \Rmt$ (relative to the $g$-orthogonal splitting
$TM = T\Ss^k \oplus T\RR^{n-k}$) rescaled by powers of $\tau$ which together satisfy a system of the form
\begin{align}\label{eq:weqvar}
 \left|\left(D_{\tau} + \Delta\right)W^i\right| & \leq B\tau^{\beta}(|\ve{W}| + |\ve{Y}|) + B\sum_{j < i}\tau^{-\gamma_{j}}|W^j|  
\end{align}
for some nonnegative constants $\beta$, $\gamma_j$, and $B$. The strict triangular structure of the singular terms in \eqref{eq:weqvar} allows us
to absorb the unbounded coefficients on the right side of
the equation for any $W^i$ using appropriately weighted applications of the inequalities for $i^{\prime} < i$.

\subsubsection{Promoting the rate of decay to exponential}
The Carleman inequalities \eqref{eq:pdecarlemanbu} and \eqref{eq:odecarlemanbu} we ultimately use to prove the vanishing of $\ve{X}$ and $\ve{Y}$ involve a weight which, for large $|z|$ and small $\tau$, grows on the order of $\exp(C |z|^{2\delta}/\tau^{\delta})$ for some $\delta\in (0, 1)$.
In order to apply these inequalities, we first need to verify that $\ve{X}$ and $\ve{Y}$ decay rapidly enough to be integrable against this weight.
To this end, in Theorem \ref{thm:expdecay1} (proven in Section \ref{sec:expdecay}) we show
that there are constants $N_0$, $N_1 > 0$ such that 
\begin{equation*}
    \int^1_0\int_{\Ac_{r, 2r}}\left(|\ve{X}|^2 + |\nabla\ve{X}|^2 + |\ve{Y}|^2\right)e^{\frac{N_0r^2}{\tau}}\,d\mu_{g(\tau)}\,d\tau \leq N_1 
 \end{equation*}
 for all sufficiently large $r$. Here $\Ac_{r, 2r} = \Cc_{r} \setminus \overline{\Cc_{2r}}$. This argument, including the derivation of the system \eqref{eq:weqvar} above,
is perhaps the most delicate in the paper.  

We establish the decay of $\ve{W}$ and $\ve{Y}$ inductively, using the Carleman inequality \eqref{eq:pdecarleman2} in tandem with \eqref{eq:odecarleman2g} and 
\eqref{eq:odecarleman2ng} to obtain upper bounds  of the form $CL^mr^{-2m}m!$ of successively higher order on the
weighted $L^2$-norms of $\ve{W}$ and $\ve{Y}$ on $\Ss^k\times B_{r}(z_0)$ for small $r$ and $z_0\in \Cc_{r_0}$. These estimates involve
a weight approximately of the form $\tau^{-m}\exp(-|z-z_0|^2/(4\tau))$ localized about $z_0$.  Since the components of $\ve{W}$ are merely rescaled components of $\delt\Rmt$,
 the estimates on $\ve{W}$ directly yield
 corresponding estimates for $\ve{X}$, which can be summed and rescaled to obtain the asserted rate of exponential decay.
The main inequality \eqref{eq:pdecarleman2}, analogous to one established by the second author in \cite{WangCylindrical}, is ultimately modeled on the inequality proven in \cite{EscauriazaSereginSverakHalfSpace} for an application to solutions to linear parabolic inequalities on Euclidean half-spaces.

\subsubsection{Establishing the vanishing of $\ve{X}$ and $\ve{Y}$} In Section \ref{sec:backwarduniqueness}, we return to an analysis of the basic system. Knowing now that $\ve{X}$ and $\ve{Y}$ decay at an at-least exponential-quadratic rate, we
use Carleman inequalities analogous to those in \cite{KotschwarBackwardsUniqueness} and \cite{WangCylindrical} to show that they must vanish identically. This part of the argument is modeled closely on the corresponding argument in \cite{WangCylindrical}, with some modifications to handle the ODE component $\ve{Y}$, and it is here that we make essential use of the self-similarity of $h$ (and hence of $\ve{X}$ and $\ve{Y})$. The Carleman inequalities needed here and above
in the proof of the exponential decay of $\ve{X}$ and $\ve{Y}$ are proven
in Section \ref{sec:carleman}.

\begin{acknowledgement*} The authors wish to thank Ben Chow, Ovidiu Munteanu, Lei Ni, and Jiaping Wang for their interest, encouragement, and valuable suggestions. They also wish to thank the anonymous referees
for their many helpful recommendations. 
\end{acknowledgement*}

\section{Normalizing the soliton}\label{sec:normalization}
Let us now fix $1 < k < n$ once and for all, and, for the rest of the paper, continue to write simply
$\Cc = \Cc^k = \Ss^k\times \RR^{n-k}$ and $\Cc_r = \Cc_r^k$. For $a$, $b$, $r > 0$, we define
\[
\Ac_{a, b} = \Cc_a \setminus \overline{\Cc_b}, \quad
 \Sc_r = \left\{
\begin{array}{rl}
 \Ss^k \times \partial B_r(0) & k < n-1,\\
 \Ss^{n-1}\times \{r\} & k = n-1.
 \end{array}\right.
\]
We will also continue to use 
\[
  g = g_k = (2(k-1)\gc) \oplus \bar{g}, \quad f(\theta, z) = f_k(\theta, z) = \frac{|z|^2}{4} + \frac{k}{2},
\]
to denote the metric and potential of the normalized cylindrical soliton structure on $\Cc$ and to use the unadorned notation
\[
  |\cdot| = |\cdot|_{g}, \quad \nabla = \nabla_{g},
\]
for the norms and connections induced by $g$ and its Levi-Civita connection on the tensor bundles $T^{(p, q)}\Cc$.

Frequently, we will use spherical coordinates on the Euclidean factor $\RR^{n-k}$ to identify $\Cc_a$ with 
$\Ss^k \times \Ss^{n-k-1}\times (a, \infty)$
via $(\theta, z) \mapsto (\theta, \sigma, r)$, where $\sigma = z/|z|$ and $r = |z|$.

\subsection{Some preliminary estimates}
 To prove Theorem \ref{thm:main}, 
it suffices to consider the situation that $M = V = \Cc_{r_0}$ for some $r_0 > 0$ and $(\Cc_{r_0}, \gt)$ is strongly asymptotic to $(\Cc, g)$. We first record some elementary consequences of
\eqref{eq:ac1} for the soliton metric $\gt$.
\begin{lemma}\label{lem:decayconseq}
Suppose that $(\Cc_{r_0}, \gt, \ft)$ is a shrinking Ricci soliton where
\begin{equation}\label{eq:decay1}
 \sup_{\Cc_{r_0}} r^3|\nabla^{(m)}(\gt - g)| < \infty 
\end{equation}
for $m =0, 1, 2$. Then there are $r_1 \geq r_0$ and $k_0$, $K_0 > 0$
such that
\begin{equation}\label{eq:decayconseq1}
 \frac{1}{2} g \leq \gt \leq 2g, \quad |\delt \ft| \leq K_0(r+ 1), \quad |\nabla \ft| \leq K_0(r+1),
 \end{equation}
and  
\begin{equation}\label{eq:decayconseq2}
\frac{1}{8}r^2 \leq \tilde{f} \leq \frac{1}{4}(r + k_0)^2,
\end{equation}
on $\Cc_{r_1}$. 
\end{lemma}
Quadratic bounds for the potential with sharp coefficient $1/4$ have been established for general complete shrinking solitons by Cao-Zhou \cite{CaoZhou} (see also \cite{HaslhoferMueller1}). The weaker
bounds for $\ft$ in \eqref{eq:decayconseq2} (which we must verify from scratch, given the incompleteness of $\Cc_{r_0}$) will be sufficient for our purposes, however.
\begin{proof}
It follows directly from \eqref{eq:decay1} that we can arrange for  $(1/2)g \leq \gt \leq 2g$ and $\Rt \geq k/4$ to hold on $\Cc_{a}$ by choosing $a\geq r_0$ large enough. 
The identity $\Rt + |\delt \ft|_{\gt}^2 = \ft$ then implies that we will have  $\ft \geq k/4$ and $|\nabla \ft |^2 \leq 2|\delt \ft|^2_{\gt} \leq 2\ft$ on the same set.
Integrating along along integral curves of $\pd{}{r}$ we then see that
\begin{equation}\label{eq:ftest}
 \tilde{f}^{1/2}(\theta, \sigma, r) - \tilde{f}^{1/2}(\theta, \sigma, a) \leq \int_a^r|\nabla \ft^{1/2}|\leq r - a,
\end{equation}
for all $(\theta, \sigma)\in \Ss^k\times \Ss^{n-k-1}$.
In particular, $|\delt \ft| \leq 2|\nabla \ft|\leq 4(r+K)$ on $\Cc_{a}$ for some $K$ depending on $\sup_{\Cc_{a}} \ft$. This proves the the last two inequalities in \eqref{eq:decayconseq1} provided $r_1 \geq a$.

Next, using the soliton equation, we have
\begin{align*}
\nabla_i \nabla_j \ft & = \nabla_i \nabla_j \ft - \delt_i\delt_j \ft - \tilde{R}_{ij} + \frac{\tilde{g}_{ij}}{2} \\
&= (\Gamt_{ij}^k - \Gamma_{ij}^k)\nabla_k \ft - (\tilde{R}_{ij} - R_{ij}) + \frac{1}{2}(\gt_{ij} - g_{ij})  -R_{ij} + \frac{g_{ij}}{2}\\
&= A_{ij}^k\nabla_k \ft + S_{ij}  -R_{ij} + \frac{g_{ij}}{2},
\end{align*}
where $A_{ij}^k$ and $S_{ij}$ are polynomials in $g^{-1}$, $\gt^{-1}$, and $\nabla^{(m)}(\gt - g)$ for $m \leq 2$. 
So, using \eqref{eq:decay1} and that $|\nabla \ft| \leq 4(r + K)$, we have
\begin{equation}\label{eq:hess1}
\frac{1}{2} - \frac{K}{r^2}  \leq \frac{\partial^2 \ft}{\partial r^2} \leq \frac{1}{2} + \frac{K}{r^2},
\end{equation}
for some possibly larger $K$. Integrating both inequalities in \eqref{eq:hess1} along integral curves of $\pd{}{r}$ starting at points in $\Sc_a$,
we obtain
\[
\frac{r}{2} - K^{\prime} \leq \left\langle \nabla \ft, \pd{}{r} \right\rangle \leq \frac{r}{2} + K^{\prime}, 
\]
for some $K^{\prime} > 0$ depending on $a$. Hence
\[
  \frac{r^{2}}{4} - K^{\prime}r - \frac{a^2}{4}\leq \ft(\theta, \sigma, r) \leq  \frac{r^{2}}{4} + K^{\prime}r  + \ft(\theta, \sigma, r_1) \leq \frac{r^2}{4} + K^{\prime}r + (r_1 + K^{\prime\prime})^2,
\]
for any $r_1 \geq a$ and some $K^{\prime\prime}$ depending on $a$. Here we have used \eqref{eq:ftest} to estimate $\ft(\theta, \sigma, r_1)$. Choosing then $r_1 \geq a$ large enough to ensure that the left
side is larger than $r^2/8$ on $\Cc_{r_1}$, and then choosing $k_0$ large enough depending on $r_1$ to bound the right side by $(r+ k_0)^2/4$, we obtain \eqref{eq:decayconseq2}.
\end{proof}

\subsection{Correcting the vector field by a translation} Our next step is motivated by the observation that
the assumption
that $(\Cc_{r_0}, \tilde{g})$ is strongly asymptotic to $(\Cc, g)$ -- even with the implicit normalizations in \eqref{eq:grs} -- does not uniquely determine the vector field $\delt\ft$ in the soliton structure $(\Cc_{r_0}, \gt, \delt\ft)$.  In general, the difference $\delt \ft  - \nabla f$ need not tend to zero as $|z|\to \infty$, much less decay to infinite order.  

For example,
the soliton structure $(\Cc, g, f_{z_0})$ with the potential
\[
 f_{z_0}(\theta, z) = \frac{|z-z_0|^2}{4} + \frac{k}{2}
\]
satisfies \eqref{eq:grs}  for any $z_0\in \RR^{n-k}$ (and $(\Cc_{0}, g)$ is, of course, strongly asymptotic to itself), but the difference
\[
 \nabla f - \nabla f_{z_0} = \sum_{i=1}^{n-k}\frac{z^i_0}{2}\pd{}{z^i}
\]
is constant. At the same time, the two soliton structures here can be made to agree by pulling back one by a suitable translation of the Euclidean factor. 

We show next that a similar adjustment can be made in general: by pulling back $\gt$ and $\ft$ by an appropriate translation of $\RR^{n-k}$, we can arrange for $\delt \ft  - \nabla f$ to decay to infinite order
 at infinity. Since the translation is an isometry of $g$,
 the pullback of $\gt$ will still be strongly asymptotic to $g$ on some neighborhood of infinity of the end.

\begin{proposition}\label{prop:xnorm}
Let $p \geq 2$ and suppose that $(\Cc_{r_0}, \gt, \ft)$ satisfies \eqref{eq:grs} and 
\begin{equation}\label{eq:decay2}
 \sup_{\Cc_{r_0}} r^l|\nabla^{(m)}(\gt - g)| < \infty 
\end{equation}
for all $l\geq 0$ and $m\leq p$. Then there is a constant vector field $V$ tangent to the $\RR^{n-k}$ factor 
such that
 \begin{equation}
    \delt \ft = \frac{r}{2}\pd{}{r} + V  +  E,
 \end{equation}
where $E$ satisfies
\begin{equation}\label{eq:edecay1}
\sup_{\Cc_{r_0}}r^{l}|\nabla^{(m)} E| < \infty
\end{equation}
for all $l\geq 0$ and $0\leq m \leq p-1$.
\end{proposition}
\begin{proof}

Let $X = \nabla f = \frac{r}{2}\pd{}{r}$ and $\tilde{X} = \delt \ft$. From \eqref{eq:grs}, we compute that
\begin{align*}
  \nabla_i \tilde{X}^j &= \delt_i\tilde{X}^j+ (\Gamma_{ik}^j- \Gamt_{ik}^j)\tilde{X}^k =\nabla_iX^j + (g^{jk}R_{ik} - \gt^{jk}\tilde{R}_{ik}) + (\Gamma_{ik}^j- \Gamt_{ik}^j)\tilde{X}^k.
\end{align*}
Using \eqref{eq:decay2} and that $|\tilde{X}| \leq K_0(r+1)$ from Lemma \ref{lem:decayconseq}, we thus see that
 $W = \tilde{X} - X$ satisfies
\[
 \sup_{C_{r_0}} r^l|\nabla^{(m)} W|  < \infty
\]
for all $l\geq 0$ and $1\leq m \leq p-1$.

Fix any $q = (\theta, z) \in \Cc_{r_0}$, and let $\{F_{q, i}\}_{i=1}^n$ be any orthonormal basis for $T_q \Cc$. Extend this basis by parallel transport to a frame $\{F_{q, i}(r)\}_{i=1}^n$
along the radial line $\gamma_q(r) = (\theta, rz/|z|)$.  For any $|z| \leq r_1 \leq r_2$, and any $l\geq 0$, we have
\begin{align}
\begin{split}\label{eq:wintest}
   \left|\langle W, F_{q, i}\rangle(\gamma_q(r_2)) - \langle W, F_{q, i}\rangle(\gamma_q(r_1))\right| &\leq \int_{r_1}^{r_2}|\nabla W|(\gamma_q(r))\,dr \leq  \frac{M_l}{r_1^{l}}
   \end{split}
\end{align}
for some $M_l$, and it follows that, for each $i =1, 2, \ldots n$, we have
\[
 \lim_{r\to\infty} \langle W, F_{q, i}\rangle (\gamma_q(r)) = V^i(q) < \infty
\]
for some numbers $V^i(q)$. Define 
\[
    V(q) = V^i(q) F_{q, i} \in T_q\Cc,
\]
and suppose we repeat this process starting from another orthonormal basis $\{\tilde{F}_{q, i}\}_{i=1}^n$. 
Then $\tilde{F}_{q, i}(r) = A_{i}^j F_{q, j}(r)$ for some fixed
orthogonal transformation
$A$, and
\[
 \tilde{V}^i(q) = \lim_{r\to\infty} \langle W, \tilde{F}_{q, i}\rangle (\gamma_q(r))
 = (A^T)^i_j V^{j}(q),
\]
so  the limit $V(q)= \tilde{V}(q)$ depends only on $q$. Taking such a limit at each $q$ thus defines a (rough) vector field on $\Cc_{r_0}$. 

By construction,  for all $\theta$ and $\sigma$ and all $r_0\leq r_1\leq r_2$, the value of $V(\theta, \sigma, r_2)$ will coincide with that of the parallel transport of $V(\theta, \sigma, r_1)$ along the radial line
connecting $(\theta, \sigma, r_1)$ and $(\theta, \sigma, r_2)$. 

We claim that $V$ is actually parallel. To see this, fix any $(\theta, \sigma)$ and  $(\tilde{\theta}, \tilde{\sigma})$ in $\Ss^k\times \Ss^{n-k-1}$ and any $r_1\geq r_0$. For $r \geq r_1$, consider the points
$q_r = (\theta, \sigma, r)$, $\tilde{q}_r = (\tilde{\theta}, \tilde{\sigma}, r)\in \mathcal{S}_r$. 
Let $\alpha:[0, 1]\to \Ss^k\times \Ss^{n-k-1}$ be any smooth path with $\alpha(0) = (\sigma, \theta)$ and $\alpha(1) = (\tilde{\sigma}, \tilde{\theta})$. Then, for $r\geq r_1$, define the curve $\lambda_r:[0, 1]\to \Sc_r\subset \Cc$ by 
$\lambda_r(s) = (\alpha(s), r) \in \Sc_r$ for $r \geq r_1$. Note that the speed of $\lambda_r$ will be bounded by $C_0(r+1)$ for some $C_0$ depending on $\alpha$. 

For each $r\geq r_1$ and $s\in [0, 1]$, Let $P_{r; s}:T_{q_r}\Cc\to T_{\lambda_r(s)}\Cc$ denote parallel translation
along $\lambda_r$. We claim that $P_{r_1; 1}(V(q_{r_1})) = V(\tilde{q}_{r_1})$. For this, observe first that the vector field $W$ above is bounded on account of the decay of $|\nabla W|$, and, by the definition of $V$ and equation \eqref{eq:wintest}, we have
\begin{equation}\label{eq:west}
|V - W| \leq \frac{M_l}{r^l}
\end{equation}
for each $l$ for some constant $M_l$.
Hence,
\begin{align*}
\begin{split}
  &|P_{r; 1}(V(q_r)) - W(\tilde{q}_r)|^2 \\
  &\qquad= |V(q_r) - W(q_r)|^2- 2\int_0^1\big\langle (D_{\pd{}{s}}W)(\lambda_r(s)), P_{r; s}(V(q_r))- W(\lambda_r(s))\big\rangle\, ds
  \end{split}\\
  &\qquad\leq |V(q_r) - W(q_r)|^2 + 2C_0(r+1)\int_0^1 |\nabla W| (|V(q_r)| + |W(\lambda_r(s))|)\, ds\\
  &\qquad\leq \frac{C_1}{r},
\end{align*}
for some $C_1$ independent of $r$. So,  using \eqref{eq:west} again, we see that
\begin{equation}\label{eq:vest}
 |P_{r; 1}(V(q_r)) - V(\tilde{q}_r)| \leq \frac{C_2}{\sqrt{r}},
\end{equation}
for some $C_2$ independent of $r$. But, by the structure of the cylindrical metric
and the fact that $V$ is parallel along radial lines, 
\[
 |P_{r_1; 1}(V(q_{r_1})) - V(\tilde{q}_{r_1})| = |P_{r; 1}(V(q_r)) - V(\tilde{q}_r)|.
\]
Consequently, sending $r\to \infty$ in \eqref{eq:vest}, we obtain that $P_{r_1; 1}(V(q_{r_1})) = V(\tilde{q}_{r_1})$. 

Fixing $\sigma = \tilde{\sigma}$ and $r\geq r_0$, and applying this conclusion for arbitrary $\theta$, $\tilde{\theta}\in \Ss^k$,
we see that each of the vector fields $V(\cdot, \sigma, r)$ are parallel relative to the round metric on $\Ss^k$. Since $k\geq 2$, these vector fields must be trivial and thus $V$ is tangent to the $\RR^{n-k}$ factor. 

Now we argue that $V = V(\sigma, r)$, regarded as a vector field on $\RR^{n-k}\setminus \overline{B_{r_0}(0)}$, is parallel. We already know that $V$ is invariant under parallel translation along any path in which either the $r$-coordinate or $\sigma$-coordinate is fixed, and therefore is also invariant along the concatenation of such paths. Since $\RR^{n-k}\setminus \overline{B_{r_0}(0)}$ has trivial local holonomy it follows that $V$ is parallel and represented by a constant vector on $\RR^{n-k}$.
\end{proof}

\subsection{Aligning the vector fields}
The previous proposition suggests the following refinement of our notion of asymptotic cylindricity which incorporates the vector field as well as the metric.
\begin{definition} We will say  $(\Cc_{r_0}, \gt, \Xt)$ is \emph{strongly asymptotic to  $(\Cc, g, X)$ as a soliton} if 
\begin{equation}\label{eq:decay3}
 \sup_{\Cc_{r_0}} |z|^l\left(|\nabla^{(m)}(\gt - g)| + |\nabla^{(m)}(\Xt - X)|\right) < \infty
\end{equation}
for all $l$, $m\geq 0$.
\end{definition}

We may then restate Proposition \ref{prop:xnorm} as follows.
\begin{proposition}\label{prop:translation}
Suppose $(\Cc_{r_0}, \gt, \nablat \ft)$ is a gradient shrinking soliton for which $(\Cc_{r_0}, \gt)$ is strongly asymptotic to $(\Cc, g)$. Then there is $r_1 \geq r_0$
and a translation $\tau_{z_0}(\theta, z) = (\theta, z- z_0)$
such that $(\Cc_{r_1}, \tau_{z_0}^*\gt, \tau_{z_0}^*(\nablat \ft))$ is strongly asymptotic to $(\Cc, g, \nabla f)$ as a soliton.
\end{proposition}
\begin{proof} Let $\Xt = \nablat \ft$ and $X = \nabla f$. By Proposition \ref{prop:xnorm}, we may write $\Xt = X + V + E$ where $V$ is a constant vector field tangent to the $\RR^{n-k}$ factor and $E$ satisfies
\[
   \sup_{\Cc_{r_0}} |z|^l |\nabla^{(m)}E|(\theta, z) < \infty
\]
for all $l$, $m\geq 0$.

Let us write the components of $V$ as $V^{i} = z_0^i/2$, and define the translation map $\tau_{z_0}: \Cc \to \Cc$ by  $\tau_{z_0}(\theta, z) =(\theta,  z - z_0)$.  Provided $r_1 > r_0 + |z_0|$,
we will have $\tau_{z_0}(\Cc_{r_1}) \subset \tau_{z_0}(\Cc_{r_0})$. Since $\tau_{z_0}$ is an isometry of $g$, the restriction of $\tau^*_{z_0}\gt$ to $\Cc_{r_1}$ will continue to
be strongly asymptotic to $g$, but we will now have in addition that
\[
 \tau_{z_0}^*\Xt(\theta, z) = X(\theta, z - z_0) + V + E(\theta, z- z_0) = X(\theta, z) + \tilde{E}(\theta, z), 
\]
where $\tilde{E}(\theta, z) = E(\theta, z- z_0)$ satisfies 
\[
 \sup_{\Cc_{r_1}} |z|^l |\nabla^{(m)}\tilde{E}|(\theta, z)  < \infty
\]
for all $l$, $m\geq 0$. 
\end{proof}

In fact, after adjusting metric and potential
by a further diffeomorphism, we can arrange for the gradient vector field of $(\Cc_{r_0}, \gt, \ft)$ to coincide with the vector field of the standard cylindrical structure.
\begin{theorem}\label{thm:vfnormalization}
Suppose $(\Cc_{r_0}, \gt, \nablat \ft)$ is strongly asymptotic to the cylinder $(\Cc_{r_0}, g, \nabla f)$ as a soliton. Then there is $r_1\geq r_0$
and an injective local diffeomorphism $\Phi:\Cc_{r_1}\to \Cc_{r_0}$ for which $\Cc_{2r_1} \subset\Phi(\Cc_{r_1})$, $(\Cc_{r_1}, \Phi^*\gt)$ is  strongly asymptotic to  $(\Cc, g)$, and 
\begin{equation}\label{eq:vfnorm}
 \Phi^*(\nablat\ft)  =\nabla f = \frac{r}{2}\pd{}{r}
\end{equation}
on $\Cc_{r_1}$.
\end{theorem}
The construction of the map $\Phi$ is straightforward but somewhat technical and also conceptually independent of the rest of the paper. We postpone it until Appendix \ref{app:normalization}.

 \section{Reduction to a problem of parabolic unique continuation}
 \label{sec:reduction}
 
 Now, we recast Theorem \ref{thm:main} as a problem of uniqueness for the backward Ricci flow, by converting the cylinder and the unknown soliton into their shrinking self-similar counterparts.
The reduction in the previous section will allow us to assume that both solutions are flowing relative to a family of diffeomorphisms generated by the same vector field
and thus that their difference is also self-similar.
\begin{proposition}\label{prop:selfsim} Write $X = \nabla f$ and suppose that $(\Cc_{r_0}, \gt, X)$ is strongly asymptotic to $(\Cc_{r_0}, g, X)$ as a soliton. Let $\Psi: \Cc_{r_0}\times (0, 1] \to \Cc_{r_0}$
 be the map $\Psi(\theta, z, \tau) = (\theta, z/\sqrt{\tau})$ and put $\Psi_{\tau} = \Psi(\cdot, \cdot, \tau)$. Then
\[
 g(\tau) = \tau \Psi_{\tau}^*g =  (2(k-1)\tau \gc)\oplus \bar{g}, \quad \gt(\tau) = \tau \Psi^*_{\tau}\gt, 
\]
solve \eqref{eq:brf} on $\Cc_{r_0}\times (0, 1]$, and
$h(\tau) = (\gt - g)(\tau)= \tau\Psi_{\tau}^*h(1)$
satisfies
\begin{equation}\label{eq:hd}
\sup_{\Cc_{r_0}\times(0, 1]} \frac{|z|^{2l}}{\tau^l}|\nabla_{g(\tau)}^{(m)}h(\tau)|_{g(\tau)} < \infty
\end{equation}
for each $l$, $m\geq 0$.
\end{proposition}

\begin{proof} 
The map $\Psi$ satisfies
\begin{equation}\label{eq:psieq}
    \pd{\Psi}{\tau}(\theta, z, \tau) = -\frac{1}{\tau}(X \circ \Psi)(\theta, z, \tau), \quad \Psi(\theta, z, 1) = (\theta, z),
\end{equation}
and it is a standard calculation (see, e.g., \cite{ChowKnopf}) that $g(\tau) = \tau\Psi_{\tau}^*g$ and $\gt(\tau)= \tau\Psi_{\tau}^*\gt$ solve \eqref{eq:brf} .
Equation \eqref{eq:hd} follows then by scaling: fixing $l$, $m\geq 0$, we have
 \begin{align*} 
   \frac{|z|^{2l}}{\tau^l}|\nabla_{g(\tau)}^{(m)}h(\tau)|_{g(\tau)}(\theta, z, \tau) &= \frac{|z|^{2l}}{\tau^{l+ \frac{m}{2}}}|\nabla^{(m)}_{g(1)}h(1)|_{g(1)}\left(\theta, \frac{z}{\sqrt{\tau}}\right) \leq \frac{M_{l, m}}{r_0^m}
 \end{align*}
on $\Cc_{r_0}$ for some constant $M_{l, m}$
 by our assumption on $h$.
\end{proof}
Going forward, we will write simply
\[
\gt = \gt(\tau), \quad g = g(\tau), \quad h= h(\tau), \quad |\cdot| = |\cdot|_{g(\tau)}, \quad \nabla = \nabla_{g(\tau)}.
\]

\subsection{A reformulation of Theorem \ref{thm:main}}
Our parabolic restatement of Theorem 1.2 asserts that any shrinking self-similar solution to the backward Ricci flow
which is flowing along the cylindrical vector field and agrees to infinite order with the shrinking cylinder near spatial infinity and $\tau = 0$ in the sense of \eqref{eq:hd} must coincide with the shrinking cylinder.
\begin{theorem}\label{thm:mainp} Suppose $\gt(\tau) = \tau \Psi_{\tau}^*g(1)$ is a self-similar solution to \eqref{eq:rf} on $\Cc_{r_0} \times (0, 1]$ for some $r_0> 0$,
where $\Psi:\Cc_{r_0}\times (0, 1]\to \Cc_{r_0}$ is the map $\Psi_{\tau}(\theta, z) = (\theta, z/\sqrt{\tau})$, and $g = g(\tau) =  (2(k-1)\tau \gc) \oplus \bar{g}$. 
If, for all $l$, $m\geq 0$, there exist constants $M_{l, m} > 0$ such that $h = g - \gt$ satisfies
\begin{equation}\label{eq:hdecay}
 \sup_{\Cc_{r_0}\times(0, 1]} \frac{|z|^{2l}}{\tau^l}|\nabla^{(m)}h| \leq M_{l, m},
\end{equation}
then $h\equiv 0$ on $\Cc_{r_1}\times (0, \tau_0]$ for some $r_1 \geq r_0$ and $0 < \tau_0 \leq 1$.
\end{theorem}
In fact, $g(\tau)$ and $\gt(\tau)$ will be isometric on all of $\Cc_{r_0}\times (0, 1]$. We will prove Theorem \ref{thm:mainp} in Section \ref{sec:backwarduniqueness} once we have the necessary ingredients in place. For now, we note that it indeed implies Theorem \ref{thm:main} and show how to derive Corollary \ref{cor:maincor} from Theorem \ref{thm:main}.

\begin{proof}[Proof of Theorem \ref{thm:main} assuming Theorem \ref{thm:mainp}] Let $(\Mt, \gt, \ft)$ be a shrinking Ricci soliton for which $(\Mt, \gt)$ is strongly asymptotic to $(\Cc, g)$ along the end $V\subset (\tilde{M}, \gt)$. 
Then, for some $r_0 > 0$,
there is a diffeomorphism 
$\varphi: \Cc_{r_0}\to V$ such that $(\Cc_{r_0}, \varphi^*g)$ is strongly asymptotic to $(\Cc, \gt)$.  By Proposition \ref{prop:translation}, there is $r_1 > r_0$ and an injective local diffeomorphism
$\psi:\Cc_{r_1}\to \Cc_{r_0}$ such
that $(\Cc_{r_1}, (\varphi\circ \psi)^*\gt, (\varphi\circ\psi)^*\delt\ft)$ is strongly asymptotic to $(\Cc, g, \nabla f)$ as a soliton structure. Finally, by Theorem \ref{thm:vfnormalization}, there is $r_2 > r_1$ and an injective local diffeomorphism
$\Phi: \Cc_{r_2}\to \Cc_{r_1}$ such that $(\Cc_{r_2}, (\varphi \circ \psi \circ \Phi)^*\gt, \nabla f)$ is strongly asymptotic to $(\Cc, g, \nabla f)$.

Write $\hat{g} =  (\varphi\circ \psi \circ \Phi)^*\gt$.
Using Proposition \ref{prop:selfsim}, we can construct a self-similar solution $\gh(\tau)= \tau\Psi_{\tau}\gh(1)$ on $\Cc_{r_1}\times (0, 1]$ from $\gh = \gh(1)$ and $\nabla f$ for which $h = \gh - g$ satisfies
\begin{equation*}
 \sup_{\Cc_{r_2}\times(0, 1]} \frac{|z|^{2l}}{\tau^l}|\nabla^{(m)}h| < \infty
\end{equation*}
for all $l$, $m\geq 0$.

By Theorem \ref{thm:mainp}, $h \equiv 0$ on $\Cc_{r_3}\times (0, \tau_0]$ for some $\tau_0 > 0$ and $r_3 \geq r_2$. Fixing any $a \in (0, \tau_0]$, we then have $\gh(a) = a\Psi_a^*\gh(1) = a \Psi_a^*g(1) = g(a)$ on $\Cc_{r_3}$,
so $\gh = (\varphi\circ\psi \circ \Phi)^*\gt = g$ on $\Cc_{r_4}$ where $r_4 = r_3/\sqrt{a}$. However, as Ricci solitons, 
both $\gh$ and $g$ are real-analytic relative to atlases consisting of their own geodesic normal coordinate charts 
\cite{IveyLocalSoliton}, and the isometry between them on $\Cc_{r_4}$ can be extended to an isometry on $\Cc_{r_2}$ by continuation along paths. So  $\gh$ and $\gt$ are in fact isometric on $\Cc_{r_2}$. Similarly, $\varphi^*\gt$ and $g$ are isometric on $\Cc_{r_0}$, and $(V, \gt)$ is isometric to $(\Cc_{r_0}, g)$.
\end{proof}

\begin{proof}[Proof of Corollary \ref{cor:maincor}]
Suppose now that $(\tilde{M}, \gt)$ is complete. By Theorem \ref{thm:main}, $(V, \gt)$ is isometric to $(\Cc_{r_0}, g)$ for some $r_0 > 0$. Then the lift $(M^{\prime}, g^{\prime})$ of $(\tilde{M}, \tilde{g})$ to the universal cover $M^{\prime}$ of $M$ is complete, real-analytic,
and isometric to $(\Cc, g)$ on an open set. Since $\Cc$ and $M^{\prime}$ are simply connected, it follows that $(M^\prime, g^{\prime})$ is globally isometric to $(\Cc, g)$.  So $(\tilde{M}, \gt)$
must be a quotient of $(\Cc, g)$ by a discrete subgroup $\Gamma$ of isometries acting freely and properly on $\Cc$. 

To identify this quotient, let $\pi:\Cc\to \tilde{M}$ be the covering map, and consider $V^{\prime} = \pi^{-1}(V)$.  
By \cite{Wylie}, the fundamental group of $\tilde{M}$
is finite, so $\pi$ is proper, and we may write $V^{\prime}$ as the disjoint union of finitely many connected components $V_i^{\prime}$, $i=1, 2, \ldots, N$. Each $V_i^{\prime}$ is itself an end of 
$(\Cc, g)$, and, since  $V$ is open and simply connected,
the restriction of $\pi$ to any $V_i^{\prime}$ is a diffeomorphism. 

When $2 \leq k < n-1$, we must have $N =1$ since $(\Cc, g)$ is connected at infinity. Thus $\pi: \Cc\to \tilde{M}$ is a diffeomorphism and $\Gamma = \{\operatorname{Id}\}$ in this case.
When $k= n-1$, $(\Cc, g)$ has two ends, and we must have $N \leq 2$ and $|\Gamma| \leq 2$.  Any isometry $\gamma$ of $(\Cc, g)$ must take the form $\gamma(\theta, r) = (F(\theta), G(r))$,
and, if $\gamma\in \Gamma$, we know that both $F$ and $G$ have order no more than two. For $G$ this means that $G(r) = r$ or $G(r) = -r + c$ for some $c$. If $G(r) = r$, then either $\gamma = \operatorname{Id}$ or $F(\theta) = -\theta$. However, the latter is impossible 
since no end of $\RP^{n-1}\times \RR$ is isometric to $\Ss^{n-1}\times (a, \infty)$ for any $a$. If instead $G(r) = -r + c$ for some $c$, then $\gamma$ fixes $\Ss^{n-1}\times \{c/2\}$. This forces $F$ to have the form $F(\theta) = -\theta$, if $\gamma$ is not to fix any points.
Thus, when $k = n-1$, either $\Gamma = \{\operatorname{Id}\}$ or $\Gamma = \{\operatorname{Id}, \gamma\}$ where $\gamma(\theta, r) = (-\theta, -r + c)$ is a reflection on both factors.
\end{proof}

\section{The basic system}
\label{sec:pdeode}

Next we transform Theorem \ref{thm:mainp} into a problem that we can treat with Carleman inequalities.
Following the method used in \cite{KotschwarBackwardsUniqueness}, we will first define a simple prolonged ``PDE-ODE'' system whose components satisfy a coupled system of mixed parabolic 
and ordinary differential inequalities amenable to the application of inequalities \eqref{eq:pdecarlemanbu} and \eqref{eq:odecarlemanbu} in Section \ref{sec:backwarduniqueness}.

\subsection{The setting}

First we need to establish some notation. Here, as before, $g(\tau)= (2(k-1)\tau\gc)\oplus\bar{g}$ will represent the normalized shrinking cylindrical solution
to \eqref{eq:brf} on $\Cc\times (0, \infty)$. We will use $g= g(\tau)$ and $\nabla = \nabla_{g(\tau)}$ as the reference metric and connection in our computations, and write $\tau$ $|\cdot| = |\cdot|_{g(\tau)}$, suppressing $\tau$.

Since the structural properties of the system we will describe are independent of the self-similarity of $\gt$,
we will assume in this section (except within the context of the last assertion in Proposition \ref{prop:xypdeode}) only that
$\gt = \gt(\tau)$ is a solution to the backward Ricci flow \eqref{eq:brf} on $\Cc_{r_0}\times (0, 1]$ for which $h = \gt - g$ satisfies
\begin{equation}\label{eq:hstdecay}
 \sup_{\Cc_{r_0}} \frac{|z|^{2l}}{\tau^l}|\nabla^{(m)}h|(\theta, z, \tau) <\infty
\end{equation}
for all $l$, $m\geq 0$.

It will be convenient to introduce the operator
\[
 D_{\tau} = \pdtau - R_{q}^p\Lambda^q_p
\]
acting on families of $(k, l)$ tensors $V = V(\tau)$, where
\begin{align*}
  \Lambda_p^q(V)_{b_1b_2\cdots b_k}^{a_1a_2\cdots a_l}&= 
  \delta^q_{b_1}V_{pb_2\ldots b_k}^{a_1a_2\ldots a_l}  +\delta^q_{b_2} V_{b_1p\ldots b_k}^{a_1a_2\ldots a_l}
 + \cdots +\delta^q_{b_k} V_{b_1b_2\ldots p}^{a_1a_2\ldots a_l}\\
 &\phantom{=}- \delta^{a_1}_pV_{b_1b_2\ldots b_k}^{qa_2\ldots a_l}
 - \delta^{a_2}_pV_{b_1b_2\ldots b_k}^{a_1q\ldots a_l} - \cdots - \delta^{a_l}_pV_{b_1b_2\ldots b_k}^{a_1a_2\ldots q}. 
\end{align*}
Here $R_q^p = g^{pr}R_{rq}$. (We have two metrics lurking in the background, so to avoid confusion, we will only implicitly raise and lower indices with the metric $g$, and explicitly include any instances of $\gt$ and $\gt^{-1}$.) 
When $\{e_{i}(\tau)\}_{i=1}^{n}$ is a smooth family of local orthonormal frames
evolving so as to remain
orthonormal relative to $g(\tau$), the components of $D_{\tau}V$ express the total derivatives 
\[
  D_{\tau}V_{b_1b_2 \ldots b_k}^{a_1a_2\ldots a_l} = \pdtau \left(V(e_{b_1}, e_{b_2}, \ldots, e_{b_k}, e_{a_1}^*, e_{a_2}^*, \ldots, e_{a_l}^*)\right).
\]
In particular, $D_{\tau}g \equiv 0$.

\subsection{Definition of the system}
Now consider the bundles
\[
\Xc = T^{(5, 0)}(\Cc), \quad \Yc = T^{(2, 0)}(\Cc) \oplus T^{(3, 0)}(\Cc) \oplus T^{(4, 0)}(\Cc), 
\]
over $\Cc$ equipped with the smooth families of metrics and connections induced by $g$. Let
$\ve{X}$ and $\ve{Y}$ be the family of sections of $\Xc$ and $\Yc$ over $\Cc_{r_0}\times (0, 1]$ defined
by
\begin{equation}\label{eq:xydef}
\ve{X} = \delt \Rmt = \delt \Rmt - \nabla \Rm, \quad \ve{Y} = (Y_0, Y_1, Y_2) = (h, \nabla h, \nabla\nabla h).
\end{equation}
The system $(\ve{X}, \ve{Y})$ is  equivalent to that considered in \cite{KotschwarBackwardsUniqueness}, \cite{KotschwarWangConical}. The components of $\ve{Y}$ are chosen to ensure
that, together, $\ve{X}$ and $\ve{Y}$ satisfy a closed system of differential inequalities. 

\begin{proposition}\label{prop:xypdeode}
Let $\ve{X}$ and $\ve{Y}$ denote the sections of $\Xc$ and $\Yc$ defined above. There is a constant $B > 0$ such that
\begin{align}\label{eq:xypdeode}
 \begin{split}
  |(D_{\tau} + \Delta)\ve{X}|&\leq \frac{B}{\tau}|\ve{X}| + B|\ve{Y}|,\\
  |D_{\tau} \ve{Y}| &\leq B(|\ve{X}| + |\nabla\ve{X}|) + \frac{B}{\tau}|\ve{Y}|,
 \end{split}
\end{align}
on $\Cc_{r_0}\times (0, 1]$, and, for each $l$, $m\geq 0$, constants $M_{l, m}$ such that
\begin{equation}\label{eq:xydecay}
 \sup_{\Cc_{r_0}\times (0, 1]}\frac{|z|^{2l}}{\tau^l}\left(|\nabla^{(m)}\ve{X}| + |\nabla^{(m)}\ve{Y}|\right) \leq M_{l, m}.
\end{equation}
Moreover, when $h(\tau) = \tau\Psi_{\tau}^*(h(1))$ as in Theorem \ref{thm:mainp}, $\ve{X}$ and $\ve{Y}$ are self-similar in the sense that
\begin{equation}\label{eq:xyselfsim}
 \ve{X}(\tau) = \tau\Psi_{\tau}^*(\ve{X}(1)),\quad \ve{Y}(\tau) = \tau\Psi_{\tau}^*(\ve{Y}(1)). 
\end{equation}
\end{proposition}
The decay \eqref{eq:xydecay} and self-similarity \eqref{eq:xyselfsim} of $\ve{X}$ and $\ve{Y}$ follow from the corresponding properties of $h$, and the observation that
the components of $\ve{X}$ and $\ve{Y}$ scale the same as $h$. The verification of \eqref{eq:xypdeode} is close to that of Lemma 3.1 in \cite{KotschwarWangConical}; see Proposition \ref{prop:schemev} below.
We include some of the computations on which it relies
since we will need them
in any case when we modify this system in the next section.

\subsubsection{Evolution equations}

Here and below we will use $V\ast W$ to denote linear combinations of contractions of $V\otimes W$ or $\tilde{V}\otimes \tilde{W}$
for any tensors $\tilde{V}$ and $\tilde{W}$ identified with $V$ and $W$ via the isomorphisms $T\Cc\to T^*\Cc$ and $T^*\Cc \to T\Cc$ induced by $g$. The coefficients
in these linear combinations are understood to be bounded by dimensional constants. 

We will first recall standard formulas for the difference of the Levi-Civita connections and curvature tensors of different metrics. 
\begin{lemma}\label{lem:metdiff} Let $g$, $\gt$ be any two metrics and $h = g- \gt$.
Then
\begin{align}\label{eq:invdiff}
  \gt^{ij} - g^{ij} &= -\gt^{ia}g^{jb}h_{ab} = \gt^{-1}\ast h,\\
\label{eq:delgt}
  \nabla_k \gt^{ij} &= -\gt^{ia}\gt^{jb} \nabla_k h_{ab} = \gt^{-2} \ast \nabla h,\\
  \label{eq:rmdiff}
  \Rmt - \Rm & = \nabla\nabla h + \gt^{-1}\ast (\nabla h)^2 + \Rm \ast h,
\end{align}
where $\Rm$ and $\Rmt$ denote the $(4, 0)$ curvature tensors of $g$ and $\gt$.

In addition, 
\begin{align}
 \label{eq:d1diff}
 \delt V - \nabla V &= \gt^{-1}\ast\nabla h \ast V,\\
\begin{split}
\label{eq:laplacediff}
\Deltat V - \Delta V &= \gt^{-2}\ast \nabla h \ast \nabla V + \gt^{-3} \ast (\nabla h)^2 \ast V \\
		     &\phantom{=}+ \gt^{-2} \ast \nabla\nabla h \ast V +\gt^{-1}\ast h \ast \nabla\nabla V,
\end{split}
 \end{align}
 for any tensor $V$ of rank at least $1$.
\end{lemma}
\begin{proof}
 We only prove \eqref{eq:rmdiff}. Writing, temporarily, 
\begin{align*}
    A_{jk}^l &= \Gamt_{jk}^l -\Gamma_{jk}^l = \frac{1}{2}\gt^{lp}\left(\nabla_j h_{kp} + \nabla_k h_{jp} - \nabla_{p}h_{jk}\right),
\end{align*}
and
\begin{align*}
    B_{jkl} &= \gt_{ml} A_{jk}^m = \frac{1}{2}\left(\nabla_j h_{kl} + \nabla_k h_{jl} - \nabla_{l}h_{jk}\right),
\end{align*}
we have
\begin{align*}
& \Rt_{ijk}^m - R_{ijk}^m = \nabla_i A_{jk}^m - \nabla_j A_{ik}^m + A_{jk}^p A_{ip}^m - A_{ik}^pA_{jp}^m\\
 \begin{split}
                         &\qquad= \gt^{pm}(\nabla_i B_{jkp} -\nabla_j B_{ikp}) -\gt^{pq}\gt^{mr}(\nabla_ih_{rs}B_{jkp}
                         -\nabla_j h_{rs}B_{ikp})\\
                         &\qquad\phantom{=} + \gt^{qm}(A_{jk}^pB_{ipq} - A_{ik}^pB_{jpq}).
\end{split}
\end{align*}
So the difference of the $(4, 0)$-tensors satisfies
\begin{align*}
 \Rt_{ijkl} - R_{ijkl} &= \gt_{lm}(\Rt_{ijk}^m - R_{ijk}^m) +R_{ijk}^m h_{lm}\\
 \begin{split}
    &= \nabla_i B_{jkl} -\nabla_j B_{ikl} -\gt^{pq}(\nabla_ih_{ls}B_{jkp}
                         -\nabla_j h_{ls}B_{ikp})\\
                         &\phantom{=} + A_{jk}^pB_{ipl} - A_{ik}^pB_{jpl} +R_{ijk}^m h_{lm}.
\end{split}\\
 &= \nabla\nabla h + \gt^{-1}\ast(\nabla h)^2 + \Rm \ast h
\end{align*}
as claimed.
\end{proof}

 Now (referring to, e.g., Section 6.1 of \cite{ChowKnopf}), we  recall the evolution equations 
 \begin{equation*}
  \label{eq:connev}
  \pdtau \Gamt_{ij}^k = \gt^{mk}\left(\delt_{i}\Rt_{jm} + \delt_j\Rt_{im} - \delt_m\Rt_{ij}\right),
 \end{equation*}
 and
 \begin{align*}
 \begin{split}
  \left(\pdtau + \Deltat\right)\Rt_{ijkl} &= -2(\Bt_{ijkl} - \Bt_{ijlk} + \Bt_{ikjl} - \Bt_{iljk}) \\
 &\phantom{=} +\gt^{pq}\left(\Rt_{ip}\Rt_{qjkl} +\Rt_{jp}R_{iqkl} + \Rt_{kp}\Rt_{ijql} +  \Rt_{lp}\Rt_{ijkq}\right),
 \end{split}
 \end{align*}
where
\[
  \Bt_{ijkl} = -\gt^{pr}\gt^{qs}\Rt_{pijq}\Rt_{rkls}.
\]
Combining these equations with a bit of further computation, we arrive at the following equation for the evolution of $\delt\Rmt$.
\begin{lemma}\label{lem:drmtev}
If $\gt$ satisfies \eqref{eq:brf}, then
\begin{align*}
 \begin{split}
 &\left(\pdtau + \Deltat\right)\delt_a\Rt_{ijkl}= -2\delt_a\left(\Bt_{ijkl} - \Bt_{ijlk} + \Bt_{ikjl} - \Bt_{iljk}\right)\\
 &\quad+ 2\gt^{pr}\gt^{qs}\bigg(\Rt_{iqap}\delt_r \Rt_{sjkl} + \Rt_{jqap}\delt_r \Rt_{iskl} + \Rt_{kqap}\delt_r\Rt_{ijsl} + \Rt_{lqap}\delt_r\Rt_{ijks}\bigg)\\
 &\quad +\gt^{pq}\bigg(\Rt_{ap}\delt_q\Rt_{ijkl} + \Rt_{ip}\delt_a\Rt_{qjkl}+\Rt_{jp}\delt_a\Rt_{iqkl}+ \Rt_{kp}\delt_a\Rt_{ijql} +  \Rt_{lp}\delt_a\Rt_{ijkq}\bigg).
 \end{split}
\end{align*}

\end{lemma}

Note that, according to our normalization, the curvature tensor of the cylindrical metric $g$
satisfies
\[
 |\Rm|^2 = \frac{k}{2(k-1)\tau^2}.
\]
The first assertion in Proposition \ref{prop:xypdeode} is now a consequence of the decay assumption \eqref{eq:hstdecay}, Lemma \ref{lem:metdiff}, and the following schematic evolution equations. 
\begin{proposition}\label{prop:schemev}
 The tensors $h$ and $\delt\Rmt$ satisfy
\begin{align}
 \label{eq:dth}
    D_{\tau} h &= \gt^{-1}\ast \nabla\nabla h +  \gt^{-2}\ast (\nabla h)^2 + \gt^{-1}\ast\Rm \ast h + \Rm \ast h,
\end{align}
\begin{align}
\begin{split}\label{eq:dtdh}
    D_{\tau}\nabla h &= \gt^{-1}\ast \nablat\Rmt  + \gt^{-2}\ast \nabla h \ast \nabla\nabla h + \gt^{-3}\ast (\nabla h)^3\\
    &\phantom{=} + \gt^{-2} \ast\Rm \ast h \ast \nabla h  + \Rm \ast \nabla h,
\end{split}
\end{align}
\begin{align}
\begin{split}\label{eq:dtddh}
D_{\tau} \nabla\nabla h &= \gt^{-2}\ast\nabla h \ast \delt\Rmt + \gt^{-1}\ast \nabla \delt\Rmt\\
&\phantom{=} + \gt^{-3}\ast\nabla\nabla h \ast (\nabla h)^2  + \gt^{-4}\ast(\nabla h)^4  + \gt^{-2}\ast (\nabla\nabla h)^2\\
&\phantom{=}  + \gt^{-3} \ast \Rm \ast h \ast (\nabla h)^2 + \gt^{-2}\ast \Rm \ast h \ast \nabla\nabla h \\
&\phantom{=}+\gt^{-3}\ast \Rm \ast (\nabla h)^2 + \gt^{-2}\ast\Rm\ast \nabla\nabla h + \Rm \ast \nabla\nabla h,
\end{split}
\end{align}
and
\begin{align}
\begin{split}
 \label{eq:drmtev}
 &(D_{\tau} + \Delta)\delt\Rmt = \gt^{-1}\ast\delt^{(3)}\Rmt \ast h  + \gt^{-1}\ast \delt\delt \Rmt \ast \nabla h \\
 &\phantom{=}+ \gt^{-2}\ast \nablat\Rmt \ast (\nabla h)^2 + \gt^{-1}\ast\delt\Rmt \ast \nabla\nabla h \\
 &\phantom{=} + \gt^{-2}\ast (\Rmt - \Rm) \ast \delt\Rmt+ \gt^{-2}\ast h \ast \Rm \ast \delt\Rmt \\
 &\phantom{=} + \gt^{-2} \ast \Rm \ast \delt\Rmt.
\end{split}
\end{align}
\end{proposition}
\begin{proof}
For \eqref{eq:dth}, we have
\begin{align*}
 D_{\tau} h_{ij} &= 2\Rt_{ij} - R^p_i\gt_{pj} - R^p_j \gt_{ip}=2(\Rt_{ij}-R_{ij}) - R^p_ih_{pj} - R^p_j h_{ip}\\
		 &=2\gt^{pq}(\Rt_{ipqj} -R_{ipqj}) -2\gt^{pr}g^{qs}R_{ipqj}h_{rs}  - R^p_ih_{pj} - R^p_j h_{ip},
\end{align*}
which yields the desired expression after applying \eqref{eq:rmdiff} to the first term on the right.
Equations \eqref{eq:dtdh} and \eqref{eq:dtddh} follow similarly, using that the Levi-Civita connection $\nabla$ of the cylindrical metric is time-independent.

For \eqref{eq:drmtev}, observe that, by Lemma \ref{lem:drmtev},
\begin{align*}
\begin{split}
 &\left(D_{\tau} + \Deltat\right)\delt_a\Rt_{ijkl} = -2\delt_a\left(\Bt_{ijkl} - \Bt_{ijlk} + \Bt_{ikjl} - \Bt_{iljk}\right)\\
 &+ 2\gt^{pr}\gt^{qs}\big(\Rt_{iqap}\delt_r \Rt_{sjkl} + \Rt_{jqap}\delt_r \Rt_{iskl} + \Rt_{kqap}\delt_r\Rt_{ijsl}+ \Rt_{lqap}\delt_r\Rt_{ijks}\big)\\
& +(\gt^{pq}\Rt_{ap} - R_a^q)\delt_q\Rt_{ijkl} + (\gt^{pq}\Rt_{ip} - R_i^q)\delt_a\Rt_{qjkl}+(\gt^{pq}\Rt_{jp} - R_j^q)\delt_a\Rt_{iqkl}\\
&+(\gt^{pq}\Rt_{kp} - R_k^q)\delt_a\Rt_{ijql}+  (\gt^{pq}\Rt_{lp} - R_l^q)\delt_a\Rt_{ijkq}.
 \end{split}
 \end{align*}
The desired expression then follows from \eqref{eq:laplacediff} and the observation that the terms on the left on the first two lines are all of the schematic form
$\gt^{-2}\ast\Rmt\ast \delt\Rmt$.
\end{proof}

\section{Exponential decay and a refined system}
\label{sec:pdeode2}
In order to apply the Carleman inequalities in Section \ref{sec:backwarduniqueness} which imply the vanishing of $\ve{X}$ and $\ve{Y}$, 
we need to know that these sections decay near spatial infinity and $\tau = 0$ at at least an exponential rate. 
The goal of the next two sections will be to prove
the following local estimate, which establishes this uniform exponential decay on regions of fixed size.

Below, we will write
\[
 \Dc_{r}(z_0) = \Ss^k\times B_{r}(z_0)
\]
for $r > 0$ and $z_0\in \RR^{n-k}$, and use the shorthand $\dX = d\mu_{g(\tau)} \,d\tau$.

\begin{theorem}\label{thm:expdecay1} 
 There exist positive constants  $N_0$, $N_1$ depending only on $n$, $k$, $r_0$ and finitely many of the  constants $M_{l, m}$ from \eqref{eq:hstdecay}
 such that
 \begin{equation}\label{eq:expdecay1}
    \int^1_0\int_{\Dc_{1}(z_0)}\left(|\ve{X}|^2 + |\nabla\ve{X}|^2 + |\ve{Y}|^2\right)e^{\frac{N_0}{\tau}}\,\dX \leq N_1, 
 \end{equation}
for any $z_0 \in \RR^{n-k}\setminus \overline{B_{8r_0}(0)}$.
\end{theorem}
In Proposition \ref{prop:expdecay2} we will use the self-similarity of $\ve{X}$ and $\ve{Y}$ to deduce an estimate on the \emph{space-time} vanishing rate of the sections from
\eqref{eq:expdecay1}. 
However, the self-similarity of $\ve{X}$ and $\ve{Y}$ will not be used in the proof of Theorem \ref{thm:expdecay1} or elsewhere in the the next two sections.

As discussed in the introduction, it does not seem possible here to deduce the exponential decay of  $\ve{X}$ and $\ve{Y}$ from a direct application of our Carleman estimates (Theorems \ref{thm:pdecarleman2} and \ref{thm:odecarleman2} below) to the system \eqref{eq:xypdeode} since these estimates cannot absorb the coefficients of $\ve{X}$ on the right side of \eqref{eq:xypdeode} which blow up at a rate proportional to $1/\tau$.  However, this obstacle is really only an artifact of the coarse way 
in which we have so far estimated the reaction terms in the evolution equation for $\delt\Rmt$ in \eqref{eq:xypdeode}. 

We will now analyze the algebraic structure of these terms more carefully and introduce a replacement for $\ve{X}$ in which the components of $\delt\Rmt$ relative to the splitting of $T\Cc$ are grouped and rescaled according to their own individual rates of decay. We will define
this replacement in Proposition \ref{prop:wtriangular} once the preliminary calculations are out of the way.
 Though involved, the computations in this section are guided by a fairly simple underlying strategy. See Section \ref{ssec:strategy} for a short explanation. We summarize the properties of the resulting ``refined'' system we will need in Proposition \ref{prop:wsys} at the end of the section.

\subsection{Notational conventions} We will not make use the self-similarity of $\gt$ in the computations below, so for the rest of this section, 
$\gt = \gt(\tau)$ will simply represent a smooth solution to the backward Ricci flow \eqref{eq:brf} satisfying \eqref{eq:hstdecay}  on $\Cc_{r_0}\times (0, 1]$. 
We will continue to use $g = g(\tau)$ to represent the normalized shrinking cylindrical solution on $\Cc\times (0, 1]$.

Let $\Hc$ and $\Kc$ denote the subbundles of $T\Cc$ with fibers $\Hc_{(\theta, z)} = T_{(\theta, z)}(\Ss^k\times \{z\})$ and $\Kc_{(\theta, z)} = T_{(\theta, z)}(\{\theta\}\times \RR^{n-k})$,
and let $\Ps: T\Cc\to \Hc$ and $\Pb:T\Cc\to \Kc$ denote the corresponding $g$-orthogonal projections onto these subbundles. The projections $\Ps$ and $\Pb$ are smooth, globally defined families of $(1, 1)$-tensor fields
on $\Cc\times(0, 1]$ satisfying
\[
 \Ps^2 = \Ps, \quad \Pb^2 =\Pb, \quad \Ps + \Pb = \operatorname{Id}_{T\Cc}, \quad g(\Ps \cdot, \Pb \cdot) = 0,
\]
and
\[
  \nabla \Ps = \nabla \Pb = 0,\quad  \pdtau \Ps = \pdtau \Pb =  D_{\tau}\Ps = D_{\tau}\Pb = 0.
\]

Using $\Ps$ and $\Pb$, we can track the components of any tensor relative to the splitting $T\Cc = \Hc \oplus \Kc$.
We will use a notational system of underlined and barred indices to distinguish these components.  Underlined indices will denote components acting
on directions tangent to the spherical factor and barred indices will denote components acting on directions tangent to the Euclidean factor. Thus, for example, we will write
\[
 \Rt_{\ub{a}\ub{b}} = \Rt_{ij}\Ps_a^i\Ps_b^j, \ \Rt_{\ub{a}\bar{b}} = \Rt_{ij}\Ps_a^i\Pb_b^j, \ \Rt_{\bar{a}\ub{b}} = \Rt_{ij}\Pb_a^i\Ps_b^j, \ \Rt_{\bar{a}\bar{b}} = \Rt_{ij}\Pb_a^i\Pb_b^j.
\]
An unadorned index will represent an unmodified component, e.g.,
\[
 \Rt_{\ub{a}b} =\Rt_{ib}\Ps_a^i.
\]

We emphasize that each of the above expressions represent \emph{globally-defined} tensor fields and that the underlined and barred indices denote modifications to the tensor field, 
not the expression of the components of the tensor 
relative to a particular local frame. 

In general, we will not need to pay careful attention to the algebraic structure of terms that are quadratic or better in factors of $h$ and its derivatives, and it will be convenient to have an economical notation for tensors with rapid space-time decay whose precise form we can safely ignore.

\begin{notation} The expression $\oinf$ will denote various families of tensor fields $V = V(\tau)$ that vanish to infinite order in space and time in the sense that
 \[
   \sup_{\Cc_{r_0}\times(0, 1]}\left(\frac{|z|^{2l}}{\tau^l}\right)|V|(\theta, z, \tau) < \infty 
 \]
 for all $l\geq 0$. Here $|\cdot| = |\cdot|_{g(\tau)}$ as before.
\end{notation}
Finally, we will also use a repeated index to denote a contraction with the metric $g$, and write out explicitly any contraction with $\gt$.

\subsection{The gradient of the scalar curvature.}
We begin our analysis by examining the evolution of the differential of the scalar curvature. In this and the calculations that follow,
we will focus our attention on the structure of the linearization of the reaction terms based at the cylindrical solution $g$.

\begin{proposition}\label{prop:dscalev}
The differential $\delt \Rt$ of the scalar curvature of $\gt$ satisfies
\begin{equation}\label{eq:dscalev}
 |(D_{\tau} + \Delta)\tau^2\nablat_a \Rt| \leq |\oinf| (|h| + |\nabla h| +  |\nablat \Rct|) + 2\sqrt{n-k}\tau |\nablat_a\Rt_{\bar{b}\bar{c}}|
\end{equation}
on $\Cc_{r_0}\times (0, 1]$.
\end{proposition}
\begin{proof} From the standard formula
\[
 \left(\pdtau + \Deltat\right)\Rt = -2|\Rct|^2_{\gt},
\]
we have
\begin{align*}
\left(\pdtau + \Deltat\right)\nablat_a \Rt & = -4\gt^{pq}\gt^{rs}\nablat_a \Rt_{pr}\Rt_{qs} + \gt^{pq}\Rt_{ap}\nablat_q \Rt,
\end{align*}
and then
\begin{align*}
\begin{split}
 \left(D_{\tau} + \Deltat\right)\nablat_a \Rt & =  -4\gt^{pq}\gt^{rs}\delt_a \Rt_{pr}\Rt_{qs} + (\gt^{pq}\Rt_{ap} - g^{pq}R_{ap})\nablat_q \Rt.
\end{split}
 \end{align*} 
 Using \eqref{eq:hstdecay} and \eqref{eq:laplacediff}, and the fact that $R_{ij} = (1/2\tau)\Ps_{ij}$, where $\Ps_{ij} = g_{jk}\Ps_{i}^k = g_{\ub{i}\ub{j}}$, we may rewrite this as
 \begin{align*}
 \begin{split}
   \left(D_{\tau} + \Delta\right)\nablat_a \Rt & = \Delta \nablat_a \Rt - \Deltat\nablat_a \Rt -4\gt^{pq}\gt^{rs}\nablat_a \Rt_{pr}\Rt_{qs} + (\gt^{pq}\Rt_{ap} - g^{pq}R_{ap})\nablat_q \Rt
 \end{split}\\
    &= \oinf \ast (h + \nabla h + \delt \Rct)  -4\delt_a \Rt_{pq} R_{pq}\\
    &= \oinf \ast (h +  \nabla h + \delt \Rct)  -\frac{2}{\tau}\nablat_a \Rt_{pq} \Ps_{pq},
 \end{align*}
 and, using our indexing convention, again as
 \begin{align}\nonumber
   \left(D_{\tau} + \Delta\right)\nablat_a \Rt  
   &=\oinf \ast (h + \nabla h + \nablat \Rct)  -\frac{2}{\tau}\nablat_a \Rt_{\ub{p}\ub{p}}\\
  \nonumber
    &= \oinf \ast (h + \nabla h + \nablat \Rct)  -\frac{2}{\tau}\nablat_a \Rt_{pp} + \frac{2}{\tau} \nablat_a \Rt_{\bar{p}\bar{p}}\\
  \label{eq:scalevpre1}
    &= \oinf \ast( h + \nabla h + \nablat \Rct) -\frac{2}{\tau}\nablat_a \Rt + \frac{2}{\tau} \nablat_a \Rt_{\bar{p}\bar{p}}.
 \end{align}
Here, to obtain the second line in the above computation, we used that
\[
 \nablat_a \Rt_{pp} = g^{pq}\nablat_a \Rt_{pq} = (g^{pq} - \gt^{pq})\nablat_a \Rt_{pq}  + \nablat_a \Rt = \oinf \ast \nablat \Rct + \nablat_a \Rt.
\]
We then multiply $\nablat \Rt$ by $\tau^2$ so that an application of $D_{\tau}$ will pick off the second term on the right in \eqref{eq:scalevpre1}. This yields equation \eqref{eq:dscalev}.
\end{proof}

\subsubsection{A remark on the strategy}\label{ssec:strategy}
Above, in the computation leading to \eqref{eq:scalevpre1}, we have traded the singular term proportional to $\nablat_a \Rt_{\ub{p}\ub{p}}$
for a singular term proportional to $\nablat_a \Rt_{\bar{p}\bar{p}}$, exchanging a tensor with two underlined indices for one with two barred indices. 
Although we have not eliminated the singular coefficient, we have reassigned it 
from a primarily spherical component of $\delt\Rct$ to a primarily Euclidean one.

The computations for $\delt \Rct$ and $\delt \Rmt$ that follow are essentially just more elaborate versions of this ``under-for-over'' exchange, aimed at rearranging appropriately rescaled components of $\nablat\Rt$, $\nabla \Rct$, and $\nablat\Rmt$ into a system whose singular part has a strictly triangular structure.
This structure will allow us to transfer the blow-up in the equations for the spherical and mixed components of the system to the equations of components with fewer spherical directions. 
At the end of the line are the principally Euclidean components of $\nablat\Rmt$ which satisfy evolution equations with reaction terms that are quadratic-or-better in the other elements of the system and have the capacity to absorb the blow-up that we have sent in their direction. 

\subsection{Decomposition of $\delt\Rct$}
We next examine the evolution of the covariant derivative of the Ricci tensor.   Define
\[
  \Gt_{ijk} = \nablat_i \Rt_{jk} - \nablat_{j}\Rt_{ik},
\]
and, for convenience,
\[
    \cnst = \frac{1}{k-1}.
\]

\begin{proposition}\label{prop:drcev}
 The components of $\delt \Rct$ satisfy the equations
 \begin{align}
  \begin{split}
  \label{eq:drc11x}
   |(D_{\tau} + \Delta)\tau\nablat_{\bar{a}}\Rt_{\bar{\jmath}k}| &\lesssim |\nablat_{\bar{a}}\Rt_{\bar{\imath} \bar{\jmath} \bar{l} k}|,
  \end{split}
  \end{align}
  \begin{align}
  \begin{split}
  \label{eq:drc110}
   |(D_{\tau} + \Delta)\tau^{1-\cnst}\nablat_{\bar{a}}\Rt_{\ub{j}\ub{k}}| &\lesssim \tau^{-\cnst}(|\nablat_{\bar{a}}\Rt| 
   + |\nablat_{\bar{a}}\Rt_{\bar{\jmath}\bar{k}}| + |\nablat_{\bar{a}}\Rt_{\ub{i}\bar{\jmath}\bar{k} \ub{l}}|),
  \end{split}
  \end{align}
  \begin{align}
   \begin{split}
  \label{eq:drc011}
   |(D_{\tau} + \Delta)\tau\nablat_{\ub{a}}\Rt_{\bar{\jmath}\bar{k}}| 
   &\lesssim
   |\nablat_{\ub{a}}\Rt_{\bar{\imath}\bar{\jmath}\bar{k}\bar{l}}|,
  \end{split}
  \end{align}
  \begin{align}
  \begin{split}
  \label{eq:drc001}
   &|(D_{\tau} + \Delta)\tau^{1-\cnst}\nablat_{\ub{a}}\Rt_{\ub{j}\bar{k}}| \lesssim \tau^{-\cnst}(|\nablat_{\bar{a}}\Rt| + |\Gt_{\ub{a}\ub{j}\bar{k}}| +
   |\nablat_{\bar{a}}\Rt_{\bar{\jmath}\bar{k}}| + |\nablat_{\bar{a}}\Rt_{\ub{i} \bar{\jmath}\bar{k} \ub{l}}|),
  \end{split}
  \end{align}
 and
 \begin{align}
  \begin{split}
  \label{eq:drc000}
 & |(D_{\tau} + \Delta)\tau^{1-3\cnst}\nablat_{\ub{a}}\Rt_{\ub{j}\ub{k}}|\\
   &\qquad \lesssim \tau^{-3\cnst}\big(|\nablat_{\ub{a}}\Rt| + |\Gt_{\ub{a}\ub{j}\ub{k}}| +
   |\nablat_{\ub{a}}\Rt_{\bar{\jmath}\bar{k}}| + |\nablat_{\bar{a}}\Rt_{\ub{j}\bar{k}}|
   + |\nablat_{\ub{a}}\Rt_{\ub{i} \bar{\jmath}\bar{k} \ub{l}}|\big),
  \end{split}
 \end{align}
 where the notation $|U| \lesssim |V|$ indicates that
 \[
      |U| \leq |\oinf|(|h| + |\nabla h| + |\nablat \Rmt|) + C|V|
 \]
for some constant $C = C(n) > 0$. The components of  $\Gt_{\ub{a}\ub{j}k}$ satisfy 
\begin{align}
 \begin{split}\label{eq:g001}
 |(D_{\tau} + \Delta)\tau^{1+\cnst}\Gt_{\ub{a}\ub{j}\bar{k}}| &\lesssim \tau^{\cnst}|\nablat_{\bar{a}}\Rt_{\ub{i} \ub{j}\bar{k} \bar{l}}|,
 \end{split}\\
 \begin{split}\label{eq:g000}
 |(D_{\tau} + \Delta)\tau \Gt_{\ub{a}\ub{j}\ub{k}}| &\lesssim 
 |\nablat_{\ub{a}}\Rt| +
   |\nablat_{\ub{a}}\Rt_{\bar{\jmath}\bar{k}}| + |\nablat_{\bar{a}}\Rt_{\ub{j}\bar{k}}| + |\nablat_{\bar{a}}\Rt_{\ub{i} \ub{j}\ub{k} \bar{l}}|.
   \end{split}
\end{align}
\end{proposition}
\begin{proof}
 Starting from the equation
\[
\left(\pdtau + \Deltat\right)\Rt_{jk} = -2\gt^{pr}\gt^{qs}\Rt_{jpqk}\Rt_{rs} + 2\gt^{pq}\Rt_{jp}\Rt_{kq},
\]
we obtain
\begin{align*}
\begin{split}
&\left(\pdtau + \Deltat \right)\nablat_a \Rt_{jk} = 
 \gt^{pq}\left(\Rt_{ap}\nablat_q \Rt_{jk} + \Rt_{jp}\nablat_a \Rt_{qk} + \Rt_{kp} \nablat_a \Rt_{jq}\right)\\
&
-2\gt^{pr}\gt^{qs}\left(\nablat_a\Rt_{jpqk} \Rt_{rs} + \Rt_{jpqk}\nablat_a \Rt_{rs}
+\Rt_{pajq}\nablat_r \Rt_{sk} + \Rt_{pakq}\nablat_r \Rt_{sj}\right),
\end{split}
\end{align*}
and hence 
\begin{align*}
\begin{split}
&\left(D_{\tau} + \Delta \right)\nablat_a \Rt_{jk} = \Delta \nablat_a \Rt_{jk} - \Deltat \nablat_a \Rt_{jk}\\
& + (\gt^{pq}\Rt_{ap} - R_a^q)\nablat_q \Rt_{jk} + (\gt^{pq}\Rt_{jp} - R_j^q)\nablat_a \Rt_{qk} + (\gt^{pq}\Rt_{kp} - R^q_k) \nablat_a \Rt_{jq}\\
&
-2\gt^{pr}\gt^{qs}\left(\nablat_a\Rt_{jpqk} \Rt_{rs} + \Rt_{jpqk}\nablat_a \Rt_{rs}
+\Rt_{pajq}\nablat_r \Rt_{sk} + \Rt_{pakq}\nablat_r \Rt_{sj}\right).
\end{split}
\end{align*}
So, in view of \eqref{eq:hstdecay} and \eqref{eq:laplacediff}, we have
\begin{align*}
\begin{split}
&\left(D_{\tau} + \Delta \right)\nablat_a \Rt_{jk} = \oinf \ast(h + \nabla h + \nablat \Rmt) + E_{ajk},
\end{split}
\end{align*}
where
\begin{align}
\begin{split}
 \label{eq:edef}
E_{ajk} &= -2 \big(\nablat_a\Rt_{pq}R_{jpqk} +\nablat_p \Rt_{qj}R_{pakq} + \nablat_p \Rt_{qk}R_{pajq} + \nablat_a \Rt_{jpqk} R_{pq}\big).
\end{split}
\end{align}

Now, according to our normalization, on the evolving cylinder we have
\[
 R_{ijkl} = \frac{\cnst}{2\tau}(\Ps_{il}\Ps_{jk} - \Ps_{ik}\Ps_{jl}), \quad R_{ij} = \frac{1}{2\tau}\Ps_{ij},
\]
so \eqref{eq:edef} becomes
\begin{align}
\begin{split}\nonumber
E_{ajk} &= -\frac{1}{\tau} \nablat_a \Rt_{jpqk}\Ps_{pq} + \frac{\cnst}{\tau}\nablat_a\Rt_{pq}(\Ps_{jq}\Ps_{pk} - \Ps_{jk}\Ps_{pq})\\
&\phantom{=}+\frac{\cnst}{\tau}\nablat_p \Rt_{qj}(\Ps_{pk}\Ps_{aq} - \Ps_{pq}\Ps_{ak})  + \frac{\cnst}{\tau}\nablat_p \Rt_{qk}(\Ps_{pj}\Ps_{aq} - \Ps_{pq}\Ps_{aj})
\end{split}\\
\begin{split}\label{eq:e1}
&= -\frac{1}{\tau} \nablat_a \Rt_{j\underline{pp}k} + \frac{\cnst}{\tau}(\nablat_a\Rt_{\ub{j}\ub{k}} - \Ps_{jk}\nablat_a\Rt_{\ub{p}\ub{p}}) + \frac{\cnst}{\tau}(\nablat_{\ub{k}} \Rt_{\ub{a}j} - \nablat_{\ub{p}}\Rt_{\ub{p}j}\Ps_{ak})\\
&\phantom{=}
 + \frac{\cnst}{\tau}(\nablat_{\ub{j}} \Rt_{\ub{a}k} -  \nablat_{\ub{p}}\Rt_{\ub{p}k}\Ps_{aj}).
\end{split}
\end{align}
Computing as in the proof of Proposition \ref{prop:dscalev}, we see that
\[
 \nablat_a \Rt_{j\ub{p}\ub{p}k} = \oinf \ast \nablat \Rmt + \nablat_a \Rt_{jk} - \nablat_a \Rt_{j\bar{p}\bar{p}k},
\]
and
\begin{align*}
 \nablat_a\Rt_{\ub{p}\ub{p}} &= \oinf \ast \nablat \Rct + \nablat_a\Rt - \nablat_a\Rt_{\bar{p}\bar{p}},\\
 \nablat_{\ub{p}}\Rt_{\ub{p}j} &=\oinf \ast\nablat \Rct + \frac{1}{2}\nablat_j \Rt -\nablat_{\bar{p}}\Rt_{\bar{p}j}.
\end{align*}

Returning, then, to \eqref{eq:e1} and putting things together, we obtain
\begin{align}
\begin{split}\label{eq:e2}
E_{ajk} &= \oinf \ast \nablat \Rmt -\frac{1}{\tau}\nablat_a \Rt_{jk} +\frac{1}{\tau} \nablat_a \Rt_{j\bar{p}\bar{p}k}\\
 &\phantom{=} +\frac{\cnst}{\tau}\left(\nablat_a\Rt_{\ub{j}\ub{k}} + \nablat_{\ub{k}} \Rt_{\ub{a}j} + \nablat_{\ub{j}} \Rt_{\ub{a}k}\right)+\frac{\cnst}{\tau}\Ps_{jk}\left(\nablat_a\Rt_{\bar{p}\bar{p}}- \nablat_a \Rt\right)\\
 &\phantom{=}  + \frac{\cnst}{\tau}\Ps_{ak}\left(\nablat_{\bar{p}}\Rt_{\bar{p}j}
 - \frac{1}{2}\nablat_jR\right) + \frac{\cnst}{\tau}\Ps_{aj}\left(\nablat_{\bar{p}}\Rt_{\bar{p}k}- \frac{1}{2}\nablat_k \Rt\right)\\
 &= \oinf \ast \nablat \Rmt + F_{ajk},
\end{split}
\end{align}
where, by inspection, the components of the tensor $F_{ajk}$ satisfy
\begin{align}
\label{eq:f11x}
 F_{\bar{a}\bar{\jmath}k} &= - \frac{1}{\tau}\nablat_{\bar{a}} \Rt_{\bar{\jmath}k} +\frac{1}{\tau} \nablat_{\bar{a}} \Rt_{\bar{\jmath}\bar{p}\bar{p}k},
 \end{align}
 \begin{align}
\begin{split}
\label{eq:f100}
 F_{\bar{a}\ub{j}\ub{k}} &= - \left(\frac{1-\cnst}{\tau}\right)\nablat_{\bar{a}}\Rt_{\ub{j}\ub{k}} + \frac{1}{\tau}\nablat_{\bar{a}} \Rt_{\ub{j}\bar{p}\bar{p}\ub{k}}\\
 &\phantom{=}+\frac{\cnst}{\tau}\Ps_{jk}\left(\nablat_{\bar{a}}\Rt_{\bar{p}\bar{p}}- \nablat_{\bar{a}} \Rt\right),
\end{split}
\end{align}
\begin{align}
\begin{split}
\label{eq:f011}
 F_{\ub{a}\bar{\jmath}\bar{k}} &= - \frac{1}{\tau}\nablat_{\ub{a}} \Rt_{\bar{\jmath}\bar{k}} +\frac{1}{\tau} \nablat_{\ub{a}} \Rt_{\bar{\jmath}\bar{p}\bar{p}\bar{k}},
\end{split}
\end{align}
\begin{align}
\begin{split}
\label{eq:f001}
 F_{\ub{a}\ub{j}\bar{k}} &= - \left(\frac{1-\cnst}{\tau}\right)\nablat_{\ub{a}} \Rt_{\ub{j}\bar{k}} -\frac{\cnst}{\tau}\Gt_{\ub{a}\ub{j}\bar{k}} 
 +\frac{1}{\tau} \nablat_{\ub{a}} \Rt_{\ub{j}\bar{p}\bar{p}\bar{k}}\\
		    &\phantom{=} +\frac{\cnst}{\tau}\Ps_{aj}\left(\nablat_{\bar{p}}\Rt_{\bar{p}\bar{k}}- \frac{1}{2}\nablat_{\bar{k}} \Rt\right),
\end{split}
\end{align}
and
\begin{align}
\begin{split}
\label{eq:f000}
F_{\ub{a}\ub{j}\ub{k}} &= -\left(\frac{1-3\cnst}{\tau}\right)\nablat_{\ub{a}} \Rt_{\ub{j}\ub{k}} + \frac{\cnst}{\tau}(\Gt_{\ub{j}\ub{a}\ub{k}} 
+ \Gt_{\ub{k}\ub{a}\ub{j}}) +\frac{1}{\tau} \nablat_{\ub{a}} \Rt_{\ub{j}\bar{p}\bar{p}\ub{k}}\\
 &\phantom{=} +\frac{\cnst}{\tau}\Ps_{jk}\left(\nablat_{\ub{a}}\Rt_{\bar{p}\bar{p}}- \nablat_{\ub{a}} \Rt\right) + \frac{\cnst}{\tau}\Ps_{ak}\left(\nablat_{\bar{p}}\Rt_{\bar{p}\ub{j}}
 - \frac{1}{2}\nablat_{\ub{j}}\Rt\right)\\
 &\phantom{=} + \frac{\cnst}{\tau}\Ps_{aj}\left(\nablat_{\bar{p}}\Rt_{\bar{p}\ub{k}}- \frac{1}{2}\nablat_{\ub{k}} \Rt\right).
\end{split}
 \end{align} 

The relations \eqref{eq:drc11x} - \eqref{eq:drc000} then follow directly from the identities \eqref{eq:f11x} - \eqref{eq:f000} for $F_{ajk}$.
For example, using that 
\[
D_{\tau} \Ps = D_{\tau}\Pb = \Delta\Ps = \Delta \Pb = 0,
\]
we have
\begin{align*}
 (D_{\tau} + \Delta)\nablat_{\bar{a}}\Rt_{\bar{\jmath}k} &= \Pb_a^p\Pb_j^q(D_{\tau} + \Delta)\nablat_{p}\Rt_{qk}\\
 &= \oinf \ast(h + \nabla h + \nablat \Rmt) + F_{\bar{a}\bar{\jmath}k}.
\end{align*}
Then, using \eqref{eq:f11x}, we see that
\begin{align*}
  (D_{\tau} + \Delta)(\tau \nablat_{\bar{a}}\Rt_{\bar{\jmath}k}) &= \oinf \ast(h + \nabla h + \nablat \Rmt) + \nablat_{\bar{a}}\Rt_{\bar{\jmath}k} + \tau F_{\bar{a}\bar{\jmath}k}\\
  &=\oinf \ast(h + \nabla h + \nablat \Rmt) + \nablat_{\bar{a}}\Rt_{\bar{\jmath}\bar{p}\bar{p}k},
\end{align*}
which implies \eqref{eq:drc11x}.  Relations \eqref{eq:drc110}- \eqref{eq:drc000} can be verified similarly. For \eqref{eq:drc001}, we use the Bianchi identities to estimate the third term in \eqref{eq:f001}.

The identities \eqref{eq:g001} - \eqref{eq:g000} follow in the same way from the identities 
\begin{align*}
\begin{split}
 F_{\ub{a}\ub{j}\bar{k}} - F_{\ub{j}\ub{a}\bar{k}} &= - \left(\frac{1+\cnst}{\tau}\right)\Gt_{\ub{a}\ub{j}\bar{k}} 
 -\frac{1}{\tau}\nablat_{\bar{p}} \Rt_{\ub{a}\ub{j}\bar{p}\bar{k}},
\end{split}\\
\begin{split}
F_{\ub{a}\ub{j}\ub{k}} - F_{\ub{j}\ub{a}\ub{k}} &= -\frac{1}{\tau}\Gt_{\ub{a}\ub{j}\ub{k}} +\frac{1}{\tau}\nablat_{\bar{p}} \Rt_{\ub{a}\ub{j}\ub{k}\bar{p}}+ \frac{\cnst}{\tau}\Ps_{ak}\left(\nablat_{\bar{p}}\Rt_{\bar{p}\ub{j}} - \nablat_{\ub{j}}\Rt_{\bar{p}\bar{p}} 
 + \frac{1}{2}\nablat_{\ub{j}}\Rt\right)\\ 
&\phantom{= } -\frac{\cnst}{\tau}\Ps_{jk}\left( \nablat_{\bar{p}}\Rt_{\bar{p}\ub{a}} - \nablat_{\ub{a}}\Rt_{\bar{p}\bar{p}}+ \frac{1}{2} \nablat_{\ub{a}} \Rt\right),
\end{split}
\end{align*}
which are consequences of \eqref{eq:f001} and \eqref{eq:f000} and the  Bianchi identities.
\end{proof}

\subsection{Decomposition of $\nablat \Rmt$}
Finally we examine the full covariant derivative of $\Rmt$. We will only need expressions for some of the components to obtain a closed system of inequalities.
\begin{proposition}\label{prop:drmev}
 The components of $\nablat \Rmt$ satisfy
\begin{align}
  \label{eq:drmx1111}
  &|(D_{\tau} + \Delta)\nablat_{a}\Rt_{\bar{\imath}\bar{\jmath}\bar{k}\bar{l}}| \lesssim 0, \\
  \label{eq:drm10111}
  &|(D_{\tau} + \Delta)\nablat_{\bar{a}}\Rt_{\ub{i}\bar{\jmath}\bar{k}\bar{l}}| \lesssim 0,\\
  \label{eq:drm10011}
  &|(D_{\tau} + \Delta)\tau^{\cnst}\nablat_{\bar{a}}\Rt_{\ub{i}\ub{j}\bar{k}\bar{l}}| \lesssim 0,\\
  \begin{split}
  \label{eq:drm10110}
  &|(D_{\tau} + \Delta)\tau^{-\cnst}\nablat_{\bar{a}}\Rt_{\ub{i}\bar{\jmath}\bar{k}\ub{l}}|\lesssim 
  \tau^{-(1+\cnst)}\left(|\nablat_{\bar{a}}\Rt_{\ub{i}\ub{j}\bar{k}\bar{l}}| + |\nablat_{\bar{a}}\Rt_{\bar{\imath}\bar{\jmath}\bar{k}\bar{l}}|
  + |\nablat_{\bar{a}}\Rt_{\bar{\imath}\bar{\jmath}}|\right),
  \end{split}\\
  \label{eq:drm10001}
  &|(D_{\tau} + \Delta)\nablat_{\bar{a}}\Rt_{\ub{i}\ub{j}\ub{k}\bar{l}}| \lesssim 
  \tau^{-1}\left(|\nablat_{\bar{a}}\Rt_{\ub{i}\bar{\jmath}\bar{k}\bar{l}}| 
  + |\nablat_{\bar{a}}\Rt_{\ub{i}\bar{\jmath}}|\right),\\
  \label{eq:drm10000}
  &|(D_{\tau} + \Delta)\nablat_{\bar{a}}\Rt_{\ub{i}\ub{j}\ub{k}\ub{l}}| \lesssim 
  \tau^{-1}\left(|\nablat_{\bar{a}}\Rt_{\ub{i}\bar{\jmath}\bar{k}\ub{l}}| + |\nablat_{\bar{a}}\Rt_{\ub{i}\ub{j}}|\right),\\
   \begin{split}
   \label{eq:drm00110}
   &|(D_{\tau} + \Delta)\tau^{-3\cnst}\nablat_{\ub{a}}\Rt_{\ub{i}\bar{\jmath}\bar{k}\ub{l}}|\\
   &\qquad\lesssim 
   \tau^{-(1+3\cnst)}\left( 
   |\nablat_{\bar{a}}\Rt_{\ub{i}\ub{j}\ub{k}\bar{l}}|+ |\nablat_{\bar{a}}\Rt_{\ub{i}\bar{\jmath}\bar{k}\bar{l}}|      
  + |\nablat_{\bar{a}}\Rt_{\ub{i}\bar{\jmath}}| 
  + |\nablat_{\ub{a}}\Rt_{\bar{\imath}\bar{\jmath}}|\right),
  \end{split}\\
   \begin{split}
   \label{eq:drm00001}
   &|(D_{\tau} + \Delta)\tau^{-2\cnst}\nablat_{\ub{a}}\Rt_{\ub{i}\ub{j}\ub{k}\bar{l}}|\\
   &\qquad \lesssim 
   \tau^{-(1+2\cnst)}\left(
   |\nablat_{\bar{a}}\Rt_{\ub{i}\ub{j}\ub{k}\ub{l}}|
    + |\nablat_{\bar{a}}\Rt_{\ub{i}\bar{\jmath}\bar{k}\ub{l}}| + |\nablat_{\ub{a}}\Rt_{\ub{i}\bar{\jmath}}| + |\nablat_{\bar{a}}\Rt_{\ub{i}\ub{j}}| \right),
  \end{split}\\
   \label{eq:drm00000}
&   |(D_{\tau} + \Delta)\tau^{-2\cnst}\nablat_{\ub{a}}\Rt_{\ub{i}\ub{j}\ub{k}\ub{l}}| \lesssim 
   \tau^{-(1+2\cnst)}\left(|\nablat_{\ub{a}}\Rt_{\ub{i}\bar{\jmath}\bar{k}\ub{l}}| + |\nablat_{\ub{a}}\Rt_{\ub{i}\ub{j}}|\right),
\end{align}
where, here, by $|U|\lesssim |V|$, we mean
\[
 |U| \lesssim |\oinf| (|h| + |\nabla h| + |\nablat \Rmt|) + C|V|
\]
for some constant $C = C(n) > 0$.
\end{proposition}

We will break the proof of Proposition \ref{prop:drmev} into a few smaller pieces. First, note that, from 
the proof of Lemma \ref{lem:drmtev}, we have
\begin{align*}
\begin{split}
 &\left(D_{\tau} + \Deltat\right)\delt_a\Rt_{ijkl} = -2\delt_a\left(\Bt_{ijkl} - \Bt_{ijlk} + \Bt_{ikjl} - \Bt_{iljk}\right)\\
 &\quad+ 2\gt^{pr}\gt^{qs}\big(\Rt_{iqap}\delt_r \Rt_{sjkl} + \Rt_{jqap}\delt_r \Rt_{iskl} + \Rt_{kqap}\delt_r\Rt_{ijsl}+ \Rt_{lqap}\delt_r\Rt_{ijks}\big)\\
&\quad+\big((\gt^{pq}\Rt_{ap} - R_a^q)\delt_q\Rt_{ijkl} + (\gt^{pq}\Rt_{ip} - R_i^q)\delt_a\Rt_{qjkl}(\gt^{pq}\Rt_{jp} - R_j^q)\delt_a\Rt_{iqkl}\\
&\quad+ +(\gt^{pq}\Rt_{kp} - R_k^q)\delt_a\Rt_{ijql} 
+  (\gt^{pq}\Rt_{lp} - R_l^q)\delt_a\Rt_{ijkq}\big).
 \end{split}
 \end{align*}
Thus, as a preliminary step, we may write
\begin{align}
 \begin{split}\label{eq:drmev1}
 (D_{\tau} + \Delta)\nablat_a\Rt_{ijkl} = \oinf \ast (h + \delt h + \nablat \Rmt) + J_{aijkl} + L_{aijkl},
 \end{split}
\end{align}
where
\begin{align*}
 J_{aijkl} &=-2\nablat_a\left(\Bt_{ijkl} - \Bt_{ijlk} + \Bt_{ikjl} - \Bt_{iljk}\right)
\end{align*}
and
\[
 L_{aijkl} = 2\big(R_{iqap}\nablat_p \Rt_{qjkl} + R_{jqap}\nablat_p \Rt_{iqkl} + R_{kqap}\nablat_p\Rt_{ijql} + R_{lqap}\nablat_p\Rt_{ijkq}\big).
\]

So, for the inequalities \eqref{eq:drmx1111} - \eqref{eq:drm00000}, we only need to analyze the structure of the tensors $J$ and $L$. We consider the tensor $J$ first. 

\begin{proposition}\label{prop:jcomp}
 The components of the tensor
 \begin{align*}
 J_{aijkl} &=-2\nablat_a\left(\Bt_{ijkl} - \Bt_{ijlk} + \Bt_{ikjl} - \Bt_{iljk}\right),
\end{align*}
satisfy the relations
 \begin{align}
  \begin{split}
    \label{eq:jxx111}
      J_{ai\bar{\jmath}\bar{k}\bar{l}} &\simeq 0, 
  \end{split}\\
   \begin{split}
    \label{eq:jx0011}
      J_{a\ub{i}\ub{j}\bar{k}\bar{l}} &\simeq - \frac{\cnst}{\tau}\nablat_a\Rt_{\ub{i}\ub{j}\bar{k}\bar{l}}, 
  \end{split}\\
    \begin{split}
    \label{eq:jx0110}
      J_{a\ub{i}\bar{\jmath}\bar{k}\ub{l}} &\simeq  \frac{\cnst}{\tau}\left(\nablat_a\Rt_{\ub{i}\bar{\jmath}\bar{k}\ub{l}} + \nablat_a \Rt_{\ub{i}\ub{l}\bar{\jmath}\bar{k}}
      + (\nablat_a \Rt_{\bar{\jmath}\bar{p}\bar{p}\bar{k}} - \nablat_a \Rt_{\bar{\jmath}\bar{k}})\Ps_{il}\right),      
  \end{split}\\
      \begin{split}
    \label{eq:jx0001}
      J_{a\ub{i}\ub{j}\ub{k}\bar{l}} &\simeq  
       \frac{\cnst}{\tau}\left((\nablat_a \Rt_{\ub{j}\bar{l}} - \nablat_a \Rt_{\ub{j}\bar{p}\bar{p}\bar{l}} )\Ps_{ik}
      - (\nablat_a \Rt_{\ub{i}\bar{l}} - \nablat_a \Rt_{\ub{i}\bar{p}\bar{p}\bar{l}})\Ps_{jk}\right), 
  \end{split}\\
  \begin{split}
    \label{eq:jx0000}
      J_{a\ub{i}\ub{j}\ub{k}\ub{l}} &\simeq  
       \frac{\cnst}{\tau}\left((\nablat_a \Rt_{\ub{j}\ub{l}} - \nablat_a \Rt_{\ub{j}\bar{p}\bar{p}\ub{l}} )\Ps_{ik}
      + (\nablat_a \Rt_{\ub{i}\ub{k}} - \nablat_a \Rt_{\ub{i}\bar{p}\bar{p}\ub{k}})\Ps_{jl}\right) \\
      &\phantom{\simeq} -\frac{\cnst}{\tau}\left((\nablat_a \Rt_{\ub{i}\ub{l}} - \nablat_a \Rt_{\ub{i}\bar{p}\bar{p}\ub{l}} )\Ps_{jk}
      + (\nabla_a \Rt_{\ub{j}\ub{k}} - \nablat_a \Rt_{\ub{j}\bar{p}\bar{p}\ub{k}})\Ps_{il}\right), 
  \end{split}
 \end{align}
where, here, $U \simeq V$ signifies that
\[
  U = \oinf \ast  \delt\Rmt + V.
\]
\end{proposition}
\begin{proof}
We first compute that
\begin{align*}
\nablat_a \Bt_{ijkl} &=  -\gt^{pr}\gt^{qs}\left(\nablat_a\Rt_{pijq}\Rt_{rkls} + \Rt_{pijq}\nablat_a\Rt_{rkls}\right)\\
&= \oinf \ast \nablat \Rmt - \nablat_a\Rt_{pijq}R_{pklq} - \nablat_a\Rt_{pklq}R_{pijq}\\
\begin{split}
&=\oinf \ast \nablat \Rmt - \frac{\cnst}{2\tau}\nablat_a\Rt_{pijq}(\Ps_{pq}\Ps_{kl} - \Ps_{pl}\Ps_{kq}) \\
&\phantom{=} -\frac{\cnst}{2\tau} \nablat_a\Rt_{pklq}(\Ps_{pq}\Ps_{ij}- \Ps_{pj}\Ps_{iq})\\
&=\oinf \ast \nablat \Rmt - \frac{\cnst}{2\tau}\left(\nablat_a\Rt_{i\ub{p}\ub{p}j}\Ps_{kl} + \nablat_a\Rt_{k\ub{p}\ub{p}l}\Ps_{ij}\right)\\
&\phantom{=}+ \frac{\cnst}{2\tau}\left( \nablat_a\Rt_{\ub{l}ij\ub{k}} + \nablat_a \Rt_{\ub{j}kl\ub{i}}\right)
\end{split}
\end{align*}
for any $a$, $i$, $j$, $k$, $l$. Permuting the indices in this identity and summing, we obtain
\begin{align*}
\begin{split}
J_{aijkl}  &= \oinf \ast \nablat \Rmt \\
&\phantom{=}+ \frac{\cnst}{\tau} \left( \nablat_a\Rt_{i\ub{p}\ub{p}j}\Ps_{kl} + \nablat_a\Rt_{k\ub{p}\ub{p}l}\Ps_{ij} -  \nablat_a\Rt_{\ub{l}ij\ub{k}}- \nablat_a \Rt_{\ub{j}kl\ub{i}}\right)\\ 
&\phantom{=}- \frac{\cnst}{\tau} \left( \nablat_a\Rt_{i\ub{p}\ub{p}j}\Ps_{lk} + \nablat_a\Rt_{l\ub{p}\ub{p}k}\Ps_{ij} -  \nablat_a\Rt_{\ub{k}ij\ub{l}}- \nablat_a \Rt_{\ub{j}lk\ub{i}}\right)\\
&\phantom{=}+ \frac{\cnst}{\tau} \left( \nablat_a\Rt_{i\ub{p}\ub{p}k}\Ps_{jl} + \nablat_a\Rt_{j\ub{p}\ub{p}l}\Ps_{ik} -  \nablat_a\Rt_{\ub{l}ik\ub{j}}- \nablat_a \Rt_{\ub{k}jl\ub{i}}\right)\\
&\phantom{=}- \frac{\cnst}{\tau} \left( \nablat_a\Rt_{i\ub{p}\ub{p}l}\Ps_{jk} + \nablat_a\Rt_{j\ub{p}\ub{p}k}\Ps_{il} -  \nablat_a\Rt_{\ub{k}il\ub{j}}- \nablat_a \Rt_{\ub{l}jk\ub{i}}\right), 
\end{split}
\end{align*}
that is,
\begin{align}
\begin{split}\nonumber
J_{aijkl}  &= \oinf \ast \nablat \Rmt  - \frac{\cnst}{\tau}(\operatorname{tr}_{\Ps}(\nablat_a\Rmt) \odot \Ps)_{ijkl}\\
&\phantom{=} + \frac{\cnst}{\tau} \left(\nablat_a\Rt_{\ub{k}ij\ub{l}} + \nablat_a \Rt_{\ub{j}lk\ub{i}} -  \nablat_a\Rt_{\ub{l}ij\ub{k}}- \nablat_a \Rt_{\ub{j}kl\ub{i}}\right)\\ 
&\phantom{=}+ \frac{\cnst}{\tau} \left( \nablat_a\Rt_{\ub{k}il\ub{j}}+ \nablat_a \Rt_{\ub{l}jk\ub{i}}-  \nablat_a\Rt_{\ub{l}ik\ub{j}}- \nablat_a \Rt_{\ub{k}jl\ub{i}}\right) 
\end{split}\\
\begin{split}\label{eq:j1}
&= \oinf \ast \nablat \Rmt\\
&\phantom{=}- \frac{\cnst}{\tau}\left((\operatorname{tr}_{\Ps}(\nablat_a\Rmt) \odot \Ps)_{ijkl}+ \nablat_a\Rt_{\ub{i}\ub{j}kl} + \nablat_a \Rt_{ij\ub{k}\ub{l}}\right)\\ 
&\phantom{=}+ \frac{\cnst}{\tau} \left( \nablat_a\Rt_{\ub{k}il\ub{j}}+ \nablat_a \Rt_{\ub{l}jk\ub{i}}-  \nablat_a\Rt_{\ub{l}ik\ub{j}}- \nablat_a \Rt_{\ub{k}jl\ub{i}}\right), 
\end{split}
\end{align}
where
\[
 \operatorname{tr}_{\Ps}(\nablat_a\Rmt)_{ij} = \nablat_a \Rt_{i\ub{p}\ub{p}j},
\]
and $U\odot V$ denotes the Kulkarni-Nomizu product
\[
 (U\odot V)_{ijkl} = U_{il}V_{jk} + U_{jk} V_{il} - U_{ik}V_{jl} - U_{jl}V_{ik}.
\]

A case by case examination of of \eqref{eq:j1}, using the first Bianchi identity and the observation that 
\[
\operatorname{tr}_{\Ps}(\nablat_a\Rmt)_{ij} = \nablat_a \Rt_{i\ub{p}\ub{p}j} = \oinf \ast \nablat \Rmt + \nablat_a \Rt_{ij} - \nablat_a \Rt_{i\bar{p}\bar{p}j},
\]
yields \eqref{eq:jxx111} - \eqref{eq:jx0000}.
\end{proof}

Now we perform a similar analysis for the tensor $L$.
\begin{proposition}\label{prop:lcomp}
 The components of the tensor
 \[
 L_{aijkl} = 2\big(R_{iqap}\nablat_p \Rt_{qjkl} + R_{jqap}\nablat_p \Rt_{iqkl} + R_{kqap}\nablat_p\Rt_{ijql} + R_{lqap}\nablat_p\Rt_{ijkq}\big)
\]
satisfy the relations
 \begin{align}
    \label{eq:l1xxxx}
      L_{\bar{a}ijkl} &\simeq 0,\\
   \label{eq:l01111}   
      L_{\ub{a}\bar{\imath}\bar{\jmath}\bar{k}\bar{l}} &\simeq 0,\\
    \begin{split}
     \label{eq:l00111}
     L_{\ub{a}\ub{i}\bar{\jmath}\bar{k}\bar{l}} &\simeq \frac{\cnst}{\tau}\left(\nablat_{\ub{a}}\Rt_{\ub{i}\bar{\jmath}\bar{k}\bar{l}} + \nablat_{\bar{\jmath}}\Rt_{\ub{a}\ub{i}\bar{k}\bar{l}}\right)
     + \frac{\cnst}{\tau}\Ps_{ia}\left(\nablat_{\bar{p}}\Rt_{\bar{p}\bar{\jmath}\bar{k}\bar{l}} - \nablat_{\bar{l}}\Rt_{\bar{k}\bar{\jmath}} +\nablat_{\bar{k}}\Rt_{\bar{\jmath}\bar{l}}\right),
    \end{split}\\
     \begin{split}
     \label{eq:l00011}
     L_{\ub{a}\ub{i}\ub{j}\bar{k}\bar{l}} &\simeq \frac{\cnst}{\tau}\nablat_{\ub{a}}\Rt_{\ub{i}\ub{j}\bar{k}\bar{l}}   + \frac{\cnst}{\tau}\Ps_{ia}\left(\nablat_{\bar{p}}\Rt_{\bar{p}\ub{j}\bar{k}\bar{l}} 
     - \nablat_{\bar{l}}\Rt_{\bar{k}\ub{j}} +\nablat_{\bar{k}}\Rt_{\ub{j}\bar{l}}\right)\\
     &\phantom{\simeq}
     - \frac{\cnst}{\tau}\Ps_{ja}\left(\nablat_{\bar{p}}\Rt_{\bar{p}\ub{i}\bar{k}\bar{l}} - \nablat_{\bar{l}}\Rt_{\bar{k}\ub{i}} +\nablat_{\bar{k}}\Rt_{\ub{i}\bar{l}}\right),
    \end{split}\\
      \begin{split}
    \label{eq:l00110}
     L_{\ub{a}\ub{i}\bar{\jmath}\bar{k}\ub{l}} &\simeq \frac{\cnst}{\tau}\left(2\nablat_{\ub{a}}\Rt_{\ub{i}\bar{\jmath}\bar{k}\ub{l}} 
     - \nablat_{\bar{\jmath}}\Rt_{\ub{a}\ub{i}\ub{l}\bar{k}} - \nablat_{\bar{k}}\Rt_{\ub{a}\ub{l}\ub{i}\bar{\jmath}}\right)
      \\
      &\phantom{\simeq}
	+ \frac{\cnst}{\tau}\Ps_{ia}\left(\nablat_{\bar{p}}\Rt_{\bar{p}\bar{\jmath}\bar{k}\ub{l}} - \nablat_{\ub{l}}\Rt_{\bar{k}\bar{\jmath}} 
	+\nablat_{\bar{k}}\Rt_{\ub{l}\bar{\jmath}}\right)\\
&\phantom{\simeq}	- \frac{\cnst}{\tau}\Ps_{la}\left(\nablat_{\bar{p}}\Rt_{\bar{p}\bar{k}\ub{i}\bar{\jmath}} - \nablat_{\bar{\jmath}}\Rt_{\ub{i}\bar{k}}
+\nablat_{\ub{i}}\Rt_{\bar{\jmath}\bar{k}}\right),
    \end{split}\\
     \begin{split}
    \label{eq:l00001}
     L_{\ub{a}\ub{i}\ub{j}\ub{k}\bar{l}} &\simeq \frac{\cnst}{\tau}\left(2\nablat_{\ub{a}}\Rt_{\ub{i}\ub{j}\ub{k}\bar{l}} 
	+\nablat_{\bar{l}}\Rt_{\ub{a}\ub{k}\ub{i}\ub{j}}\right)
	+ \frac{\cnst}{\tau}\Ps_{ia}\left(\nablat_{\bar{p}}\Rt_{\bar{p}\ub{j}\ub{k}\bar{l}} - \nablat_{\bar{l}}\Rt_{\ub{j}\ub{k}} 
	+\nablat_{\ub{k}}\Rt_{\bar{l}\ub{j}}\right)\\
    &\phantom{\simeq}
    - \frac{\cnst}{\tau}\Ps_{ja}\left(\nablat_{\bar{p}}\Rt_{\bar{p}\ub{i}\ub{k}\bar{l}} - \nablat_{\bar{l}}\Rt_{\ub{i}\ub{k}} 
    +\nablat_{\ub{k}}\Rt_{\bar{l}\ub{i}}\right)\\
&\phantom{\simeq}	+ \frac{\cnst}{\tau}\Ps_{ka}\left(\nablat_{\bar{p}}\Rt_{\bar{p}\bar{l}\ub{i}\ub{j}} 
- \nablat_{\ub{j}}\Rt_{\ub{i}\bar{l}} +\nablat_{\ub{i}}\Rt_{\ub{j}\bar{l}}\right),
    \end{split}\\  
   \begin{split}
    \label{eq:l00000}
     L_{\ub{a}\ub{i}\ub{j}\ub{k}\ub{l}} &\simeq \frac{2\cnst}{\tau}\nablat_{\ub{a}}\Rt_{\ub{i}\ub{j}\ub{k}\ub{l}}
     + \frac{\cnst}{\tau}\Ps_{ia}\left(\nablat_{\bar{p}}\Rt_{\bar{p}\ub{j}\ub{k}\ub{l}} - \nablat_{\ub{l}}\Rt_{\ub{j}\ub{k}} 
     +\nablat_{\ub{k}}\Rt_{\ub{l}\ub{j}}\right)\\
    &\phantom{\simeq}
    - \frac{\cnst}{\tau}\Ps_{ja}\left(\nablat_{\bar{p}}\Rt_{\bar{p}\ub{i}\ub{k}\ub{l}} - \nablat_{\ub{l}}\Rt_{\ub{i}\ub{k}} 
    +\nablat_{\ub{k}}\Rt_{\ub{l}\ub{i}}\right)\\
&\phantom{\simeq}	+ \frac{\cnst}{\tau}\Ps_{ka}\left(\nablat_{\bar{p}}\Rt_{\bar{p}\ub{l}\ub{i}\ub{j}} 
- \nablat_{\ub{j}}\Rt_{\ub{i}\ub{l}} +\nablat_{\ub{i}}\Rt_{\ub{j}\ub{l}}\right)\\
&\phantom{\simeq} - \frac{\cnst}{\tau}\Ps_{la}\left(\nablat_{\bar{p}}\Rt_{\bar{p}\ub{k}\ub{i}\ub{j}} - \nablat_{\ub{j}}\Rt_{\ub{i}\ub{k}} 
+\nablat_{\ub{i}}\Rt_{\ub{j}\ub{k}}\right),
    \end{split}
 \end{align}
 where here $U\simeq V$ signifies that
 \[
  U = \oinf\ast \nablat \Rmt + V.
 \]
\end{proposition}
\begin{proof}
Note that
\begin{align*}
 R_{iqap}\nablat_p \Rt_{qkjl} &= \frac{\cnst}{2\tau}(\Ps_{ip}\Ps_{qa}- \Ps_{ia}\Ps_{qp})\nablat_p\Rt_{qjkl}
 = \frac{\cnst}{2\tau}(\nablat_{\ub{i}} \Rt_{\ub{a}jkl} - \Ps_{ia}\nablat_{\ub{p}}\Rt_{\ub{p}jkl}),
\end{align*}
and so
\begin{align*}
 L_{aijkl} &= 2\big(R_{iqap}\nablat_p \Rt_{qjkl} - R_{jqap}\nablat_p \Rt_{qikl} + R_{kqap}\nablat_p\Rt_{qlij} - R_{lqap}\nablat_p\Rt_{qkij}\big)\\
\begin{split} 
 &= \frac{\cnst}{\tau}\left(\nablat_{\ub{i}}\Rt_{\ub{a}jkl} - \nablat_{\ub{j}}\Rt_{\ub{a}ikl} - \Ps_{ia}\nablat_{\ub{p}}\Rt_{\ub{p}jkl} 
 + \Ps_{ja}\nablat_{\ub{p}}\Rt_{\ub{p}ikl}\right)\\
&\phantom{=} + \frac{\cnst}{\tau}\left(\nablat_{\ub{k}}\Rt_{\ub{a}lij} - \nablat_{\ub{l}}\Rt_{\ub{a}kij} - \Ps_{ka}\nablat_{\ub{p}}\Rt_{\ub{p}lij} 
+ \Ps_{la}\nablat_{\ub{p}}\Rt_{\ub{p}kij}\right).
\end{split}
\end{align*}
Using the identity
\[
 \gt^{pq}\nablat_{p}\Rt_{qjkl} = \nablat_l \Rt_{kj} - \nablat_k\Rt_{lj},
\]
we may rewrite the terms of the form $\nablat_{\ub{p}}\Rt_{\ub{p}jkl}$ in the above equation as
\[
 \nablat_{\ub{p}}\Rt_{\ub{p}jkl} = \oinf \ast \nablat \Rmt + \nablat_l \Rt_{kj} - \nablat_k\Rt_{lj} - \nablat_{\bar{p}}\Rt_{\bar{p}jkl}. 
\]
The relations \eqref{eq:l1xxxx}-\eqref{eq:l00000} then follow from a case-by-case inspection of the above identity for $L_{aijkl}$ using the Bianchi identities to combine terms.
\end{proof}

Now we combine the above computations to complete the proof of the main result of the section.
\begin{proof}[Proof of Proposition \ref{prop:drmev}]

As in the proof of Proposition \ref{prop:drcev}, the inequalities in \eqref{eq:drmx1111}-\eqref{eq:drm00000} follow from \eqref{eq:drmev1} and an inspection of the expressions
\eqref{eq:jxx111}-\eqref{eq:jx0000} and \eqref{eq:l1xxxx}-\eqref{eq:l00000} for the corresponding components of the tensors $J$ and $L$. 
We further use the Bianchi identities to estimate 
 $|\delt_{\ub{a}}\Rt_{\bar{\imath}\bar{\jmath}\bar{k}\bar{l}}|\leq 2|\delt_{\bar{a}}\Rt_{\ub{i}\bar{\jmath}\bar{k}\bar{l}}|$ and
$|\delt_{\ub{a}}\Rt_{\ub{i}\ub{j}\bar{k}\bar{l}}|\leq 2|\delt_{\bar{a}}\Rt_{\ub{i}\ub{j}\ub{k}\bar{l}}|$ in \eqref{eq:drm00110},
$|\delt_{\ub{a}}\Rt_{\ub{i}\bar{\jmath}\bar{k}\bar{l}}| \leq 2|\delt_{\bar{a}}\Rt_{\ub{i}\bar{\jmath}\bar{k}\ub{l}}|$
and $|\delt_{\bar{a}}\Rt_{\ub{i}\ub{j}\bar{k}\bar{l}}| \leq 2|\delt_{\bar{a}}\Rt_{\ub{i}\bar{\jmath}\bar{k}\ub{l}}|$
in \eqref{eq:drm00001},
and $|\delt_{\bar{a}}\Rt_{\ub{i}\ub{j}\ub{k}\bar{l}}|\leq 2|\delt_{\ub{a}}\Rt_{\ub{i}\bar{\jmath}\bar{k}\ub{l}}|$ in \eqref{eq:drm00000}.
\end{proof}

\subsection{Assembling the components of the system}
Next we use Propositions \ref{prop:dscalev}, \ref{prop:drcev}, and \ref{prop:drmev} to organize the rescaled components of $\nablat \Rt$, $\nablat\Rct$, and $\nablat\Rmt$ into groupings
which satisfy a closed system of inequalities whose singular part has a triangular structure. 

Define $\ve{W} = (W^{0}, W^{1}, \ldots, W^{5})$ by
\begin{align}
\label{eq:wdef}
\begin{split}
  W^{0} &= (\delt_{a} \Rt_{\bar{\imath}\bar{\jmath}\bar{k}\bar{l}}, \delt_{\bar{a}}\Rt_{\ub{i}\bar{\jmath}\bar{k}\bar{l}}, \tau^{\cnst}\delt_{\bar{a}}\Rt_{\ub{i}\ub{j}\bar{k}\bar{l}}),\\
  W^{1} &= (\tau \delt_{a}\Rt_{\bar{\imath}\bar{\jmath}}, \tau \delt_{\bar{a}}\Rt_{\ub{i}\bar{\jmath}}, \tau^{1+\cnst}\Gt_{\ub{a}\ub{i}\bar{\jmath}}),\\
  W^{2} &= (\tau^2\delt_a\Rt, \tau^{-\cnst}\delt_{\bar{a}}\Rt_{\ub{i}\bar{\jmath}\bar{k}\ub{l}}, \delt_{\bar{a}}\Rt_{\ub{i}\ub{j}\ub{k}\bar{l}} ),\\
  W^{3} &= (\tau^{1-\cnst}\delt_{\bar{a}}\Rt_{\ub{i}\ub{j}},\tau^{1-\cnst}\delt_{\ub{a}}\Rt_{\ub{i}\bar{\jmath}}, \tau \Gt_{\ub{a}\ub{i}\ub{j}}, \tau^{-3\cnst}\delt_{\ub{a}}\Rt_{\ub{i}\bar{\jmath}\bar{k}\ub{l}}  ),\\  
  W^{4} &=  (\tau^{1-3\cnst}\delt_{\ub{a}}\Rt_{\ub{i}\ub{j}}, \delt_{\bar{a}}\Rt_{\ub{i}\ub{j}\ub{k}\ub{l}}),\\
   W^{5} &= (\tau^{-2\cnst}\delt_{\ub{a}}\Rt_{\ub{i}\ub{j}\ub{k}\bar{l}}, \tau^{-2\cnst}\delt_{\ub{a}}\Rt_{\ub{i}\ub{j}\ub{k}\ub{l}}),
\end{split}
\end{align}
where, as before, $\cnst = 1/(k-1)$.

\begin{proposition}\label{prop:wtriangular}
The components $W^i$ of $\ve{W}$ satisfy the system
\begin{align*}
 |(D_{\tau} +\Delta)W^0| &\lesssim 0,\\
 |(D_{\tau} +\Delta)W^1| &\lesssim |W^0|,\\
 |(D_{\tau} +\Delta)W^2| &\lesssim \tau^{-(1+2\cnst)}|W^0| +\tau^{-(2+\cnst)}|W^1|,\\
 \begin{split}
 |(D_{\tau} +\Delta)W^3| &\lesssim \tau^{-(1+3\cnst)}|W^0| +\tau^{-(2+3\cnst)}|W^1|+ \tau^{-\max\{1+3\cnst, 2+\cnst\}}|W^2|,
\end{split}\\
\begin{split}
 |(D_{\tau} +\Delta)W^4| &\lesssim  \tau^{-(1+3\cnst)}|W^1| + \tau^{-(2+3\cnst)}|W^2|
+ \tau^{-\max\{1+3\cnst, 2-\cnst\}}|W^3|,
\end{split}\\
\begin{split}
  |(D_{\tau} +\Delta)W^5| &\lesssim \tau^{-(1+\cnst)}|W^2| + \tau^{-(2+\cnst)}|W^3|
+ \tau^{-\max\{1+2\cnst, 2-\cnst\}}|W^4|,
\end{split}
\end{align*}
on $\Cc_{r_0}\times (0, 1]$. Here, $|U|\lesssim |V|$ means that 
\[
  |U| \leq |\oinf| (|h| + |\nabla h| + |\nablat \Rmt|) + C|V|
\]
 for some constant $C = C(n) > 0$.  Moreover, we have
\begin{equation}\label{eq:drmxcomp}
|\delt\Rmt| + |\nabla \delt\Rmt| \leq C(|\ve{W}| + |\nabla \ve{W}|)
\end{equation}
on $\Cc_{r_0}\times (0, 1]$ for some $C = C(n)$.
\end{proposition}
\begin{proof}
Let us observe that  \eqref{eq:drmxcomp} is satisfied first. Using the symmetries of $\delt\Rmt$ and the Bianchi identities, we have
\begin{align*}
\begin{split}
 |\delt_a\Rt_{ijkl}| &\leq C\big(|\delt_{\bar{a}}\Rt_{\bar{\imath}\bar{\jmath}\bar{k}\bar{l}}| + |\delt_{\bar{a}}\Rt_{\ub{i}\bar{\jmath}\bar{k}\bar{l}}|
  + |\delt_{a}\Rt_{\ub{i}\bar{\jmath}\bar{k}\ub{l}}| + |\delt_{\ub{a}}\Rt_{\ub{i}\ub{j}\ub{k}\bar{l}}|+ |\delt_{\ub{a}}\Rt_{\ub{i}\ub{j}\ub{k}\ub{l}}|\big) 
\end{split}\\
 &\leq C\big(|W^0| + \tau^\cnst|W^2|
  + \tau^{3\cnst}|W^3| + \tau^{2\cnst}|W^5|\big)
\end{align*}
for some $C = C(n) > 0$. Similarly, $|\nabla\delt\Rmt|$ can be controlled by the sum of $|\nabla W^0|$, $|\nabla W^2|$, $|\nabla W^3|$, and $|\nabla W^5|$.

Now we verify the system of inequalities satisfied by the components of $\ve{W}$. Denoting the components of $W^i$ by $W^{i, j}$, we first see from \eqref{eq:drmx1111}-\eqref{eq:drm10011} that
\[
 |(D_{\tau} + \Delta)W^{0, j}| \lesssim 0
\]
for $j = 0, 1, 2$. The inequality for $W^0$ follows. Next, from \eqref{eq:drc11x} and \eqref{eq:drc011}, we have
\begin{align*}
 |(D_{\tau} + \Delta)W^{1, 0}| &\lesssim |\delt_a \Rt_{\bar{\imath}\bar{\jmath}\bar{k}\bar{l}}| = |W^{0, 0}|,
\end{align*}
and
\begin{align*}
 |(D_{\tau} + \Delta)W^{1, 1}| &\lesssim |\delt_{\bar{a}}\Rt_{\ub{i}\bar{\jmath}\bar{k}\bar{l}}| = |W^{0, 1}|,
\end{align*}
while, from \eqref{eq:g001}, that
\begin{align*}
 |(D_{\tau} + \Delta)W^{1, 2}| &\lesssim \tau^{\cnst}|\nablat_{\bar{a}}\Rt_{\ub{i}\ub{j}\bar{k} \bar{l}}|
			       = |W^{0, 2}|.
\end{align*}
Taken together, these inequalities yield the relation for $W^1$.

For $W^2$, we start with \eqref{eq:dscalev}, which implies
\[
 |(D_{\tau} + \Delta)W^{2, 0}| \lesssim \tau|\delt_{a}\Rt_{\bar{\imath}\bar{\jmath}}| = |W^{1, 0}|.
\]
Then \eqref{eq:drm10110} and \eqref{eq:drm10001} yield, respectively, that
\begin{align*}
 |(D_{\tau} + \Delta)W^{2, 1}|&\lesssim \tau^{-(1+\cnst)}(|\delt_{\bar{a}}\Rt_{\bar{\imath}\bar{\jmath}\bar{k}\bar{l}}| + 
 |\delt_{\bar{a}}\Rt_{\ub{i}\ub{j}\bar{k}\bar{l}}|  + |\delt_{\bar{a}}\Rt_{\bar{\imath}\bar{\jmath}}|)\\
 &\lesssim\tau^{-(1+\cnst)}|W^{0, 0}| + \tau^{-(1+2\cnst)}|W^{0, 2}| + \tau^{-(2+\cnst)}|W^{1, 0}|,
\end{align*}
and
\[
 |(D_{\tau} + \Delta)W^{2, 2}|\lesssim \tau^{-1}(|\delt_{\bar{a}}\Rt_{\ub{i}\bar{\jmath}\bar{k}\bar{l}}| + |\delt_{\bar{a}}\Rt_{\ub{i}\bar{\jmath}}|)
 \lesssim \tau^{-1}|W^{0, 1}| + \tau^{-2}|W^{1, 1}|,
\]
and the inequality for $W^2$ follows.

Similarly, using \eqref{eq:drc110} and \eqref{eq:drc001}, we see that
\begin{align*}
  \begin{split}
   |(D_{\tau} + \Delta)W^{3, 0}| &\lesssim \tau^{-\cnst}(|\nablat_{\bar{a}}\Rt| 
   + |\nablat_{\bar{a}}\Rt_{\bar{\jmath}\bar{k}}| + |\nablat_{\bar{a}}\Rt_{\ub{i}\bar{\jmath}\bar{k} \ub{l}}|)\\
   &\lesssim    \tau^{-(1+\cnst)}|W^{1, 0}| + \tau^{-(2+\cnst)}|W^{2, 0}| + |W^{2, 1}|,
  \end{split}
\end{align*}
and
\begin{align*}
  \begin{split}
&   |(D_{\tau} + \Delta)W^{3, 1}| \lesssim \tau^{-\cnst}(|\nablat_{\bar{a}}\Rt| + |\Gt_{\ub{a}\ub{j}\bar{k}}| +
   |\nablat_{\bar{a}}\Rt_{\bar{\jmath}\bar{k}}| + |\nablat_{\bar{a}}\Rt_{\ub{i} \bar{\jmath}\bar{k} \ub{l}}|)\\
    &\qquad\lesssim  \tau^{-(1+\cnst)}|W^{1, 0}| + \tau^{-(1+2\cnst)}|W^{1, 2}|+  \tau^{-(2+\cnst)}|W^{2, 0}|   + |W^{2, 1}|,
  \end{split}
\end{align*}
while, using \eqref{eq:g000} and \eqref{eq:drm00110}, we see that
\begin{align*}
 |(D_{\tau} + \Delta)W^{3, 2}| &\lesssim 
 |\nablat_{\ub{a}}\Rt| +
   |\nablat_{\ub{a}}\Rt_{\bar{\jmath}\bar{k}}| + |\nablat_{\bar{a}}\Rt_{\ub{j}\bar{k}}| + |\nablat_{\bar{a}}\Rt_{\ub{i} \ub{j}\ub{k} \bar{l}}|\\
   &\lesssim  \tau^{-1}|W^{1, 0}| + \tau^{-1}|W^{1, 1}| +  \tau^{-2}|W^{2, 0}| +  |W^{2, 2}|,
\end{align*}
and
\begin{align*}
   &|(D_{\tau} + \Delta)W^{3, 3}| \lesssim 
   \tau^{-(1+3\cnst)}\bigg(
   |\nablat_{\bar{a}}\Rt_{\ub{i}\ub{j}\ub{k}\bar{l}}|+ |\nablat_{\bar{a}}\Rt_{\ub{i}\bar{\jmath}\bar{k}\bar{l}}| + |\nablat_{\bar{a}}\Rt_{\ub{i}\bar{\jmath}}|
   + |\nablat_{\ub{a}}\Rt_{\bar{\imath}\bar{\jmath}}|\bigg)\\
   &\qquad\lesssim \tau^{-(1+3\cnst)}(|W^{0, 1}| + |W^{2, 2}|) + \tau^{-(2+3\cnst)}(|W^{1, 0}| + |W^{1, 1}|).
\end{align*}
Combining these relations yields the inequality for $W^3$.

Next, from \eqref{eq:drc000} and \eqref{eq:drm10000}, we have
\begin{align*}
  \begin{split}
   &|(D_{\tau} + \Delta)W^{4, 0}| \\
   &\qquad\lesssim \tau^{-3\cnst}\big(|\nablat_{\ub{a}}\Rt| + |\Gt_{\ub{a}\ub{j}\ub{k}}| +
   |\nablat_{\ub{a}}\Rt_{\bar{\jmath}\bar{k}}| + |\nablat_{\bar{a}}\Rt_{\ub{j}\bar{k}}| + |\nablat_{\ub{a}}\Rt_{\ub{i} \bar{\jmath}\bar{k} \ub{l}}|\big)\\
   &\qquad\lesssim \tau^{-(2+3\cnst)}|W^{2, 0}| +\tau^{-(1+3\cnst)}(|W^{1, 0}| + |W^{1, 1}| + |W^{3, 2}|) + |W^{3, 3}|, 
  \end{split}
\end{align*}
and
\begin{align*}
  |(D_{\tau} + \Delta)W^{4, 1}|&\lesssim 
  \tau^{-1}\left(|\nablat_{\bar{a}}\Rt_{\ub{i}\bar{\jmath}\bar{k}\ub{l}}| + |\nablat_{\bar{a}}\Rt_{\ub{i}\ub{j}}|\right)\lesssim tau^{-(1-\cnst)}|W^{2, 1}| + \tau^{-(2-\cnst)}|W^{3, 0}|,
\end{align*}
which together yield the inequality for $W^4$.

Finally, to obtain the inequality for $W^5$, we use 
\begin{align*}
\begin{split}
&   |(D_{\tau} + \Delta)W^{5, 0}| \lesssim 
   \tau^{-(1+2\cnst)}\bigg(
   |\nablat_{\bar{a}}\Rt_{\ub{i}\ub{j}\ub{k}\ub{l}}|
    + |\nablat_{\bar{a}}\Rt_{\ub{i}\bar{\jmath}\bar{k}\ub{l}}|  
   + |\nablat_{\ub{a}}\Rt_{\ub{i}\bar{\jmath}}| + |\nablat_{\bar{a}}\Rt_{\ub{i}\ub{j}}| \bigg)
\end{split}\\
&\qquad \lesssim   \tau^{-(1+\cnst)}|W^{2, 1}| + \tau^{-(2+\cnst)}|W^{3, 0}| +\tau^{-(2+\cnst)}|W^{3, 1}|  + \tau^{-(1+2\cnst)}|W^{4, 1}|,
\end{align*}
from \eqref{eq:drm00001}, and
\begin{align*}
|(D_{\tau} + \Delta)W^{5, 1}| &\lesssim 
   \tau^{-(1+2\cnst)}\left(|\nablat_{\ub{a}}\Rt_{\ub{i}\bar{\jmath}\bar{k}\ub{l}}| + |\nablat_{\ub{a}}\Rt_{\ub{i}\ub{j}}|\right)\\
   &\lesssim \tau^{-(1-\cnst)}|W^{3, 3}| + \tau^{-(2-\cnst)}|W^{4, 0}|,
\end{align*}
from \eqref{eq:drm00000}.

\end{proof}

Note that the largest exponent of $\tau$ which appears in the denominator of the coefficients of $|W^i|$ on the right side of the above relations is $\gamma = 2+ 3\cnst$.
Returning to Proposition \ref{prop:wtriangular} and unwinding the notation $\lesssim$, we summarize the findings of this section as follows.
\begin{proposition}\label{prop:wsys}
 For all $\beta > 0$, there is a constant $B_0 = B_0(\beta)$ depending on finitely many of the constants $M_{l, m}$
in \eqref{eq:xydecay} such that $\ve{W} = (W^0, W^1, \ldots, W^5)$ and $\ve{Y} = (h, \nabla h, \nabla \nabla h)$ together satisfy
\begin{equation}\label{eq:wsimple}
  |(D_{\tau} + \Delta)W^i| \leq B_0\tau^{\beta}(|\ve{W}| + |\ve{Y}|) + B_0\sum_{j=0}^{i-1}\tau^{-\gamma}|W^j|
\end{equation}
for $i = 0, 1, \dots 5$, and 
\begin{equation}\label{eq:yweq}
|D_{\tau}\ve{Y}| \leq B_0(|\ve{W}| + |\nabla\ve{W}|) + B_0\tau^{-1}|\ve{Y}|
\end{equation}
on $\Cc_{r_0}\times (0, 1]$. Here, $\gamma = 2 + 3/(k-1)$. Moreover, 
\[
 |\ve{X}| + |\nabla\ve{X}| \leq C(|\ve{W}| + |\nabla\ve{W}|)
\]
for some constant $C = C(n)$.
\end{proposition}

\section{Exponential Decay: The induction argument}\label{sec:expdecay}
The advantage of the system \eqref{eq:wsimple}-\eqref{eq:yweq} over the system \eqref{eq:xypdeode} is that the terms with singular coefficients in \eqref{eq:wsimple} appear in a strictly triangular form.
In this section, we will prove decay estimates for general systems with this triangular structure, and use these estimates to deduce Theorem \ref{thm:expdecay1}. 
These estimates will use the weights
\begin{equation}\label{eq:sigmagdef}
 \sigma(\tau) = \tau e^{\frac{T-\tau}{3}}, \quad G_{z_0}(z, \tau) = e^{-\frac{|z-z_0|^2}{4\tau}},
\end{equation}
for fixed $z_0 \in \RR^{n-k}$.
Note that $\sigma$ is comparable to $\tau$ in the sense that
\begin{equation}\label{eq:sigmataucomp}
  \tau \leq \sigma(\tau) \leq e^{\frac{T}{3}}\tau
\end{equation}
for $0 \leq \tau \leq T$, and that $\sigma^{\prime}(\tau) > 0$ and  $\sigma(\tau) \leq 1$ on $[0, T]$ as long as $T \leq 1$.

\begin{proposition}\label{prop:xypolydecay}
Let the bundles $\Wc = \oplus_{i=0}^q T^{(k_i, l_i)}(\Cc)$ and $\Yc = \oplus_{i=0}^{q^{\prime}}T^{(k_i, l_i)}(\Cc)$ be equipped with the family of metrics and connections induced by $g = g(\tau)$.
Suppose that $\ve{W} = (W^0, \ldots, W^{q})$ and $\ve{Y} = (Y^0, \ldots, Y^{q^{\prime}})$ are families of sections of $\Wc$ and $\Yc$ over $\Cc_{r_0}\times (0, 1]$ satisfying the following two conditions:
\begin{enumerate}
 \item[(a)] There are nonnegative constants $\beta$, $\gamma$, $\mu$, and $B$ such that
\begin{align}
\label{eq:wysys}
\begin{split}
  |(D_{\tau} + \Delta) W^i| &\leq B\tau^{\beta}(|\ve{W}| + |\ve{Y}|) + B\sum_{j=0}^{i-1}\tau^{-\gamma}|W^j|,\\
  |D_{\tau}\ve{Y}| & \leq B\tau^{-\mu}(|\ve{W}| + |\nabla\ve{W}|) + B\tau^{-1}|\ve{Y}|,
\end{split}
\end{align}
for each $i= 0, \ldots, q$ on $\Cc_{r_0}\times (0, 1]$.
\item[(b)] For each $l\geq 0$,
\begin{equation}\label{eq:spacetimedecay3}
   \sup_{\Cc_{r_0}\times (0, 1]} \frac{|z|^{2l}}{\tau^l} \left(|\ve{W}| + |\nabla\ve{W}| + |\ve{Y}|\right) \leq M_{l}
\end{equation}
for some constant $M_l\geq 0$.
\end{enumerate}

Then, there are positive constants $\beta_0 = \beta_0(k, n, q, \gamma, \mu)$ and $\lambda_0 = \lambda_0(k, n, \mu)$, and $L_0$,  $K_0$, and $T_0\leq 1$ depending on $k$, $n,$ $\gamma$, $\mu$, $B$, and finitely many of the constants $M_l$, such that, 
if $\beta \geq \beta_0$, the inequality
\begin{align}\label{eq:xypolydecay}
\begin{split}
&\int_0^T\int_{\Dc_{r}(z_0)}\left(\tau|\ve{W}|^2 + \tau^2|\nabla \ve{W}|^2+ \tau^{\lambda_0}|\ve{Y}|^2\right)\sigma^{-m}G_{z_0}\,\dX \\
&\qquad\qquad\leq K_0L_0^mr^{-2m}m!
\end{split}
\end{align}
holds 
for all $m\geq 0$ and all $r$, $T$, and $z_0$ with $0 < r^2 \leq T \leq T_0$ and $B_{4r}(z_0) \subset \RR^{n-k}\setminus\overline{B_{r_0}(0)}$.
\end{proposition}
The point is that the constants $\beta_0$, $\lambda_0$, $L_0$, $K_0$, and $T_0$ do not depend on $m$. 

\subsection{Proof of Theorem \ref{thm:expdecay1}}

We will prove Proposition \ref{prop:xypolydecay} by an induction argument in the next subsection. First we show that it
indeed implies Theorem \ref{thm:expdecay1}.

\begin{proof}[Proof of Theorem \ref{thm:expdecay1}, assuming Proposition \ref{prop:xypolydecay}]
By choosing the constant $B_0 = B_0(\beta)$ appropriately large in \eqref{eq:wsimple} and \eqref{eq:yweq}, we may assume that \eqref{eq:wysys} is satisfied with $\beta \geq \beta_0$, $\gamma = 2 + 3/(k-1)$, and $\mu = 0$.
Let $z_0\in \RR^{n-k}\setminus \overline{B_{8r_0}(0)}$ and $0 < T\leq T_0$.
Since $r_0 > 1$, we are assured that $B_{4r}(w) \subset  \RR^{n-k}\setminus\overline{B_{r_0}(0)}$
whenever $w\in B_{2}(z_0)$ and $0 \leq r \leq \sqrt{T} \leq 1$. At any such $w$, we may then combine \eqref{eq:drmxcomp} with \eqref{eq:xypolydecay} to obtain that,
for all $r\leq \sqrt{T}$ and $m\geq 0$, the inequality
\[
\int_0^{T}\int_{\Dc_{r}(w)}\tau^{p}(|\ve{X}|^2 + |\nabla \ve{X}|^2 + |\ve{Y}|^2)\sigma^{-m}G_{w}\,\dX \leq NL_0^mr^{-2m}m!
\]
holds for some $N = N(K_0)$ and fixed integer $p =\max\{\lambda_0, 2\}$.

Using that $\sigma(\tau) \leq \sqrt{e}\tau$, we then have 
\[
\frac{1}{(m-p)!}\int_0^{T}\int_{\Dc_{r}(w)}(|\ve{X}|^2 + |\nabla \ve{X}|^2 + |\ve{Y}|^2)\left(\frac{r^2}{4L\tau}\right)^{m-p}G_w\,\dX \leq \frac{N^{\prime} m^p}{4^mr^{2p}}
\]
for $L = \max\{\sqrt{e}L_0, 1\}$ and some $N^{\prime} = N^{\prime}(p, N, L_0)$.
Summing both sides of this inequality over all $m\geq p$ yields
\[
\int_0^{T}\int_{\Dc_{r}(w)}(|\ve{X}|^2 + |\nabla \ve{X}|^2 + |\ve{Y}|^2)e^{\frac{r^2-L|z-w|^2}{4L\tau}}\,\dX \leq N^{\prime\prime}r^{-2p},
\]
for some $N^{\prime\prime} = N^{\prime\prime}(p, N^{\prime})$,
and, consequently, that
\begin{equation}\label{eq:expdsmall}
 \int_0^{T}\int_{\Dc_{\frac{r}{2\sqrt{L}}(w)}}(|\ve{X}|^2 + |\nabla \ve{X}|^2 + |\ve{Y}|^2)e^{\frac{r^2}{8L\tau}}\,\dX \leq N^{\prime\prime}r^{-2p}.
\end{equation}

Returning to the statement of Theorem \ref{thm:expdecay1}, consider first the interval $[0, T]$ where $T = \min\{1, T_0\}$. We may cover $\Dc_{1}(z_0)$ with finitely many sets of the form
$\Dc_r(w_i)$, $i=1, \ldots, \nu$, where $r= \sqrt{T}/(2\sqrt{L})$ and $w_i\in \overline{B_1(z_0)} \subset B_2(z_0)$. This can be done so that the number of sets in the cover satisfies $\nu \leq C(L/T)^{(n-k)/2}$ for some dimensional constant $C$. 
Since $B_{4r}(w_i) \subset  \RR^{n-k}\setminus\overline{B_{r_0}(0)}$ for each $i$,
we may apply the estimate in \eqref{eq:expdsmall} on each $B_{\sqrt{T}/(2\sqrt{L})}(w_i)$ and sum to obtain that 
\begin{align*}
\int_0^{T}\int_{\Dc_{1}(z_0)}(|\ve{X}|^2 + |\nabla \ve{X}|^2 + |\ve{Y}|^2)e^{\frac{1}{8L\tau}}\,\dX &\leq CN^{\prime\prime}L^{\frac{n-k}{2}}T^{-p-\frac{n-k}{2}}.
\end{align*}
If $T_0 = 1$, we are done. Otherwise, if $T_0 < 1$, we may obtain an estimate of the same form on $[T_0, 1]$ since
 \[
   \int_{T_0}^{1}\int_{\Dc_{1}(z_0)}(|\ve{X}|^2 + |\nabla \ve{X}|^2 + |\ve{Y}|^2)e^{\frac{1}{8L\tau}}\,\dX \leq N^{\prime\prime\prime}(1-T_0)e^{\frac{1}{8LT_0}}
 \]
for some $N^{\prime\prime\prime}$ depending on $M_{0, m}$ for $m \leq 4$. Combining this estimate with the one on the interval $[0, T_0]$ then proves \eqref{eq:expdecay1}.
\end{proof}

\subsection{Three Carleman-type estimates}
We will prove Proposition \ref{prop:xypolydecay} by induction on the degree $m$ of polynomial decay. The induction step is based on the application of the following Carleman-type estimates
to $\ve{W}$ and $\ve{Y}$.  The estimates apply to arbitrary compactly supported families of sections of bundles $\mathcal{Z}$ of the form $\mathcal{Z} = \bigoplus T^{(k_i, l_i)}\Cc$
on $\Cc\times(0, 1]$ with metrics and connections induced by $g= g(\tau)$. 

The first Carleman estimate will be applied to a suitably cut-off version of the ``PDE'' component $\ve{W}$ of our system. A similar estimate was proven by the second author in \cite{WangCylindrical},
following \cite{EscauriazaSereginSverakHalfSpace}.
\begin{theorem}\label{thm:pdecarleman2}
Assume $0 < T \leq 2$. Then, for any $\alpha \geq 1$ and $z_0\in \RR^{n-k}$, the estimate
\begin{align}\label{eq:pdecarleman2}
\begin{split}
&\iint \sigma^{-2\alpha}\tau(\alpha|\ve{Z}|^2 + \tau|\nabla \ve{Z}|^2)G_{z_0}\,\dX \leq 10\iint \sigma^{-2\alpha}\tau^2|(D_{\tau} + \Delta)\ve{Z}|^2G_{z_0}\,\dX
\end{split}
\end{align}
holds for any smooth family of sections $\ve{Z}$ of $\mathcal{Z}$ with compact support in $\Cc\times(0, T)$.
\end{theorem}
We will use the next two estimates to control the component $\ve{Y}$.
\begin{theorem}\label{thm:odecarleman2}
Assume $0 < T\leq 2$ and let $D$, $U\subset \Cc$ be open sets such that $D$ is precompact and $\overline{D}\subset U$. For any $\lambda > 0$, there is $\alpha_0 = \alpha_0(\lambda, k) \geq 1$
such that, for all $\alpha \geq \alpha_0$ and
$z_0 \in \RR^{n-k}$
the estimates
 \begin{align}
\begin{split}\label{eq:odecarleman2g}
&2\alpha\int_0^T\int_D \tau^{\lambda}\sigma^{-2\alpha}|\ve{Z}|^2G_{z_0}\,\dX \\
&\qquad\leq  \int_0^T\int_D\tau^{\lambda-1}\sigma^{-2\alpha}|z- z_0|^2|\ve{Z}|^2G_{z_0}\,\dX\\
&\qquad\phantom{\leq}
+ 50\alpha^{-1}\int_0^T\int_D\tau^{\lambda+2}\sigma^{-2\alpha}|D_{\tau}\ve{Z}|^2G_{z_0}\,\dX,
\end{split}
\end{align}
and
\begin{align}
\begin{split}
 \label{eq:odecarleman2ng}
\alpha^2\int_0^T\int_D \tau^{\lambda}\sigma^{-2\alpha}|\ve{Z}|^2\,\dX \leq 16\int_0^T\int_D\tau^{\lambda+2} \sigma^{-2\alpha}|D_{\tau}\ve{Z}|^2\,\dX,
\end{split}
\end{align}
hold for all smooth families of sections $\ve{Z}$ of $\Zc$ over $U\times(0, T)$ with $\operatorname{supp} \ve{Z} \subset U\times [a, b]$ for some $0 < a < b < T$.
\end{theorem}
 Here the support of $Z$ need not be contained inside $D\times [a, b]$. We will prove Theorems \ref{thm:pdecarleman2} and \ref{thm:odecarleman2} in Section \ref{ssec:carlemandecay} below.

\subsection{A delocalization procedure}
Ideally, we would next apply the Gaussian-localized estimates \eqref{eq:pdecarleman2} and \eqref{eq:odecarleman2g} 
directly to (suitably cut-off versions of) $\ve{W}$ and $\ve{Y}$ and sum the resulting inequalities
to obtain the decay estimate needed for the induction step. However, the estimate \eqref{eq:odecarleman2g} turns out to be too lossy to allow us to do this in a single application. We will need to supplement it with estimates of $\ve{W}$ and $\ve{Y}$ relative to the purely time-dependent weight $\sigma$ on regions of spacetime where  $|z-z_0|^2/\tau > cm$ for some $c$. 

The lack of a sufficiently strong counterpart to \eqref{eq:pdecarleman2} for the ODE component is in fact the reason we need to employ an induction argument at all. By contrast, in \cite{KotschwarWangConical}, where the background metric converges smoothly to a conical metric as $\tau \to 0$, and in \cite{WangCylindrical}, 
  where the analysis reduces to that of a strictly parabolic inequality for a scalar equation,
 the exponential decay can be deduced in a single step.

 In our proof of Proposition \ref{prop:xypolydecay} in the next subsection, we will use the following two technical lemmas 
 to blend the localized estimates with the unlocalized ones.  The purpose of the first of these is to convert Gaussian-weighted $L^2$-bounds on 
 $\ve{W}$, $\nabla\ve{W}$ and $\ve{Y}$ on sets $\Dc_r(z)$ of a fixed radius $r$ into slightly weaker bounds minus the Gaussian weights on sets $\Dc_{s}(z)$
with $s \ll r$.  The proof is by an elementary covering argument.

\begin{lemma}\label{lem:fgtofng}
Suppose $0 < T^{\prime}\leq 1$ and $F$ is a positive smooth function on $\Cc_{r_0}\times (0, T^{\prime})$ with $|F| \leq M$ for some $M> 0$.
For all $\epsilon \in (0, 1/4)$ and $a> (n-k)/2$, there exists a constant $C_{a} = C_{a}(n, k)$
with the following property: 

Whenever, for some integer $m\geq 0$,
the inequality
\begin{align}\label{eq:fm}
\int_0^T\int_{\Dc_{r}(z_0)}F\sigma^{-m}G_{z_0}\,\dX
&\leq NL^mr^{-2m}m!
\end{align}
holds for some $N \geq M$ and $L\geq 1/(4\epsilon)^{2}$ and all $r$, $T$,  $z_0$ satisfying   $0 < r^2 \leq  T \leq T^{\prime}$ and
$B_{4r}(z_0) \subset \RR^{n-k}\setminus\overline{B_{r_0}(0)}$, 
the inequality
\begin{align}\label{eq:fml2}
\begin{split}
&\int_0^T\int_{\Dc_{4\epsilon r}(z_0)}\tau^{a}F\sigma^{-m}\,\dX \leq C_{a}NL^m((1-\epsilon)r)^{-2m}m!
\end{split}
\end{align}
holds for the same such $r$, $T$, and $z_0$.
\end{lemma}

\begin{proof} Fix $\epsilon \in (0, 1/4)$ and $a > (n-k)/2$ and suppose the inequality \eqref{eq:fm} holds for some $m\geq 0$ and $L\geq 1/(4\epsilon)^2$ and $N\geq M$,
for all $0 < r^2 \leq T \leq T^{\prime}$ and all $z_0\in \RR^{n-k}$ with $B_{4r}(z_0)\subset \RR^{n-k}\setminus\overline{B_{r_0}(0)}$.

Then fix a specific such $r$, $T$, and $z_0$ and let us verify that \eqref{eq:fml2} continues to hold. To begin, let $0 < \delta < 16\epsilon^2r^2$, and
split up the time interval to obtain the preliminary estimate
\begin{align}
\begin{split}\label{eq:fbd1}
  \int_{\delta}^T\int_{\Dc_{4\epsilon r}(z_0)}\tau^{a}F\sigma^{-m}\,\dX &=\left(\int_{\delta}^{16\epsilon^2r^2} + \int_{16\epsilon^2r^2}^T\right)\int_{\Dc_{4\epsilon r}(z_0)}\tau^{a}F\sigma^{-m}\,\dX\\
&\leq\int_{\delta}^{16\epsilon^2r^2}\int_{\Dc_{4\epsilon r}(z_0)}\tau^{a}F\sigma^{-m}\,\dX + CM(4\epsilon r)^{-2m}
\end{split}
\end{align}
for some $C = C(n, k)$.  

To estimate the first term on the right in \eqref{eq:fbd1}, observe that, for any $0 < s \leq 4\epsilon r$, we can cover $B_{4\epsilon r}(z_0)$ by
a collection of balls $\{B_{s}(w_i)\}_{i=1}^{\nu}$ with $w_i\in \overline{B_{4\epsilon r}(z_0)}$. The $w_i$ can be chosen so that their total number $\nu = \nu(s)$ will satisfy the bound
\[
 \nu(s) \leq c\left(\frac{4\epsilon r}{s}\right)^{n-k} 
\]
for some $c = c(n, k)$. We now define $s_j = 4\epsilon r/2^{j}$ and $\nu_j = \nu(s_j)$ for $j = 0, 1, 2, \ldots$, and apply this observation to choose collections 
$\{w_{i, j}\}_{i=1}^{\nu_j} \subset \overline{B_{4\epsilon r}(z_0)}$ of such points.

Since $w_{i, j}\in \overline{B_{4\epsilon r}(z_0)}$, 
\[
B_{4(1-\epsilon)r}(w_{i, j}) \subset B_{4r}(z_0) \subset \RR^{n-k}\setminus\overline{B_{r_0}(0)},
\]
and so the estimate \eqref{eq:fm} for $F$ is valid over $B_{(1-\epsilon)r}(w_{i, j})$. In particular, for each $w_{i, j}$, $j\geq 1$,
we have
\begin{align}
\nonumber
 \int_{s_{j}^2}^{s_{j-1}^2}\int_{\Dc_{s_{j}}(w_{i, j})}\tau^{a}F\sigma^{-m}\,\dX &\leq e^{\frac{1}{4}}s_{j-1}^{2a}\int_{s_{j}^2}^{s_{j-1}^2}\int_{\Dc_{s_j}(w_{i, j})}F\sigma^{-m}G_{w_{i, j}}\,\dX\\
\nonumber
 &\leq e^{\frac{1}{4}}\left(\frac{8\epsilon r}{2^{j}}\right)^{2a}\int_{0}^{T}\int_{\Dc_{(1-\epsilon)r}(w_{i, j})}F\sigma^{-m}G_{w_{i, j}}\,\dX\\
\label{eq:smallstball}
 &\leq e^{\frac{1}{4}}\left(\frac{1}{4^a}\right)^{j} \frac{(8\epsilon r)^{2a}NL^mm!}{((1-\epsilon)r)^{2m}}.
\end{align}
(In the second inequality, we have used that $s_{j} \leq 2\epsilon r < (1-\epsilon)r$ since $\epsilon < 1/4$.) We then may apply \eqref{eq:smallstball} to obtain that
\begin{align*}
\begin{split}
\int_{s_j^2}^{s_{j-1}^2}\int_{\Dc_{4\epsilon r}(z_0)}\tau^{a}F \sigma^{-m}\,\dX
&\leq \sum_{i=1}^{\nu_j}\int_{s_j^2}^{s_{j-1}^2}\int_{\Dc_{s_j}(w_{i, j})}\tau^{a}F\sigma^{-m}\,\dX\\
&\leq \left(\frac{1}{2^{2a - n + k}}\right)^{j}\frac{ce^{\frac{1}{4}}(8\epsilon r)^{2a}NL^mm!}{((1-\epsilon)r)^{2m}}
\end{split} 
\end{align*}
for each $j\geq 1$. 

Summing over $j$, we see that
\begin{align}
\nonumber
\int_{\delta}^{16\epsilon^2r^2}\int_{\Dc_{4\epsilon r}(z_0)}\tau^{a}F\sigma^{-m}\,\dX
&\leq \sum_{j=1}^{\infty} \int_{s_j^2}^{s_{j-1}^2}\int_{\Dc_{4\epsilon r}(z_0)}\tau^{a}F\sigma^{-m}\,\dX\\
\label{eq:fterm1}
&\leq C^{\prime}_a\frac{NL^mm!}{((1-\epsilon)r)^{2m}},
\end{align}
for some $C^{\prime}_a = C^{\prime}_a(n, k)$. Combining this with \eqref{eq:fbd1}, and sending $\delta\to 0$, we obtain
\begin{align*}
\begin{split}
 \int_0^T\int_{\Dc_{4\epsilon}(z_0)}\tau^{a}F\sigma^{-m}\,\dX &\leq\frac{C^{\prime}_aNL^mm!}{((1-\epsilon)r)^{2m}} + \frac{CM}{(4\epsilon r)^{2m}}\\
 &\leq (C + C^{\prime}_a)\frac{NL^mm!}{((1-\epsilon)r)^{2m}},
\end{split}
\end{align*}
since we have assumed that $L \geq 1/(4\epsilon)^2$ and $N \geq M$. So \eqref{eq:fml2} holds with the choice $C_{a} = C^{\prime}_a + C$.
\end{proof}

\subsection{Advancing the unlocalized bounds}
For the next lemma, we return to the setting of the statement of  Proposition \ref{prop:xypolydecay} and let $\ve{W}$ and $\ve{Y}$ be families of sections of $\Wc$ and $\Yc$ over $\Cc_{r_0}\times (0, 1]$
satisfying \eqref{eq:wysys} and \eqref{eq:spacetimedecay3} for some constants $\beta$, $\mu$, $B$, and $M_l$. 
We will use this lemma to convert $L^2$-bounds with time-dependent weights of degree $m$ on $\ve{W}$ and $\nabla \ve{W}$ into corresponding bounds of degree $m+1$ on $\ve{Y}$. The proof is a simple application of the estimate \eqref{eq:odecarleman2ng}, using the pointwise control of $D_{\tau} \ve{Y}$ by $\ve{W}$ and $\nabla\ve{W}$ implied by \eqref{eq:wysys}.

\begin{lemma}\label{lem:xmngtoympng} Fix $a\geq 0$ and $\lambda \geq 2\mu + a$. There is an integer $m_0\geq 0$ depending on $\lambda$, $k$, $B$, and $M_0$, such that whenever, for some $m\geq m_0$, $L\geq 2$,  and $N\geq 1$, 
the inequality
\begin{align}\label{eq:xyl2b}
\begin{split}
&\int_0^T\int_{\Dc_{r}(z_0)}\!\!\tau^a\left(|\ve{W}|^2 + \tau|\nabla \ve{W}|^2\right)\sigma^{-m}\,\dX \leq NL^mr^{-2m}m!
\end{split}
\end{align}
holds for some $r$, $T$, and $z_0$ satisfying   $0 < r^2 \leq  T \leq 1$ 
and $B_{2r}(z_0) \subset \RR^{n-k}\setminus\overline{B_{r_0}(0)}$,
the inequality
\begin{align}\label{eq:yl2mp}
\begin{split}
&\int_0^T\int_{\Dc_{r}(z_0)}\!\!\tau^{\lambda} |\ve{Y}|^2\sigma^{-(m+1)}\,\dX \leq NL^{m}r^{-2m}(m-1)!
\end{split}
\end{align}
also holds for the same $r$, $T$, and $z_0$.
\end{lemma}
\begin{proof} For now, we will take $m_0$ to be some large fixed integer; we will set lower bounds for it over the course of the proof. Suppose that \eqref{eq:xyl2b} holds
for some $m\geq m_0$ and $L\geq 2$,  and $N \geq 1$ at some $r$, $T$,  $z_0$ satisfying   $0 < r^2 \leq  T \leq 1$ 
and $B_{2r}(z_0) \subset \RR^{n-k}\setminus\overline{B_{r_0}(0)}$. 

For any $0 < \epsilon  < T/4$, let $\xi_{\epsilon}\in C^{\infty}(\RR)$ be a bump function with support in $(\epsilon, 3T/4)$
which is identically one on $[2\epsilon, T/2]$ and satisfies $|\xi_{\epsilon}^{\prime}| \leq C\epsilon^{-1}$ on $[\epsilon, 2\epsilon]$ and $|\xi_{\epsilon}^{\prime}|\leq CT^{-1}$ on $[T/2, 3T/4]$.
Here and below, $C$ will denote various positive constants depending at most on $n$ and $k$.

Define $\ve{W}_{\epsilon} = \xi_{\epsilon}\ve{W}$ and $\ve{Y}_{\epsilon} = \xi_{\epsilon}\ve{Y}$. Then, by \eqref{eq:wysys},
\[
 |D_{\tau}\ve{Y}_{\epsilon}|^2 \leq CB^2 \tau^{-2}|\ve{Y}_{\epsilon}|^2 + CB^2\tau^{-2\mu}(|\ve{W}_{\epsilon}|^2 + \tau|\nabla \ve{W}_{\epsilon}|^2) + 
 C|\xi_{\epsilon}^{\prime}|^2\ |\ve{Y}|^2.
\]
The first constraint we impose on $m_0$ is that $m_0 \geq 2\alpha_0(k, \lambda)$, where $\alpha_0$ is as in Theorem \ref{thm:odecarleman2}.  This allows us to
apply \eqref{eq:odecarleman2ng} with $D = \Dc_r(z_0)$, $\ve{Z} = \ve{Y}_{\epsilon}$, and $\alpha = (m+1)/2$ to obtain
\begin{align*}
\begin{split}
&\int_0^T\int_{\Dc_r(z_0)} \tau^{\lambda}\sigma^{-(m+1)}|\ve{Y}_{\epsilon}|^2\,\dX \\
&\quad\leq \frac{CB^2}{(m+1)^2}\int_0^T\int_{\Dc_{r}(z_0)} \tau^{\lambda}\sigma^{-(m+1)}|\ve{Y}_{\epsilon}|^2\,\dX\\
&\quad\phantom{\leq}+ \frac{CB^2}{(m+1)^2}\int_0^T\int_{\Dc_{r}(z_0)} \sigma^{-(m+1)}\tau^{\lambda-2\mu}(\tau|\ve{W}_{\epsilon}|^2 + \tau^2|\nabla \ve{W}_{\epsilon}|^2)\,\dX\\
&\quad\phantom{\leq}+ \frac{CB^2}{\epsilon^2(m+1)^2}\int_{\epsilon}^{2\epsilon}\int_{\Dc_{r}(z_0)}\sigma^{-(m+1)}\tau^{\lambda+2}|\ve{Y}|^2\,\dX\\
&\quad\phantom{\leq}+ \frac{CB^2}{T^2(m+1)^2}\int_{\frac{T}{2}}^{\frac{3T}{4}}\int_{\Dc_{r}(z_0)}\sigma^{-(m+1)}\tau^{\lambda+2}|\ve{Y}|^2\,\dX.
\end{split}
\end{align*}

Provided $m_0$ is taken greater still (say, $m_0 > \sqrt{2C}{B}$), we may hide the first term on the right in the term on the left. Having done this, we see that all of the integrands on the right are integrable on $(0, T]$ by our decay assumption \eqref{eq:spacetimedecay3},
and the third term will tend to $0$ when we send $\epsilon \searrow 0$. Taking this limit thus yields
\begin{align*}
\begin{split}
&\int_0^{\frac{T}{2}}\int_{\Dc_r(z_0)} \tau^{\lambda}\sigma^{-(m+1)}|\ve{Y}|^2\,\dX\\
&\qquad\leq \frac{CB^2}{(m+1)^2}\int_0^{\frac{3T}{4}}\int_{\Dc_r(z_0)} \sigma^{-(m+1)}\tau^{\lambda-2\mu}(\tau|\ve{W}|^2 + \tau^2|\nabla \ve{W}|^2)\,\dX\\
&\qquad\phantom{\leq}+ \frac{CB^2}{T^2(m+1)^2}\int_{\frac{T}{2}}^{\frac{3T}{4}}\int_{\Dc_{r}(z_0)}\sigma^{-(m+1)}\tau^{\lambda+2}|\ve{Y}|^2\,\dX.
\end{split}
\end{align*}
 
Since we assume $\lambda \geq 2\mu + a $, we may use \eqref{eq:xyl2b} (and that $\tau \leq \sigma$) to estimate 
\begin{align*}
&\int_0^{\frac{3T}{4}}\int_{\Dc_r(z_0)} \sigma^{-(m+1)}\tau^{\lambda-2\mu}(\tau|\ve{W}|^2 + \tau^{2}|\nabla\ve{W}|^2)\,\dX\\
&\qquad\leq \int_0^{T}\int_{\Dc_r(z_0)} \sigma^{-m}\tau^a(|\ve{W}|^2 + \tau|\nabla \ve{W}|^2)\,\dX \leq NL^mr^{-2m}m!.
\end{align*}
We may also estimate directly that
\begin{align*}
\int_{\frac{T}{2}}^{\frac{3T}{4}}\int_{\Dc_{r}(z_0)}\sigma^{-(m+1)}\tau^{\lambda+2}|\ve{Y}|^2\,\dX
&\leq CM_0^2r^{n-k}2^mT^{\lambda+ \frac{k}{2}+2 - m} \leq CM_0^2T^22^mr^{-2m}.
\end{align*}
Putting these two pieces together, we obtain 
\begin{align*}
\int_0^{\frac{T}{2}}\int_{\Dc_r(z_0)} \tau^{\lambda}\sigma^{-(m+1)}|\ve{Y}|^2\,\dX
&\leq CB^2\left(\frac{1 + M_0^2}{m+1}\right) NL^mr^{-2m}(m-1)!.
\end{align*}

On the other hand,
\[
 \int_{\frac{T}{2}}^{T}\int_{\Dc_r(z_0)} \tau^{\lambda}\sigma^{-(m+1)}|\ve{Y}|^2\,\dX
 \leq CM_0^2T^{\lambda}2^mr^{-2m} \leq CM_0^2NL^mr^{-2m},
\]
which, when added to the previous inequality, yields \eqref{eq:yl2mp},
provided $m_0$ is chosen larger still to ensure
\[
  m_0  \geq 1 + C(B^2 + (1+B^2)M_0^2).
\]
This completes the proof.
\end{proof}

\subsection{The induction argument}
In this section we prove Proposition \ref{prop:xypolydecay} using Lemmas \ref{lem:fgtofng} and \ref{lem:xmngtoympng}. We will use the notation
\[
  \Ac_{r, s}(z_0) = \Dc_{s}(z_0)\setminus \overline{\Dc_{r}(z_0)} =  \Ss^k\times (B_{s}(z_0) \setminus \overline{B_{r}(z_0)})
\]
for $0 < r < s$ and $z_0\in\RR^{n-k}$.  Note that $\Ac_{r, s}(0) = \Ac_{r, s}$.

\begin{proof}[Proof of Proposition \ref{prop:xypolydecay}] 
Define $\lambda_0 = 2\mu + (n-k)/2 + 2$ and fix any $b > \lambda_0/2$.
Then choose $\beta_0 = (q+1)b + q\gamma$, and let $m_0 = m_0(\lambda_0, k, B, M_0)$ be the constant guaranteed by Lemma \ref{lem:xmngtoympng}. Here $M_0$ is the constant from \eqref{eq:spacetimedecay3}. (This choice ensures
as well that $m_0 \geq 2\alpha_0$, where $\alpha_0 = \alpha_0(\lambda_0, k)$ is as in Theorem \ref{thm:odecarleman2}.)

Our proof is by induction on $m$. In view of the assumption \eqref{eq:spacetimedecay3} of infinite-order decay, we may start our induction at as large an integer $m_1$ as we like. It will be convenient to specify the value $m_1$ over the course of the proof, and to do the same for the constants $K_0$, $L_0$, and $T_0$ in the statement.
The choices of these constants will depend only on the external parameters $k$, $n$, $q$, $\beta$, $\mu$, $B$, and $M_0$, however, we will specify $K_0$ in terms of $m_1$, so logically it should be understood that $m_1$
is defined first. 

To help keep track of these dependencies, we'll use $C$ to denote a sequence of positive constants depending only on $k$, $n$, and $q$, and use $N$ to denote a sequence of positive constants potentially depending also on $B$ and $M_0$.
To begin, we'll
assume at least that $m_1 \geq m_0$, $K_0, L_0 > 0$, and $0 < T_0 \leq 1$.

Using the assumption \eqref{eq:spacetimedecay3}, we may assume there is $K\geq 1$ depending
on $m_1$, $n$, $k$ and finitely many of the constants $M_{l}$ such that
\begin{align*}
&\sup_{\Cc_{r_0}\times(0, 1]} \tau^{-m_1}(|\ve{W}|^2 + |\nabla \ve{W}|^2 + |\ve{Y}|^2)  \\
&\qquad + \int_0^1\int_{\Cc_{r_0}}\sigma^{-m_1}\left(|\ve{W}|^2 + |\nabla \ve{W}|^2 + |\ve{Y}|^2\right)\,\dX \leq K.
\end{align*}
In particular, for any $w$, $r$, and $T$ satisfying  $0 < r^2 \leq  T \leq T_0$ and
$B_{4r}(w) \subset \RR^{n-k}\setminus\overline{B_{r_0}(0)}$, we will have
\begin{align}\label{eq:ih}
\int_0^T\int_{\Dc_{r}(w)}\sigma^{-m_1}\left(\tau|\ve{W}|^2 + \tau^2|\nabla \ve{W}|^2 +\tau^{\lambda_0}|\ve{Y}|^2\right)G_{w}\, \dX \leq K. 
\end{align}
Provided $K_0 \geq M$ and $L_0 \geq 1$, at least, the inequality \eqref{eq:xypolydecay} will hold for all such admissible choices of $r$, $w$, and $T$ and all $m\leq m_1$.

Proceeding by induction, assume that $m > m_1$ and that \eqref{eq:xypolydecay} holds for all integers up to $m-1$. Fix 
$r$, $z_0$, and $T$ satisfying $0 < r^2 \leq T \leq T_0$ and $B_{4r}(z_0) \subset \RR^{n-k}\setminus\overline{B_{r_0}(0)}$. We will show that \eqref{eq:xypolydecay} also holds with exponent $m$
for $r$, $z_0$, and $T$.

We start by applying the Carleman inequality \eqref{eq:pdecarleman2} to a fixed component $W^i$ of $\ve{W}$ that has been cut-off in space and time. Let $\phi \in C^{\infty}(\RR^{n-k})$ be a smooth bump function
with support in $B_{2r}(z_0)$ which is identically one on $\overline{B_{r}(z_0)}$. Regarding $\phi$ as a function on $\Cc$ that is independent of $\theta \in \Ss^k$, we have $\phi \equiv 1$
on $\overline{\Dc_{r}(z_0)}$ and $\operatorname{supp}(\phi) \subset \Dc_{2r}(z_0)$.
For each $\epsilon < T/4$, let $\xi_{\epsilon}\in C^{\infty}(\RR)$  be a bump function with support in $(\epsilon, 3T/4)$
which is identically one on $[2\epsilon, T/2]$. These functions may chosen to satisfy the inequalities
\[
 r|\nabla \phi| + r^2|\Delta \phi| \leq C, \quad \epsilon|\xi_{\epsilon}^{\prime}|\chi_{[\epsilon, 2\epsilon]} + T|\xi_{\epsilon}^{\prime}|\chi_{[T/2, 3T/4]} \leq C
\]
for some $C$. (Note that $|\nabla \phi|_{\gt}(\theta, z, \tau) = |\delb \phi|_{\gb}(z)$ and $(\Delta \phi)(\theta, z, \tau) = (\bar{\Delta}\phi)(z)$.)

Define $\ve{W}_{\epsilon} = \phi\xi_{\epsilon}\ve{W}$ and $\ve{Y}_{\epsilon}= \phi\xi_{\epsilon} \ve{Y}$. Using \eqref{eq:wysys}, we compute that
\begin{align*}
 |(D_{\tau} + \Delta)W_{\epsilon}^i|^2 &\leq CB^2\tau^{2\beta}(|\ve{W}_{\epsilon}|^2 + |\ve{Y}_{\epsilon}|^2) 
 +CB^2\sum_{j=0}^{i-1}\tau^{-2\gamma}|W^j_{\epsilon}|^2\\
 &\phantom{\leq} +
 C\xi_{\epsilon}^2(|\nabla \phi|^2|\nabla W^i|^2 + |\Delta\phi|^2|W^i|^2) + C\phi^2(\xi_{\epsilon}^{\prime})^2|W^i|^2,
 \end{align*}
for each $i =0, \ldots, q$.  For each $i$, define $\nu_i =  (q-i)(\gamma + b)$ 
and apply the Carleman estimate \eqref{eq:pdecarleman2} to $W_{\epsilon}^i$ with $\alpha_i = m/2 + \nu_i$ to obtain
\begin{align}\label{eq:wiest1}
\begin{split}
&\iint \sigma^{-2\alpha_i}\tau(\alpha_i|W_{\epsilon}^i|^2 + \tau|\nabla W^i_{\epsilon}|^2)G_{z_0}\,\dX\\
&\qquad\leq N\sum_{j< i}\iint\tau^{2-2\gamma}\sigma^{-2\alpha_i}|W^j_{\epsilon}|^2G_{z_0}\,\dX\\
&\qquad\phantom{\leq}+
 N\sum_{j\geq i}\iint \tau^{2\beta+2}\sigma^{-2\alpha_i}|W_{\epsilon}^j|^2 G_{z_0}\,\dX\\
 &\qquad\phantom{\leq}+N\iint \tau^{2\beta+2}\sigma^{-2\alpha_i}|\ve{Y}_{\epsilon}|^2G_{z_0}\,\dX\\
 &\qquad\phantom{\leq}+ \frac{C}{r^2}\int_{\epsilon}^{\frac{3T}{4}}\int_{\Ac_{r, 2r}(z_0)}\tau^2\sigma^{-2\alpha_i}(|W^i|^2 + |\nabla W^i|^2)G_{z_0}\,\dX\\
 &\qquad\phantom{\leq}+\frac{C}{\epsilon^2}\int_{\epsilon}^{2\epsilon}\int_{\Dc_{2r}(z_0)}\tau^2\sigma^{-2\alpha_i}|W^i|^2\,G_{z_0}\,\dX\\
 &\qquad\phantom{\leq}+\frac{C}{T^2}\int_{\frac{T}{2}}^{\frac{3T}{4}}\int_{\Dc_{2r}(z_0)}\tau^2\sigma^{-2\alpha_i}|W^i|^2\,G_{z_0}\,\dX.
\end{split}
\end{align}

For the integrals in the first term on the right, we have immediately that
\[
\iint\tau^{2-2\gamma}\sigma^{-2\alpha_i}|W^j_{\epsilon}|^2G_{z_0}\,\dX 
\leq N\iint\tau^{2}\sigma^{-2(\alpha_j-b)}|W^j_{\epsilon}|^2G_{z_0}\, \dX,
\]
using \eqref{eq:sigmataucomp} and that $\alpha_j \geq \alpha_i + \gamma + b$ for $j < i$. For the integrals in the second term, our choice of $\beta_0$ ensures that

\begin{equation}\label{eq:ba}
 (\beta - \alpha_i) - (b - \alpha_j) \geq  (q + i - j)(b+\gamma)\geq 0, 
\end{equation}
and hence $\sigma^{2(\beta - \alpha_i)} \leq \sigma^{2(b - \alpha_j)}$, 
for all $0 \leq i \leq j \leq q$ and $\beta \geq \beta_0$. Thus	
\begin{align*}	
 \iint \tau^{2\beta+2}\sigma^{-2\alpha_i}|W_{\epsilon}^j|^2 G_{z_0}\,\dX 
 &\leq \iint \tau^{2}\sigma^{-2(\alpha_j - b)}|W_{\epsilon}^j|^2 G_{z_0}\,\dX
\end{align*}
for $i \leq j$, again using \eqref{eq:sigmataucomp}. Therefore, we may combine the first two terms to obtain
\begin{align*}
&\sum_{j< i}\iint\tau^{2-2\gamma}\sigma^{-2\alpha_i}|W^j_{\epsilon}|^2G_{z_0}\,\dX+
\sum_{j\geq i}\iint \tau^{2\beta+2}\sigma^{-2\alpha_i}|W_{\epsilon}^j|^2 G_{z_0}\,\dX\\
&\qquad\leq NT_0^{2b+1}\sum_{j= 0}^q\iint\tau\sigma^{-2\alpha_j}|W^j_{\epsilon}|^2G_{z_0}\,\dX.
\end{align*}
Equation \eqref{eq:ba} also shows that $\beta - \alpha_i \geq b -\alpha_q = b - m/2$ for all $i$, so that we can
estimate the third term in \eqref{eq:wiest1} by
\[
 \iint \tau^{2\beta+2}\sigma^{-2\alpha_i}|\ve{Y}_{\epsilon}|^2G_{z_0}\,\dX \leq \iint \tau^{2}\sigma^{2b-m}|\ve{Y}_{\epsilon}|^2G_{z_0}\,\dX.
\]

Returning to \eqref{eq:wiest1}, using that $\sigma^{-2\alpha_i} \leq \tau^{-2\nu_0}\sigma^{-m}$ in the last three terms, and summing over $i$, we obtain that
\begin{align}\label{eq:wiest2}
\begin{split}
&\sum_{i=0}^q\iint \sigma^{-2\alpha_i}\tau(\alpha_i|W_{\epsilon}^i|^2 + \tau|\nabla W^i_{\epsilon}|^2)G_{z_0}\,\dX\\
&\qquad\leq NT_0^{2b+1}\sum_{j=0}^{q}\iint\tau\sigma^{-2\alpha_j}|W^j_{\epsilon}|^2G_{z_0}\,\dX\\
&\qquad\phantom{\leq}+N\iint \tau^{2}\sigma^{2b-m}|\ve{Y}_{\epsilon}|^2G_{z_0}\,\dX\\
&\qquad\phantom{\leq} + \frac{C}{r^2}\int_{\epsilon}^{\frac{3T}{4}}\int_{\Ac_{r, 2r}(z_0)}\tau^{2-2\nu_0}\sigma^{-m}(|\ve{W}|^2 + |\nabla \ve{W}|^2)G_{z_0}\,\dX\\
&\qquad\phantom{\leq}+\frac{C}{\epsilon^2}\int_{\epsilon}^{2\epsilon}\int_{\Dc_{2r}(z_0)}\tau^{2-2\nu_0}\sigma^{-m}|\ve{W}|^2\,G_{z_0}\,\dX\\
&\qquad\phantom{\leq}+\frac{C}{T^2}\int_{\frac{T}{2}}^{\frac{3T}{4}}\int_{\Dc_{2r}(z_0)}\tau^{2-2\nu_0}\sigma^{-m}|\ve{W}|^2\,G_{z_0}\,\dX.
\end{split}
\end{align}

If $T_0$ is sufficiently small (depending on $N$ and $b$), we may bring the first term on the right side over to the left. Then, we may split the domain
of integration in the second term
to obtain that
\begin{align}
\begin{split}\label{eq:xest1}
&\sum_{i=0}^q\iint \sigma^{-2\alpha_i}(\tau|\ve{W}_{\epsilon}^i|^2 + \tau^2|\nabla \ve{W}^i_{\epsilon}|^2)G_{z_0}\,\dX\\
&\leq N\int_{2\epsilon}^{\frac{T}{2}}\int_{\Dc_{r}(z_0)} \tau^{2}\sigma^{2b-m}|\ve{Y}_{\epsilon}|^2G_{z_0}\,\dX\\
&\phantom{\leq}
+ \frac{N}{r^2}\int_{\epsilon}^{\frac{3T}{4}}\int_{\Ac_{r, 2r}(z_0)}\!\!\!\!\!\tau^{2-2\nu_0}\sigma^{-m}(|\ve{W}|^2 + |\nabla \ve{W}|^2 + |\ve{Y}|^2)G_{z_0}\,\dX\\
&\phantom{\leq}+\frac{N}{\epsilon^2}\int_{\epsilon}^{2\epsilon}\int_{\Dc_{2r}(z_0)}\!\!\!\!\tau^{2-2\nu_0}\sigma^{-m}(|\ve{W}|^2 + |\ve{Y}|^2)\,G_{z_0}\,\dX\\
&\phantom{\leq}+\frac{N}{T^2}\int_{\frac{T}{2}}^{\frac{3T}{4}}\int_{\Dc_{2r}(z_0)}\!\!\!\!\tau^{2-2\nu_0}\sigma^{-m}(|\ve{W}|^2 + |\ve{Y}|^2)\,G_{z_0}\,\dX.
\end{split}
\end{align}
  On account of our decay assumption \eqref{eq:spacetimedecay3}, we may send $\epsilon \searrow 0$ in \eqref{eq:xest1} and the third term on the right will vanish. Then, using that
  $\sigma^{-\alpha_q} \leq \sigma^{-\alpha_i}$, we have
\begin{align*}
& \iint\sigma^{-m}(\tau|\ve{W}_{\epsilon}|^2 + \tau^2|\nabla \ve{W}_{\epsilon}|^2)G_{z_0}\,\dX =\iint \sigma^{-2\alpha_q}(\tau|\ve{W}_{\epsilon}|^2 + \tau^2|\nabla \ve{W}_{\epsilon}|^2)G_{z_0}\,\dX\\
&\qquad\leq\sum_{i=0}^q\iint \sigma^{-2\alpha_i}(\tau|\ve{W}_{\epsilon}^i|^2 + \tau^2|\nabla \ve{W}^i_{\epsilon}|^2)G_{z_0}\,\dX,
\end{align*}
and so
\begin{align}
\begin{split}\label{eq:xest2}
&\int_0^{\frac{T}{2}}\int_{\Dc_{r}(z_0)}\sigma^{-m}(\tau|\ve{W}|^2 + \tau^2|\nabla \ve{W}|^2)G_{z_0}\,\dX\\
&\leq NT_0^2\int_{0}^{\frac{T}{2}}\int_{\Dc_{r}(z_0)} \tau^{\lambda_0}\sigma^{-m}|\ve{Y}|^2G_{z_0}\,\dX\\
&
+ \frac{N}{r^2}\int_{0}^{\frac{3T}{4}}\int_{\Ac_{r, 2r}(z_0)}\tau^{2-2\nu_0}\sigma^{-m}(|\ve{W}|^2 + |\nabla \ve{W}|^2 + |\ve{Y}|^2)G_{z_0}\,\dX\\
& 
+ \frac{N}{T^2}\int_{\frac{T}{2}}^{\frac{3T}{4}}\int_{\Dc_{2r}(z_0)}\tau^{2-2\nu_0}\sigma^{-m}(|\ve{W}|^2 + |\nabla \ve{W}|^2 + |\ve{Y}|^2)G_{z_0}\,\dX.
\end{split}
\end{align}
(Above, in the first term on the right, we have used our assumption that $b > \lambda_0/2$.)

Now we switch gears to estimate $\ve{Y}$. With $\xi_{\epsilon}$ defined as before, it follows from \eqref{eq:wysys} that
\[
 |D_{\tau}(\xi_{\epsilon} \ve{Y})|^2 \leq CB^2\xi_{\epsilon}^2\left(\tau^{-2\mu}(|\ve{W}|^2 + |\nabla \ve{W}|^2) + \tau^{-2}|\ve{Y}|^2\right) + C|\xi_{\epsilon}^{\prime}|^2|\ve{Y}|^2. 
\]
Since $m > m_1 \geq m_0$, we may apply the Carleman estimate \eqref{eq:odecarleman2g} to $\ve{Z} = \xi_{\epsilon}\ve{Y}$ on $\Dc_{r}(z_0)$ with $\alpha = m/2$ to obtain
that
\begin{align}
\begin{split}\label{eq:yest1}
&m\int_{0}^{\frac{3T}{4}}\int_{\Dc_r(z_0)} \tau^{\lambda_0}\sigma^{-m}\xi_{\epsilon}^2|\ve{Y}|^2G_{z_0}\,\dX \\
&\leq  \int_0^{\frac{3T}{4}}\int_{\Dc_{r}(z_0)}\tau^{\lambda_0-1}\sigma^{-m}|z- z_0|^2\xi_{\epsilon}^2|\ve{Y}|^2G_{z_0}\,\dX\\
&\phantom{\leq}
+ \frac{N}{m}\int_0^{\frac{3T}{4}}\int_{\Dc_{r}(z_0)}\tau^{\lambda_0 - 2\mu+2}\sigma^{-m}\xi_{\epsilon}^2(|\ve{W}|^2 + |\nabla \ve{W}|^2)G_{z_0}\,\dX\\
&\phantom{\leq}
+\frac{N}{m}\int_0^{\frac{3T}{4}}\int_{\Dc_{r}(z_0)}\tau^{\lambda_0}\sigma^{-m}\xi_{\epsilon}^2|\ve{Y}|^2G_{z_0}\,\dX\\
&\phantom{\leq}+\frac{C}{\epsilon^2m}\int_{\epsilon}^{2\epsilon}\int_{\Dc_{r}(z_0)}\!\!\!\!\tau^{\lambda_0+2}\sigma^{-m}|\ve{Y}|^2\,G_{z_0}\,\dX\\
&\phantom{\leq}+\frac{C}{T^2m}\int_{\frac{T}{2}}^{\frac{3T}{4}}\int_{\Dc_{r}(z_0)}\!\!\!\!\tau^{\lambda_0+2}\sigma^{-m}|\ve{Y}|^2\,G_{z_0}\,\dX.
\end{split}
\end{align}
Provided $m_1$ has been chosen large enough to satisfy that $N/m_1^2 < 1/2$ we may hide the third term on the right in the left-hand side. Then 
sending $\epsilon \searrow 0$, and using that $\lambda_0 > 2\mu$, we arrive at the inequality
\begin{align}
\begin{split}\label{eq:yest2}
&\int_{0}^{\frac{T}{2}}\int_{\Dc_{r}(z_0)} \tau^{\lambda_0}\sigma^{-m}|\ve{Y}|^2G_{z_0}\,\dX \\
&\leq  \frac{2}{m}\int_0^{\frac{3T}{4}}\int_{\Dc_{r}(z_0)}\tau^{\lambda_0-1}\sigma^{-m}|z- z_0|^2|\ve{Y}|^2G_{z_0}\,\dX\\
&\phantom{\leq}
+ \frac{N}{m^2}\int_0^{\frac{T}{2}}\int_{\Dc_{r}(z_0)}\tau^{2}\sigma^{-m}(|\ve{W}|^2 + |\nabla \ve{W}|^2)G_{z_0}\,\dX\\
&\phantom{\leq}+\frac{N}{T^2m^2}\int_{\frac{T}{2}}^{\frac{3T}{4}}\int_{\Dc_{r}(z_0)}\!\!\!\!\tau^{2}\sigma^{-m}(|\ve{W}|^2 + |\nabla \ve{W}|^2 + |\ve{Y}|^2)\,G_{z_0}\,\dX.
\end{split}
\end{align}
Here we have also absorbed part of the second term on the right of \eqref{eq:yest1} into the last term of \eqref{eq:yest2}.

Adding \eqref{eq:xest2} to \eqref{eq:yest2}, we see that if $m_1$ is taken large enough and $T_0$ small enough (depending on $N$)
we may bring some terms from the right to the left and arrive at the inequality
\begin{align}
\begin{split}\label{eq:xyest1}
&\int_0^{\frac{T}{2}}\int_{\Dc_{r}(z_0)}\sigma^{-m}(\tau|\ve{W}|^2 + \tau^2|\nabla \ve{W}|^2 + \tau^{\lambda_0}|\ve{Y}|^2)G_{z_0}\,\dX\\
&\leq \frac{N}{r^2}\int_{0}^{\frac{3T}{4}}\int_{\Ac_{r, 2r}(z_0)}\!\!\!\!\!\tau^{2-2\nu_0}\sigma^{-m}(|\ve{W}|^2 + |\nabla \ve{W}|^2 + |\ve{Y}|^2)G_{z_0}\,\dX\\
&\phantom{\leq} + \frac{N}{T^2}\int_{\frac{T}{2}}^{\frac{3T}{4}}\int_{\Dc_{2r}(z_0)}\!\!\!\!\!\tau^{2-2\nu_0}\sigma^{-m}(|\ve{W}|^2 + |\nabla \ve{W}|^2 + |\ve{Y}|^2)G_{z_0}\,\dX\\
&\phantom{\leq} +  \frac{4}{m}\int_0^{\frac{3T}{4}}\int_{\Dc_{r}(z_0)}\tau^{\lambda_0-1}\sigma^{-m}|z- z_0|^2|\ve{Y}|^2G_{z_0}\,\dX.
\end{split}
\end{align}

We now estimate each term on the right side of \eqref{eq:xyest1} in turn. For the first, note that we have $G_{z_0}(z, \tau) \leq e^{-\frac{r^2}{4\tau}}$ on
$\Ac_{r, 2r}(z_0)\times(0, 3T/4)$ and, hence
\begin{align*}
\tau^2\sigma^{-m}G_{z_0} &\leq \tau^{-m+2}e^{-\frac{r^2}{4\tau}} \leq \left(\frac{4(m-2)}{r^2}\right)^{m-2}e^{-(m-2)}\leq \left(\frac{4}{r^2}\right)^{m-2} (m-2)!
\end{align*}
by Stirling's formula.
Also, by \eqref{eq:spacetimedecay3}, we have
\[
(|\ve{W}|^2 + |\nabla \ve{W}|^2 + |\ve{Y}|^2)\tau^{-2\nu_0} \leq K
\]
on $\Cc_{r_0}\times (0, T)$, provided $m_1 \geq 2\nu_0 = 2q(\gamma + b)$. So the first term on the right side of \eqref{eq:xyest1} may be estimated from above by
\begin{align}
\begin{split}\label{eq:xyest1t1}
& \frac{1}{r^2}\int_{0}^{\frac{3T}{4}}\int_{\Ac_{r, 2r}(z_0)}\tau^{2- 2\nu_0}\sigma^{-m}(|\ve{W}|^2 + |\nabla \ve{W}|^2 + |\ve{Y}|^2)G_{z_0}\,\dX\\
 &\qquad\leq NKr^{-2m}4^{m}(m-2)!.
\end{split}
 \end{align}
 
For the second term, we simply note that
\begin{align}
\begin{split}\label{eq:xyest1t2}
&\frac{1}{T^2}\int_{\frac{T}{2}}^{\frac{3T}{4}}\int_{\Dc_{2r}(z_0)}\!\!\!\!\tau^{2-2\nu_0}\sigma^{-m}(|\ve{W}|^2 + |\nabla \ve{W}|^2 + |\ve{Y}|^2)\,G_{z_0}\,\dX \leq N2^mr^{-2m}.
\end{split}
\end{align}

The third term in \eqref{eq:xyest1} will require more work. First, we fix some $0< \delta < 1/4$ and split the domain of integration into three
spacetime regions:
\begin{align}
 \begin{split}\label{eq:ex1}
 &\frac{4}{m}\int_0^{\frac{3T}{4}}\int_{\Dc_{r}(z_0)}\tau^{\lambda_0-1}\sigma^{-m}|z- z_0|^2|\ve{Y}|^2G_{z_0}\,\dX\\
 &=\frac{4}{m}\int_0^{\frac{3T}{4}}\int_{\Ac_{4\delta r, r}(z_0)}\tau^{\lambda_0-1}\sigma^{-m}|z- z_0|^2|\ve{Y}|^2G_{z_0}\,\dX\\
 &\phantom{=}+
 \frac{4}{m}\int_{\frac{T}{2}}^{\frac{3T}{4}}\int_{\Dc_{4\delta r}(z_0)}\tau^{\lambda_0-1}\sigma^{-m}|z- z_0|^2|\ve{Y}|^2G_{z_0}\,\dX\\
 &\phantom{=}+\frac{4}{m}\int_0^{\frac{T}{2}}\int_{\Dc_{4\delta r}(z_0)}\tau^{\lambda_0-1}\sigma^{-m}|z- z_0|^2|\ve{Y}|^2G_{z_0}\,\dX.
\end{split}
 \end{align}
The first and second terms in \eqref{eq:ex1} can be estimated exactly as their counterparts in \eqref{eq:xyest1} above, to yield
\begin{align}\label{eq:ex1t2}
\begin{split}
& \frac{4}{m}\int_0^{\frac{3T}{4}}\int_{\Ac_{4\delta r, r}(z_0)}\tau^{\lambda_0-1}\sigma^{-m}|z- z_0|^2|\ve{Y}|^2G_{z_0}\,\dX
 \leq N(2\delta r)^{-2m} (m-1)!
 \end{split}
\end{align}
and
 \begin{equation}\label{eq:ex1t1}
 \frac{4}{m}\int_{\frac{T}{2}}^{\frac{3T}{4}}\int_{\Dc_{4\delta r}(z_0)}\tau^{\lambda_0-1}\sigma^{-m}|z- z_0|^2|\ve{Y}|^2G_{z_0}\,\dX
 \leq N2^{m}r^{-2m}.
\end{equation}
To estimate the third term on the right of \eqref{eq:ex1}, we will split the domain of integration further into the spacetime regions 
\begin{align*}
\Omega&=  (\Dc_{4\delta r}(z_0)\times(0, T/2)) \cap\left\{|z-z_0|^2 < \frac{m\tau}{8}\,\right\}, 
\end{align*}
and 
\begin{align*}
\Omega^{\prime} &= (\Dc_{4\delta r}(z_0)\times(0, T/2))\cap \Omega^c.
\end{align*}
Then $|z- z_0|^2G_{z_0}/\tau \leq (m/8)e^{-m/32}$ on $\Omega^{\prime}$,  provided at least that $m_1 \geq 32$,
and so 
\begin{align*}
 \begin{split}
 &\frac{4}{m}\int_0^{\frac{T}{2}}\int_{\Dc_{4\delta r}(z_0)}\tau^{\lambda_0-1}\sigma^{-m}|z- z_0|^2|\ve{Y}|^2G_{z_0}\,\dX\\
 &\qquad\leq \frac{1}{2}\iint_{\Omega}\tau^{\lambda_0}\sigma^{-m}|\ve{Y}|^2G_{z_0}\,\dX + \frac{e^{-\frac{m}{32}}}{2}\iint_{\Omega^{\prime}}\tau^{\lambda_0}\sigma^{-m}|\ve{Y}|^2\,\dX\\
 &\qquad\leq \frac{1}{2}\int_0^{\frac{T}{2}}\int_{\Dc_{r}(z_0)}\tau^{\lambda_0}\sigma^{-m}|\ve{Y}|^2G_{z_0}\,\dX + \frac{e^{-\frac{m}{32}}}{2}\int_{0}^{\frac{T}{2}}\int_{\Dc_{4r\delta}(z_0)}\tau^{\lambda_0}\sigma^{-m}|\ve{Y}|^2\,\dX.
\end{split}
\end{align*}

Putting things together, we see that the third term on the right side of \eqref{eq:xyest1} admits the bound
\begin{align}\label{eq:xyest1t3}
 \begin{split}
&\frac{4}{m}\int_0^{\frac{3T}{4}}\int_{\Dc_{r}(z_0)}\tau^{\lambda_0-1}\sigma^{-m}|z- z_0|^2|\ve{Y}|^2G_{z_0}\,\dX\\
&\qquad\leq\frac{1}{2}\int_0^{\frac{T}{2}}\int_{\Dc_{r}(z_0)}\tau^{\lambda_0}\sigma^{-m}|\ve{Y}|^2G_{z_0}\,\dX
+ \frac{e^{-\frac{m}{32}}}{2}\int_{0}^{\frac{T}{2}}\int_{\Dc_{4r\delta}(z_0)}\tau^{\lambda_0}\sigma^{-m}|\ve{Y}|^2\,\dX\\
 &\qquad\phantom{\leq} + N2^m\delta^{-2m}r^{-2m}(m-1)!
 \end{split}
\end{align}
for any $\delta \in (0, 1/4)$.
Incorporating \eqref{eq:xyest1t1}, \eqref{eq:xyest1t2}, and \eqref{eq:xyest1t3} into \eqref{eq:xyest1} then yields
\begin{align}
 \begin{split}\label{eq:xyest2}
&\int_0^{\frac{T}{2}}\int_{\Dc_{r}(z_0)}\sigma^{-m}(\tau|\ve{W}|^2 + \tau^2|\nabla\ve{W}|^2 + \tau^{\lambda_0}|\ve{Y}|^2)G_{z_0}\,\dX\\
&\qquad\leq e^{-\frac{m}{32}}\int_{0}^{\frac{T}{2}}\int_{\Dc_{4r\delta}(z_0)}\tau^{\lambda_0}\sigma^{-m}|\ve{Y}|^2\,\dX + K_04^m\delta^{-2m}r^{-2m}(m-1)!,
\end{split}
\end{align}
provided $K_0$ is sufficiently large (depending on $K$ and $N$).

We now estimate the first term on the right of \eqref{eq:xyest2}. We start by applying Lemma \ref{lem:fgtofng} with $F = \tau|\ve{W}|^2 + \tau^2|\nabla \ve{W}|^2$
and $a = \lambda_0 - 2\mu - 1 = 1 + (n-k)/2$. Choose $\delta$ so small that
\[
 0 < \delta < 1  - e^{-\frac{1}{64}}.
\]
Then, since we already have assumed that $K_0 \geq M$, if, in addition, $L_0 \geq 1/(4\delta)^2$, Lemma \ref{lem:fgtofng} and the induction hypothesis \eqref{eq:ih}
together imply that
\begin{align*}
 &\int_0^{T}\int_{D_{4\delta r}(z_0)}\sigma^{-(m-1)}\tau^{a+1}\!\!\left(|\ve{W}|^2 + \tau|\nabla \ve{W}|^2\right)\dX \leq CK_0\left(\frac{L_0}{(1-\delta)^2r^2}\right)^{m-1}(m-1)!
\end{align*}
for some $C = C_{a}$ (which, with our choice of $a$, only depends on $n$ and $k$). Then, since $m \geq m_1 \geq m_0$, provided $L_0 \geq 2$, we may apply Lemma 
\ref{lem:xmngtoympng} with $\lambda = \lambda_0$ and $a+1$ in place of $a$ to obtain that 
\begin{align*}
 \begin{split}
 e^{-\frac{m}{32}}\int_0^{T}\int_{\Dc_{4\delta r}(z_0)}\!\!\!\!\tau^{\lambda_0}\sigma^{-m}|\ve{Y}|^2\,\dX &\leq CK_0e^{-\frac{m}{32}}\left(\frac{L_0}{(1-\delta)^2r^2}\right)^{m-1}(m-2)!\\
 &\leq CK_0L_0^{m-1}r^{-2m}(m-2)!.
 \end{split}
\end{align*}

Returning to \eqref{eq:xyest2}, we see that
\begin{align*}
 \begin{split}
&\int_0^{\frac{T}{2}}\int_{\Dc_{r}(z_0)}\sigma^{-m}(\tau|\ve{W}|^2 + \tau^2|\nabla \ve{W}|^2 + \tau^{\lambda_0}|\ve{Y}|^2)G_{z_0}\,\dX\\
&\qquad\leq\frac{K_0L_0^m(m-1)!}{r^{2m}}\left(\frac{C}{L_0} + \left(\frac{4}{\delta L_0}\right)^m\right)\\
&\qquad\leq \frac{K_0L_0^m(m-1)!}{2r^{2m}}
\end{split}
\end{align*}
provided $L_0$ is taken large enough depending on $C$ and the universal constant $\delta$.
On the other hand,
\begin{align*}
 &\int_{\frac{T}{2}}^T\int_{\Dc_{r}(z_0)}\sigma^{-m}(\tau|\ve{W}|^2 + \tau^2|\nabla \ve{W}|^2 + \tau^{\lambda_0}|\ve{Y}|^2)G_{z_0}\,\dX
 \leq CM2^mr^{-2m}.
\end{align*}
Summing these two inequalities completes the proof of Proposition \ref{prop:xypolydecay} provided $K_0$ and $L_0$ are taken larger still.

\end{proof}

\section{Backward uniqueness}\label{sec:backwarduniqueness}

In this section, we will prove Theorem \ref{thm:mainp} via an analysis of the system composed of $\ve{X} = \nablat \Rmt$ and $\ve{Y}= (Y^0, Y^1, Y^2)$
from Section \ref{sec:pdeode}. Our analysis will only make use of the following properties of $\ve{X}$ and $\ve{Y}$:
\begin{enumerate}
\item There exists a constant $B$ such that 
\begin{align}\label{eq:xysys1}
\begin{split}
   |(D_{\tau} + \Delta)\ve{X}| &\leq B\tau^{-1}|\ve{X}| + B|\ve{Y}|,\\
   |D_{\tau} \ve{Y}| &\leq  B(|\ve{X}| + |\nabla\ve{X}|) + B\tau^{-1}|\ve{Y}|,
\end{split}
\end{align}
on $\Cc_{r_0}\times (0, 1]$.
\item\label{it:ss}
The sections $\ve{X}$ and $\ve{Y}$ are self-similar in the sense that, if $\overline{\ve{X}} = \ve{X}|_{\Cc_{r_0}\times\{1\}}$
and $\overline{\ve{Y}} = \ve{Y}|_{\Cc_{r_0}\times\{1\}}$,  and $\Psi_{\tau}(\theta, z) = (\theta, z/\sqrt{\tau})$, then
\[
 \ve{X}= \tau\Psi_{\tau}^*\overline{\ve{X}},\quad \ve{Y} = \tau\Psi_{\tau}^*\overline{\ve{Y}},
\] 
and
\begin{equation}\label{eq:buss}
 |\ve{X}|^2 = \tau^{-3}|\overline{\ve{X}}|_{g(1)}^2\circ\Psi_{\tau}, \quad |\ve{Y}|^2 = \sum_{i=0}^{2}\tau^{-i}|\overline{Y}^i|^2_{g(1)}\circ\Psi_{\tau}.
\end{equation}
\item There is a constant $M_0$ such that
\begin{equation}\label{eq:unifbd}
    \sup_{\Cc_{r_0}\times(0, 1]}\left(|\ve{X}|^2 + |\nabla\ve{X}|^2 + |\ve{Y}|^2\right) \leq M_0.
\end{equation}

\item There are constants $N_2$, $N_3 > 0$ and $r_1 \geq r_0$ such that 
\begin{equation}\label{eq:decay}
\int_0^1\int_{\Ac_{r, 2r}}\left(|\ve{X}|^2 + |\nabla \ve{X}|^2 + |\ve{Y}|^2\right)e^{\frac{N_2r^2}{\tau}}\,\dX \leq N_3
\end{equation}
for all $r \geq r_1$.
\end{enumerate}

The exact values of the exponents of $\tau$ in the scale factors in \eqref{it:ss} are not important for the analysis; all that we need is for $\ve{X}$ and $\ve{Y}$ to be self-similar and satisfy some relationship
akin to \eqref{eq:buss}.
We will show that these four conditions imply that $\ve{X}$ and $\ve{Y}$ must vanish identically on $\Cc_{r_2}\times (0, T_1]$ for some $r_2 \geq r_1$ and $0 < T_1 \leq 1$.
\begin{theorem}\label{thm:bu}
Suppose that $\ve{X}$ and $\ve{Y}$ are smooth sections of $\Xc$ and $\Yc$ defined on $\Cc_{r_0}\times (0, 1]$ satisfying conditions (1) - (4) above.  Then there exists
$r_2 > 0$ and $0 < T_1\leq 1$ such that $\ve{X}\equiv 0$ and  $\ve{Y}\equiv 0$ on $\Cc_{r_2}\times (0, T_1]$. 
\end{theorem}
We have already seen in Proposition \ref{prop:xypdeode} that $\ve{X}$ and $\ve{Y}$ defined by \eqref{eq:xydef} satisfy (1) - (3). The following proposition, which is essentially a corollary of Theorem \ref{thm:expdecay1},
shows that they also satisfy the exponential decay estimate in the precise form given in (4). Theorem \ref{thm:mainp} is thus a consequence of Theorem \ref{thm:bu}.

\subsection{Space-time exponential decay revisited}
Combined with the self-similarity of $\ve{X}$ and $\ve{Y}$ and the reference metric $g$, Theorem \ref{thm:expdecay1} implies that $\ve{X}$ and $\ve{Y}$ also decay in space at an exponential-quadratic rate.
\begin{proposition}\label{prop:expdecay2}
 There exist $N_2$ and $N_3$ (depending on $N_0$, $N_1$, and $r_0$) such that
 \begin{equation}\label{eq:expdecay2}
    \int^1_0\int_{\Ac_{r, 2r}}\left(|\ve{X}|^2 + |\nabla\ve{X}|^2 + |\ve{Y}|^2\right)e^{\frac{N_2r^2}{\tau}}\,\dX \leq N_3 
 \end{equation}
 or any $r \geq 16r_0$.
\end{proposition}
\begin{proof}
For simplicity, let $r_1 = 16r_0$.  The set $\Ac_{r_1, 2r_1}$ can be covered by a finite collection of sets of the form $\Dc_1(z_i)$ where $z_i \in \RR^{n-k}\setminus\overline{B_{r_1/2}(0)}$
and so we obtain from Theorem \ref{thm:expdecay1} the inequality
\begin{equation}\label{eq:expdecay2a}
  \int^1_0\int_{\Ac_{r_1, 2r_1}}\left(|\ve{X}|^2 + |\nabla\ve{X}|^2 + |\ve{Y}|^2\right)e^{\frac{N_0}{\tau}}\,\dX \leq CN_1r_0^{n-k}
\end{equation}
for some $C = C(n, k)$.

Now fix $r \geq r_1$.
Then
\[
 |\ve{X}|^2(\theta, z, \tau) = \tau^{-3}|\delt \Rmt|_{g(1)}^2(\Psi_{\tau}(\theta, z), 1),\quad  d\mu_{g(\tau)} = \tau^{k/2}d\mu_{g(1)},
\]
and so, for any $0 < a < 1$, by the change of variables
\[
 \theta^{\prime} = \theta, \quad z^{\prime} = \frac{r_1}{r}z, \quad \tau^{\prime} = \frac{r_1^2}{r^2}\tau,
\]
we have
\begin{align*}
  \int^1_a\int_{\Ac_{r, 2r}}|\ve{X}|^2e^{\frac{N_0r^2}{r_1^2\tau}}\,\dX 
  &= \left(\frac{r}{r_1}\right)^{n-4}\int^{\frac{r_1^2}{r^2}}_{\frac{ar_1^2}{r^2}}\int_{\Ac_{r_1, 2r_1}}|\ve{X}|^2e^{\frac{N_0}{\tau}}\,\dX.
\end{align*}
  Taking $N_2 = N_0/(2r_1^2) = N_0/(512r_0^2)$, then, and sending $a\to 0$, we obtain
 \begin{align*}
  \int^1_0\int_{\Ac_{r, 2r}}|\ve{X}|^2e^{\frac{N_2r^2}{\tau}}\,\dX &\leq e^{\frac{-N_0r^2}{2r_1^2}} \left(\frac{r}{r_1}\right)^{n-4}\int^{\frac{r_1^2}{r^2}}_{0}\int_{\Ac_{r_1, 2r_1}}|\ve{X}|^2e^{\frac{N_0}{\tau}}\,\dX\\
  &\leq N\int^{1}_{0}\int_{\Ac_{r_1, 2r_1}}|\ve{X}|^2e^{\frac{N_0}{\tau}}\,\dX
\end{align*}
for some $N= N(N_0)$. The estimate \eqref{eq:expdecay2} for $\ve{X}$ then follows from \eqref{eq:expdecay2a}. 
Analogous scaling arguments prove \eqref{eq:expdecay2} for the other terms in the integrand.
\end{proof}

\subsection{Carleman estimates}
To prove Theorem \ref{thm:bu}, we will use two Carleman-type inequalities with weights that grow at an approximately exponential-quadratic rate at infinity. Following \cite{WangCylindrical},
for $\alpha > 0$, $0 < T\leq 1$, and $\delta\in (7/8, 1)$, we define $\phi_{\alpha}:\Cc\times (0, \infty)\to \RR$ by
\begin{equation}\label{eq:phialphadef}
  \phi_{\alpha}(\theta, z, \tau) = 
  \alpha\eta(\tau)\left(\frac{|z|^2}{\tau}\right)^{\delta},
\end{equation}
and $\eta:[0, T]\to[0, 1]$ by
\begin{equation}\label{eq:etadef}
\eta(\tau)=  \left\{\begin{array}{ll}
	    1 & \mbox{if}\ \tau \in [0, \tau_0],\\
	    1-\frac{1}{32}\delta(4\delta -3)\left(\frac{\tau}{\tau_0} - 1\right)^2 &\mbox{if}\ \tau\in [\tau_0, 2\tau_0],\\
	    1 + \frac{1}{32}\delta(4\delta-3)\left(3-\frac{2\tau}{\tau_0}\right) &\mbox{if}\ \tau\in [2\tau_0, T],
         \end{array}\right.
\end{equation}
where
\[
  \tau_0 \dfn \frac{2\delta(4\delta - 3)T}{3\delta(4\delta -3) + 32}.
\]
The function $\eta$ has been engineered to be monotone decreasing on $[0, T]$, identically one near $\tau = 0$,  and proportional to $T-\tau$
near $\tau = T$ with $\eta(T) = 0$. 

Below, $\Zc$ will denote an arbitrary bundle of the form $\bigoplus T^{(k_i, l_i)}\Cc$ 
equipped with the family of metrics and connections induced by $g(\tau)$. 
\begin{theorem}\label{thm:pdecarlemanbu} For any $\delta \in (7/8, 1)$ and $T\leq 1$, there exists $r_3\geq 1$ depending on $n$, $k$, and $\delta$ 
such that, for all smooth families of sections $\ve{Z}$ of the bundle $\Zc$ 
with support compactly contained in $\Cc_{r_3}\times (0, T)$, we have the inequality
\begin{align}\begin{split}
\label{eq:pdecarlemanbu}
& \int_0^T\int_{\Cc_{r}}\left(\frac{\alpha}{\tau^{\delta}}|\ve{Z}|^2 + \tau|\nabla\ve{Z}|^2\right)e^{2\phi_{\alpha}}\,\dX \leq 10\int_0^T\int_{\Cc_{r}}\tau^2|(D_{\tau} + \Delta)\ve{Z}|^2e^{2\phi_{\alpha}}\,\dX
\end{split}
  \end{align}
for all $\alpha > 0$ and $r\geq r_3$.
\end{theorem}
We will apply this estimate to the PDE component $\ve{X}$ of our system. To control the ODE component $\ve{Y}$, we will use the following matching estimate.
\begin{theorem}\label{thm:odecarlemanbu}
For any $\delta \in (7/8, 1)$, and $T\leq 1$ there exists $r_4 > 0$, depending on $n$, $k$, and $\delta$,
such that, for all smooth families of sections $\ve{Z}$ of  $\Zc$ with support compactly contained in  
$\Cc_{r}\times (0, T)$, we have the inequality
\begin{equation}
\label{eq:odecarlemanbu}
 \int_0^T\int_{\Cc_{r}}\frac{\alpha}{\tau^{\delta}}|\ve{Z}|^2e^{2\phi_{\alpha}}\,\dX \leq \int_0^T\int_{\Cc_{r}}\tau^2|D_{\tau}\ve{Z}|^2e^{2\phi_{\alpha}}\,\dX 
\end{equation}
for all $\alpha \geq 1$ and $r \geq r_4$.
\end{theorem}
We will prove Theorems \ref{thm:pdecarlemanbu} and \ref{thm:odecarlemanbu} in Section \ref{sec:carleman}. For now, we will take them for granted
and use them to prove Theorem \ref{thm:bu}.

\begin{proof}[Proof of Theorem \ref{thm:bu}]  Our argument is a modification of that of Theorem 3.3 in \cite{WangCylindrical}. 
Let $r_2 \geq \max\{r_1, r_3, r_4\}$ and fix some $R\geq r_2$ and $0 < T\leq 1$. 

We will need two cutoff functions. For all $\alpha > 8$ and $0 < \epsilon < T/8$, let $\chi_{\alpha, \epsilon}$
be a smooth bump function on $[0, 1]$ with support in $(\epsilon, T-T/\alpha)$ satisfying $\chi_{\alpha, \epsilon} \equiv 1$ on $[2\epsilon, T-2T/\alpha]$,
$|\chi_{\alpha, \epsilon}^{\prime}| \leq 2/\epsilon$ on $(\epsilon, 2\epsilon)$, and $|\chi_{\alpha, \epsilon}^{\prime}| \leq 2\alpha/T$ on $(T-2T/\alpha, T-T/\alpha)$.
For the spatial cutoff, choose, for each $r > R+1$, a bump function $\psi_r$ on $\RR^{n-k}$ with support in $B_{2r}(0)\setminus\overline{B_R(0)}$ which satisfies $\psi_r \equiv 1$ on $\overline{B_{r}(0)}\setminus B_{R+1}(0)$
and the bounds 
$|\bar{\nabla} \psi_r|_{\bar{g}} +|\bar{\Delta} \psi_r|_{\bar{g}}\leq C$. We regard $\psi_r = \psi_r(\theta, z)$ as a function on $\Cc$ which is independent of $\theta$, in which case, $|\nabla \psi_r| = |\delb\psi_r|_{\gb}$ and $\Delta \psi_r = \bar{\Delta} \psi_r$.

Now define 
\[\ve{X}_{\alpha, \epsilon, r} = \chi_{\alpha, \epsilon}\psi_r\ve{X}, \quad \ve{Y}_{\alpha, \epsilon, r} = \chi_{\alpha, \epsilon} \psi_r \ve{Y}.
\]
 From \eqref{eq:xysys1}, we have
\begin{align*}
\begin{split}
 |(D_{\tau} + \Delta)\ve{X}_{\alpha, \epsilon, r}| &\leq B\tau^{-1}|\ve{X}_{\alpha, \epsilon, r}| + B|\ve{Y}_{\alpha, \epsilon, r}| + \psi_r|\chi_{\alpha, \epsilon}^{\prime}||\ve{X}| \\
 &\phantom{\leq} + 2\chi_{\alpha, \epsilon}(|\nabla \psi_r| + |\Delta \psi_r|))(|\ve{X}| + |\nabla \ve{X}|),
 \end{split}
\end{align*}
and
\begin{align*}
\begin{split}
|D_{\tau} \ve{Y}_{\alpha, \epsilon, r}| &\leq  B(|\ve{X}_{\alpha, \epsilon, r}| + |\nabla\ve{X}_{\alpha, \epsilon, r}|) + B\tau^{-1}|\ve{Y}_{\alpha, \epsilon, r}|
+ \psi_r|\chi_{\alpha, \epsilon}^{\prime}||\ve{Y}|\\
&\phantom{\leq} +B\chi_{\alpha, \epsilon}|\nabla \psi_r||\ve{X}|,
\end{split}
\end{align*}
on $\Cc_{R}\times (0, T]$.

Applying the inequalities \eqref{eq:pdecarlemanbu} and \eqref{eq:odecarlemanbu} to $\ve{X}_{\alpha, \epsilon, r}$ and $\ve{Y}_{\alpha, \epsilon, r}$ and summing the result, we arrive at the inequality
\begin{align}\label{eq:buineq1}
 \begin{split}
 &\int_0^T\int_{\Cc_{R}}\left(\alpha\tau^{-\delta}(|\ve{X}_{\alpha, \epsilon, r}|^2 + |\ve{Y}_{\alpha, \epsilon, r}|^2) +  \tau|\nabla \ve{X}_{\alpha, \epsilon, r}|^2 \right)e^{2\phi_{\alpha}}\,\dX\\
 &\qquad\leq K\int_0^T\int_{\Cc_{R}}\left(|\ve{X}_{\alpha, \epsilon, r}|^2 + |\ve{Y}_{\alpha, \epsilon, r}|^2 +  \tau^2|\nabla \ve{X}_{\alpha, \epsilon, r}|^2 \right)e^{2\phi_{\alpha}}\,\dX\\
 &\qquad\phantom{\leq}+ \frac{C}{\epsilon^2}\int_{\epsilon}^{2\epsilon}\int_{\Ac_{R, 2r}}\tau^2\left(|\ve{X}|^2 + |\ve{Y}|^2 + |\nabla \ve{X}|^2 \right)e^{2\phi_{\alpha}}\,\dX\\
 &\qquad\phantom{\leq}+ \frac{C\alpha^2}{T^2}\int_{T-\frac{2T}{\alpha}}^{T - \frac{T}{\alpha}}\int_{\Ac_{R, 2r}}\tau^2\left(|\ve{X}|^2 + |\ve{Y}|^2 + |\nabla \ve{X}|^2 \right)e^{2\phi_{\alpha}}\,\dX\\
 &\qquad\phantom{\leq}+ K\int_{\epsilon}^{T-\frac{T}{\alpha}}\int_{\Ac_{R, R+1}}\tau^2\left(|\ve{X}|^2  + |\nabla \ve{X}|^2\right)e^{2\phi_{\alpha}}\,\dX\\
 &\qquad\phantom{\leq}+ K\int_{\epsilon}^{T-\frac{T}{\alpha}}\int_{\Ac_{r, 2r}}\tau^2\left(|\ve{X}|^2  + |\nabla \ve{X}|^2\right)e^{2\phi_{\alpha}}\,\dX.
 \end{split}
\end{align}
Here and below, we use $C$ to denote a constant depending at most on $n$ and $k$, and $K$ a constant depending possibly in addition on $\delta$, $B$, $M_0$, $N_2$, and $N_3$.

Now, provided $T$ is chosen small enough (depending on $n$, $k$, $B$ and $\delta$), we can hide the first term on the right in the term on the left at the expense of enlarging the constants on the right, say,
by a factor of two. Also, using the decay estimate \eqref{eq:decay},
we can estimate the second term on the right via
\begin{align*}
&\frac{1}{\epsilon^2}\int_{\epsilon}^{2\epsilon}\int_{\Ac_{R, 2r}}\tau^2\left(|\ve{X}|^2 + |\ve{Y}|^2 + |\nabla \ve{X}|^2 \right)e^{2\phi_{\alpha}}\,\dX\\
&\qquad\leq 4e^{\left(2\alpha\left(\frac{4r^2}{\epsilon}\right)^{\delta} - \frac{N_2R^2}{2\epsilon}\right)}
\int_{\epsilon}^{2\epsilon}\int_{\Ac_{R, 2r}}\left(|\ve{X}|^2 + |\ve{Y}|^2 + |\nabla \ve{X}|^2 \right)e^{\frac{N_2R^2}{\tau}}\,\dX\\
&\qquad \leq K_{\alpha, r}e^{-\frac{N_2R^2}{4\epsilon}}
\end{align*}
for some $K_{\alpha, r}$ depending on $\alpha$, $\delta$, $r$, $R$, $N_2$, and $N_3$.
In particular, this term tends to $0$ as $\epsilon \searrow 0$ for any fixed $\alpha$ and $r$.  

Similarly, on $\Ac_{R, R+1}\times (0, T)$ and $\Ac_{r, 2r}\times (0, T)$, we have $e^{2\phi_{\alpha}} \leq K_{\alpha}e^{\frac{N_2R^2}{\tau}}$ and $e^{2\phi_{\alpha}} \leq K_{\alpha}e^{\frac{N_2r^2}{\tau}}$,
respectively, for some $K_{\alpha}$ depending on $\alpha$ and $\delta$. So, using \eqref{eq:decay}, we see that the fourth and fifth terms on the right in \eqref{eq:buineq1} converge to finite values as $\epsilon\searrow 0$. 
After taking this limit, then, we obtain from \eqref{eq:buineq1} that
\begin{align}\label{eq:buineq2}
 \begin{split}
 &\int_0^{\frac{T}{2}}\int_{\Ac_{R+1, r}}\left(\alpha\tau^{-\delta}(|\ve{X}|^2 + |\ve{Y}|^2) +  \tau|\nabla \ve{X}|^2 \right)e^{2\phi_{\alpha}}\,\dX\\
 &\qquad\leq \frac{C\alpha^2}{T^2}\int_{T-\frac{2T}{\alpha}}^{T - \frac{T}{\alpha}}\int_{\Ac_{R, 2r}}\tau^2\left(|\ve{X}|^2 + |\ve{Y}|^2 + |\nabla \ve{X}|^2 \right)e^{2\phi_{\alpha}}\,\dX\\
 &\qquad\phantom{\leq}+ K\int_{0}^{T}\int_{\Ac_{R, R+1}}\tau^2\left(|\ve{X}|^2  + |\nabla \ve{X}|^2\right)e^{2\phi_{\alpha}}\,\dX\\
  &\qquad\phantom{\leq} + K\int_{0}^{T}\int_{\Ac_{r, 2r}}\tau^2\left(|\ve{X}|^2  + |\nabla \ve{X}|^2\right)e^{2\phi_{\alpha}}\,\dX.
 \end{split}
\end{align}

Estimating as above, we see also that
\begin{align*}
 \int_{0}^{T}\int_{\Ac_{r, 2r}}\tau^2\left(|\ve{X}|^2  + |\nabla \ve{X}|^2\right)e^{2\phi_{\alpha}}\,\dX &\leq K_{\alpha} e^{-\frac{N_2r^2}{2T}},
\end{align*}
so the last term on the right of \eqref{eq:buineq2} tends to zero as $r\to\infty$.
The first term on the right of \eqref{eq:buineq2} can also be seen to be bounded above independently of $r$; we will verify this now and further show that it is bounded independently of $\alpha$.  

Let us assume from now on that $\alpha \geq \alpha_1$ where $\alpha_1 = \alpha_1(\delta)$ is large enough that $T- 2T/\alpha_1 \geq 2\tau_0$. (The constant $\tau_0$ here is from the definition
of $\eta$ in \eqref{eq:etadef}.)
Then $\eta(\tau) = c_0(T-\tau)/T$ on the interval $[T-2T/\alpha, T-T/\alpha]$ for some constant $c_0=c_0(\delta)$ and, consequently, $\phi_{\alpha} \leq 2c_0|z|^{2\delta}/\tau^{\delta}$ for $\tau$ in the same range. 
Choosing $m$ so large that $2^mR \geq r$, we may estimate that
\begin{align*}
&\int_{T-\frac{2T}{\alpha}}^{T - \frac{T}{\alpha}}\int_{\Ac_{R, 2r}}\tau^2\left(|\ve{X}|^2 + |\ve{Y}|^2 + |\nabla \ve{X}|^2 \right)e^{2\phi_{\alpha}}\,\dX \\
&\leq \int_{T-\frac{2T}{\alpha}}^{T - \frac{T}{\alpha}}\int_{\Ac_{R, 2r}}\left(|\ve{X}|^2 + |\ve{Y}|^2 + |\nabla\ve{X}|^2 \right)e^{\frac{4c_0|z|^{2\delta}}{\tau^{\delta}}}\,\dX\\
&\leq K\int_{2\tau_0}^{T}\int_{\Ac_{R, 2^{m+1}R}}\left(|\ve{X}|^2 + |\ve{Y}|^2 + |\nabla \ve{X}|^2 \right)e^{\frac{N_2|z|^2}{8\tau}}\,\dX\\
&\leq K\sum_{l=0}^{\infty} \left\{  e^{-\frac{N_2(2^lR)^2}{8T}}\int_{0}^{T}\int_{\Ac_{2^lR, 2^{l+1}R}}\!\!\!\!\left(|\ve{X}|^2 + |\ve{Y}|^2 + |\nabla \ve{X}|^2 \right)e^{\frac{N_2(2^lR)^2}{\tau}}\,\dX\right\}\\
&\leq K,
\end{align*}
for any $\alpha \geq \alpha_1$ and $r\geq r_2$.
Thus we may take the limit as $r\to \infty$ on both sides of \eqref{eq:buineq2} to obtain that
\begin{align}\label{eq:buineq3}
 \begin{split}
 &\int_0^{\frac{T}{2}}\int_{\Cc_{R+1}}\left(\alpha\tau^{-\delta}(|\ve{X}|^2 + |\ve{Y}|^2) +  \tau|\nabla \ve{X}|^2 \right)e^{2\phi_{\alpha}}\,\dX\\
 &\qquad\leq \frac{K\alpha^2}{T^2} + K\int_{0}^{T}\int_{\Ac_{R, R+1}}\tau^2\left(|\ve{X}|^2  + |\nabla \ve{X}|^2\right)e^{2\phi_{\alpha}}\,\dX.
 \end{split}
\end{align}

To estimate the second term on the right side of \eqref{eq:buineq3}, note that, by construction, $\eta(\tau) \equiv 1$ for $\tau \in (0, \tau_0]$. Using the self-similarity of $\ve{X}$ and $\nabla\ve{X}$ from (4) above,
we
have
\begin{align*}
 \int_0^{\frac{\tau_0}{4}}\int_{\Ac_{R, R+1}}\tau^2\left(|\ve{X}|^2  + |\nabla \ve{X}|^2\right)e^{2\phi_{\alpha}}\,\dX
 &\leq \int_0^{\frac{\tau_0}{4}}\int_{\Cc_R}\tau^2\left(|\ve{X}|^2  + |\nabla \ve{X}|^2\right)e^{2\phi_{\alpha}}\,\dX\\
 &\leq C\int_0^{\tau_0}\int_{\Cc_{2R}} \tau^2\left(|\ve{X}|^2  + |\nabla \ve{X}|^2\right)e^{2\phi_{\alpha}}\,\dX\\
 &\leq C\int_0^{\frac{T}{2}}\int_{\Cc_{R+1}} \tau^2\left(|\ve{X}|^2  + |\nabla \ve{X}|^2\right)e^{2\phi_{\alpha}}\,\dX.
\end{align*}
Thus, for $T$ small enough, depending on $n$, $k$, $B$, and $L$, we can convert \eqref{eq:buineq3} into
\begin{align*}
 \int_0^{\frac{T}{2}}\int_{\Cc_{R+1}}\left(|\ve{X}|^2 + |\ve{Y}|^2\right)e^{2\phi_{\alpha}}\,\dX 
 &\leq \frac{K\alpha^2}{T^2} + 
 K\int_{\frac{\tau_0}{4}}^{T}\int_{\Ac_{R, R+1}}\!\!\!\!\!\!\tau^2\left(|\ve{X}|^2  + |\nabla \ve{X}|^2\right)e^{2\phi_{\alpha}}\,\dX\\
 &\leq K\left(\frac{\alpha^2}{T^2}   + e^{2\alpha\left(\frac{4^{\delta}(R+1)^{2\delta}}{\tau_0^{\delta}}\right)}R^{\frac{n-k}{2}}\right).
\end{align*}

On the other hand,
\begin{align*}
 \begin{split}
 &e^{2\alpha\frac{8^{\delta}(R+1)^{2\delta}}{\tau_0^{\delta}}}\int_0^{\frac{\tau_0}{8}}\int_{\Cc_{R+1}}\left(|\ve{X}|^2 + |\ve{Y}|^2\right)\,\dX \leq \int_0^{\frac{T}{2}}\int_{\Cc_{R+1}}\left(|\ve{X}|^2 + |\ve{Y}|^2\right)e^{2\phi_{\alpha}}\,\dX,
\end{split}
\end{align*}
so we find that
\begin{align*}
\int_0^{\frac{\tau_0}{8}}\int_{\Cc_{R+1}}\left(|\ve{X}|^2 + |\ve{Y}|^2\right)\,\dX \leq \frac{K\alpha^2}{T^2}e^{-\frac{\alpha \epsilon(R+1)^{2\delta}}{\tau_0^{\delta}}},
\end{align*}
for all $\alpha \geq \alpha_1$, where $\epsilon = 2\cdot4^{\delta}(2^{\delta} - 1)$. Sending $\alpha\to \infty$, we conclude at last that
$\ve{X} \equiv 0$ and $\ve{Y}\equiv 0$ on $\Cc_{R+1}\times [0, \tau_0/8]$.
\end{proof}

\section{The Carleman estimates}\label{sec:carleman}
In this section, we will prove the Carleman estimates in Theorems \ref{thm:pdecarleman2}, \ref{thm:odecarleman2}, \ref{thm:pdecarlemanbu}, and \ref{thm:odecarlemanbu}.  We start by establishing
some general integral identities for families of tensors on a manifold evolving by the backward Ricci flow.
\subsection{Integral identities}
In this subsection, we will use $g = g(\tau)$ to denote an arbitrary solution to \eqref{eq:brf} on a smooth manifold $M = M^n$ for $\tau \in (0, T)$, and $\Zc$ to denote a tensor bundle over $M$. 
We will use $\nabla= \nabla_{g(\tau)}$ and $d\mu = d\mu_{g(\tau)}$ to represent the Levi-Civita connection and Riemannian density associated to $g$,
and define the operator $D_{\tau}$ in terms of $g$ as in Section \ref{sec:pdeode}. We will also continue to use the shorthand $\dX = d\mu_{g(\tau)}\,d\tau$.

Let $\phi:M\times(0, T)\to \RR$ be a smooth positive function and consider the operator
\[
    \mathcal{L} = \tau e^{\phi}(D_{\tau} + \Delta)e^{-\phi}
\]
acting on smooth families of sections of $\Zc$. Explicitly, then, we have
\[
 \mathcal{L}V = \tau\left(|\nabla \phi|^2 - \pd{\phi}{\tau} - \Delta \phi\right)V + \tau(D_\tau + \Delta)V - 2\tau\nabla_{\nabla\phi}V, 
\]
and the formal $L^2(\dX)$-adjoint of $\mathcal{L}$ is given by
\[
 \mathcal{L}^*V = \tau\left(|\nabla\phi|^2 +\Delta \phi - \pd{\phi}{\tau} - \frac{1}{\tau} - R\right)V - \tau(D_\tau - \Delta)V + 2\tau\nabla_{\nabla\phi}V.
\]
Writing $\Lc$ in terms of its symmetric and antisymmetric parts
\begin{align*}
  \Sc V &= \frac{\Lc V + \Lc^*V}{2} = \tau\left(|\nabla \phi|^2 - \pd{\phi}{\tau} - \frac{R}{2} - \frac{1}{2\tau}\right)V + \tau \Delta V,\\
  \Ac V &= \frac{\Lc V - \Lc^*V}{2} = \tau\left(\frac{R}{2} - \Delta \phi +\frac{1}{2\tau}\right)V + \tau D_{\tau} V - 2\tau\nabla_{\nabla\phi}V, 
\end{align*}
yields the identity
\begin{align}
 \begin{split}\label{eq:basicdecomposition}
    &\iint \tau^2|D_\tau Z + \Delta Z|^2e^{2\phi}\, \dX = \iint|\Lc V|^2\,\dX\\
    &\qquad\qquad= \iint\left(|\Sc V|^2 + |\Ac V|^2 + \langle [\Sc, \Ac] V, V\rangle \right)\,\dX, 
 \end{split}
\end{align}
for any smooth family $Z = e^{-\phi}V$ of sections of $\Zc$ with compact support in $\Cc\times (0, T)$.

Provided (with a judicious choice of $\phi$) we can effectively estimate the commutator $[\Sc, \Ac]$, the above identity will yield an estimate of the $L^2$-norm of $(D_{\tau} + \Delta)Z$
from below by that of $Z$. The basis of this estimate is the following explicit expression for the commutator.
\begin{proposition}\label{prop:commutator} If $V$ is any smooth family of sections of $\Zc$ with compact support in $M\times(0, T)$, we have
\begin{align}\label{eq:commutator}
\begin{split}
&\iint \langle [\Sc, \Ac] V, V\rangle\,\dX =\iint \left(\Qc^{(1)}_{\phi}(\nabla V, \nabla V) + \Qc^{(2)}_{\phi}|V|^2 + \Qc_{\phi}^{(3)}(\nabla V, V)\right)\,\dX, 
\end{split}
\end{align}
where
 \begin{align*}
  &\Qc^{(1)}_\phi(\nabla V, \nabla V) = 2\tau^2\left(2\nabla_i\nabla_j\phi - R_{ij} + \frac{g_{ij}}{2\tau}\right)\langle \nabla_i V, \nabla_j V\rangle,\\
\begin{split}
   &\Qc^{(2)}_{\phi} = \tau^2\left(\frac{\partial^2\phi}{\partial\tau^2} - \Delta^2\phi -2\pdtau|\nabla \phi|^2 + \frac{1}{2}\left(\pd{R}{\tau} + \Delta R\right) 
   - \langle \nabla R, \nabla \phi\rangle
   \right)\\
   &\phantom{\Qc^{(2)}_{\phi} =}+ 2\tau^2\left(2\nabla\nabla\phi(\nabla\phi, \nabla\phi) -\Rc(\nabla\phi, \nabla\phi) + \frac{|\nabla\phi|^2}{2\tau}\right)\\
   &\phantom{\Qc^{(2)}_{\phi} =}
   + \tau\left(\pd{\phi}{\tau} -2|\nabla \phi|^2 + \frac{R}{2}\right),
   \end{split}
\end{align*}
and
\begin{align*}
  \begin{split}
     &\Qc^{(3)}_{\phi}(\nabla V, V) = -2\tau^2\left(\nabla_i R_{ja} - \nabla_{j}R_{ia} + 2 R_{ija}^l\nabla_l \phi\right)\langle\Lambda^i_jV,  \nabla_a V\rangle.
    \end{split}
 \end{align*}
\end{proposition}

\begin{proof}
 For the time-being, write $\Sc$ and $\Ac$ as
 \[
   \Sc =  \tau(\Delta +  F\operatorname{Id}), \quad \Ac = \tau (D_{\tau}  - 2\nabla_{\nabla_\phi} + G\operatorname{Id}).
\]
Then 
\begin{align*}
 \Sc(\Ac V) &= \tau^2\big(\Delta D_{\tau}V - 2\Delta(\nabla_{\nabla_{\phi}}V) +\Delta(G V) + F D_{\tau}V - 2F(\nabla_{\nabla_\phi} V) + FG V\big),
\end{align*}
and
\begin{align*}
\begin{split}
 \Ac(\Sc V) &= \tau^2\big(D_{\tau}\Delta V + D_{\tau}(F V) -2\nabla_{\nabla\phi}(\Delta V) - 2\nabla_{\nabla_{\phi}}(F V) 
 + G \Delta V + FG V\big)\\
 &\phantom{=\tau^2\big(} + \tau(\Delta V + F V),
 \end{split}
\end{align*}
so
\begin{align*}
 \begin{split}
   [\Sc, \Ac] V &=  \tau^2\bigg([\Delta, D_{\tau}] V + 2[\nabla_{\nabla\phi}, \Delta] V  + \left(2\left\langle \nabla F, \nabla \phi\right\rangle - \pd{F}{\tau}\right)V
      \\
     &\phantom{=}+ \Delta G V + 2\nabla_{\nabla G} V- \frac{1}{\tau}(\Delta V + F V)\bigg).
 \end{split}
\end{align*}

Since $V$ has compact support, we may integrate $\langle [\Sc, \Ac]V, V\rangle$ over $\Cc\times (0, T)$
and integrate by parts in the integrals corresponding to the fourth and sixth terms of the above identity to
obtain that
\begin{align}
\begin{split}\label{eq:comm1}
  &\iint\langle [\Sc, \Ac] V, V\rangle\,\dX\\ 
  &\qquad=  \iint \tau^2\langle [\Delta, D_{\tau}] V + 2[\nabla_{\nabla\phi}, \Delta] V, V\rangle \,\dX 
   +\iint\tau|\nabla V|^2\,\dX\\
   &\qquad\phantom{=}+ \iint\left(\tau^2\left(2\langle \nabla F, \nabla \phi\rangle - \pd{F}{\tau}\right) -\tau F\right)|V|^2\,\dX.
\end{split}
\end{align}

We now simplify the commutator terms on the right side of \eqref{eq:comm1}.  First,
\begin{align*}
\iint \tau^2\langle [\Delta, D_{\tau}] V, V\rangle\,\dX &= \iint \tau^2\bigg(\left\langle [\nabla_a, D_{\tau}] \nabla_a V, V\right\rangle + \left\langle [D_{\tau}, \nabla_a] V, \nabla_aV\right\rangle \bigg)\dX \\
 &= \iint \tau^2\left(\frac{1}{2}[\nabla_a, D_{\tau}] \nabla_a |V|^2 + 2\left\langle [D_{\tau}, \nabla_a] V, \nabla_aV\right\rangle \right)\dX, 
\end{align*}
and since
\begin{align*}
 [\nabla_a, D_{\tau}] \nabla_a|V|^2 &= R_{ab}\nabla_b\nabla_a |V|^2 +(\nabla_aR_{ac} - \nabla_cR_{aa})\nabla_c|V|^2\\
				    &= \nabla_b(R_{ab}\nabla_a |V|^2) -\langle\nabla R, \nabla|V|^2\rangle,
\end{align*}
and
\[
 [D_{\tau}, \nabla_a] V = -R_{ab}\nabla_b V -(\nabla_bR_{ac} - \nabla_cR_{ab})\Lambda^b_c(V),
\]
we have
\begin{align}
\begin{split}\label{eq:deltadtcomm}
&\iint \tau^2\langle [\Delta, D_{\tau}] V, V\rangle\,\dX\\
&\qquad= 
\iint \tau^2\bigg(\frac{1}{2}\Delta R |V|^2 -2R_{ab}\langle \nabla_a V, \nabla_b V\rangle\\
&\qquad\phantom{=}\quad\quad
-2\left\langle(\nabla_bR_{ac} - \nabla_cR_{ab})\Lambda^b_c(V),\nabla_a V\right\rangle\bigg)\,\dX. 
\end{split}
\end{align}

Likewise, for the second commutator term in \eqref{eq:comm1}, we compute that
\begin{align*}
 \begin{split}
\iint \tau^2\langle \nabla_{\nabla\phi}(\Delta V), V\rangle  \,\dX &= -\iint\tau^2\bigg\{
\Delta\phi\langle \Delta V , V\rangle + \langle \nabla_{\nabla\phi}V, \Delta V\rangle\bigg\}\,\dX,
\end{split}
\end{align*}
and
\begin{align*}
 \begin{split}
\iint \tau^2\langle \Delta(\nabla_{\nabla\phi}V), V\rangle\,\dX
&= \iint\tau^2\bigg\{\Delta\phi|\nabla V|^2 
- 2\langle [\nabla_a, \nabla_b]V, \nabla_a V\rangle \nabla_b\phi\\
&\qquad\phantom{=}\quad-2\nabla_a\nabla_b\phi\langle \nabla_a V, 
\nabla_b V\rangle - \langle \nabla_{\nabla \phi} V, \Delta V\rangle \bigg\}\,\dX.
\end{split}
\end{align*}
Using that
\[
 \nabla_d\phi[\nabla_a, \nabla_d]V = -R_{bcad}\nabla_d\phi\Lambda^b_c(V),
\]
we then have
\begin{align}
 \begin{split}\label{eq:ddphideltacomm}
&2\iint \tau^2\left\langle[\nabla_{\nabla\phi}, \Delta]V, V\right\rangle \,\dX \\
&\quad=\iint \tau^2\bigg\{4\nabla_a\nabla_b\phi\langle\nabla_a V, \nabla_b V\rangle - 4 R_{bcad}\nabla_d\phi\langle\Lambda^b_c(V), \nabla_a V\rangle\\
&\quad\phantom{=\iint \tau^2\bigg\{} -\Delta^2\phi|V|^2\bigg\}\,\dX.
\end{split}
\end{align}

Now we expand the third term on the right of \eqref{eq:comm1}. Since
\[
F = |\nabla\phi|^2 - \pd{\phi}{\tau} - \frac{R}{2} - \frac{1}{2\tau},
\]
we compute that
\begin{align*}
 2\langle \nabla F, \nabla \phi\rangle &= 4\nabla\nabla\phi(\nabla\phi, \nabla\phi) - 2\left\langle\nabla\pd{\phi}{\tau}, \nabla\phi\right\rangle
 -\langle\nabla R, \nabla \phi\rangle\\
 &=  4\nabla\nabla\phi(\nabla\phi, \nabla\phi) - 2\Rc(\nabla\phi, \nabla\phi) - \pdtau|\nabla\phi|^2
 -\langle\nabla R, \nabla \phi\rangle,
\end{align*}
and
\[
 \pd{F}{\tau} = \pdtau |\nabla \phi|^2 - \frac{\partial^2\phi}{\partial\tau^2} - \frac{1}{2}\pd{R}{\tau} + \frac{1}{2\tau^2},
\]
so
\begin{align*}
 &\iint\left(2\tau^2\langle \nabla F, \nabla \phi\rangle - \tau^2\pd{F}{\tau}-\tau F\right)|V|^2\,\dX \\
&\qquad= \iint2\tau^2\left(2\nabla\nabla\phi(\nabla\phi, \nabla \phi) - \Rc(\nabla\phi, \nabla \phi) + \frac{|\nabla\phi|^2}{2\tau}\right)|V|^2\,\dX\\
&\qquad\phantom{=}+\iint\bigg\{\tau^2\left(\frac{\partial^2\phi}{\partial\tau^2}  - 2\pdtau |\nabla\phi|^2 + \frac{1}{2}\pd{R}{\tau}- \langle \nabla R, \nabla \phi\rangle\right)\\
&\qquad\phantom{=+\iint\bigg\{}
+ \tau\left(\pd{\phi}{\tau} + \frac{R}{2} -2|\nabla\phi|^2\right)\bigg\}|V|^2\,\dX.
\end{align*}
Combining this with \eqref{eq:comm1}, \eqref{eq:deltadtcomm}, and \eqref{eq:ddphideltacomm} yields \eqref{eq:commutator}.
\end{proof}

 \begin{remark}
 \label{rem:solitoncommutator}
  When $g(\tau)$ is a shrinking self-similar solution to \eqref{eq:brf} in the sense that $(M, g(1), f(1))$ satisfies \eqref{eq:grs} and $g(\tau) = \tau\Psi^*_{\tau}g(1)$, $f(\tau) = f\circ\Psi_{\tau}^*f(1)$  
  where $\pd{\Psi}{\tau} = -\tau^{-1}(\nabla_{g(1)} f(1))\circ\Psi$ and $\Psi_{1} = \operatorname{Id}$, the quantities $\Qc^{(i)}_{\phi}$, $i =1, 2, 3$, on the right side of \eqref{eq:commutator} vanish identically with the choice
  $\phi = -\frac{f}{2}$. This can be seen immediately for $\Qc^{(1)}_{\phi}$ and $\Qc^{(3)}_{\phi}$ given the identities
 \[
  R_{ij} + \nabla_i\nabla_j f = \frac{g_{ij}}{2\tau}, \quad \nabla_i R_{jk} - \nabla_jR_{ik} = R_{ijk}^l\nabla_l f,
 \]
 satisfied by $g$ and $f$ on $M\times (0, T)$. The vanishing of $\Qc^{(2)}_{\phi}$ follows from the additional identities
 \[
 \Delta f + R = \frac{n}{2\tau}, \quad  \pd{f}{\tau} = - |\nabla f|^2, \quad   \pd{R}{\tau} =- \langle \nabla R, \nabla f\rangle - \frac{R}{\tau},
 \]
 since
 \begin{align*}
\begin{split}
   \Qc^{(2)}_{\frac{-f}{2}} &= \frac{\tau^2}{2}\left(\left(\pd{R}{\tau} + \langle \nabla R, \nabla f\rangle + \frac{R}{\tau}\right) + \Delta(\Delta f + R)\right)\\
		    &\phantom{=}-\frac{\tau^2}{2}\left(\pdtau + \frac{1}{\tau}\right)\left(\pd{f}{\tau} +|\nabla f|^2\right)\\
		    &\phantom{=}-\frac{\tau^2}{2}\left(\Rc(g) + \nabla\nabla f - \frac{g}{2\tau}\right)(\nabla f, \nabla f)\\
		    &=0.
\end{split}
\end{align*}
 \end{remark}

We will use the simple energy estimate in the next proposition to control $|\nabla Z|$ by $|(D_{\tau} + \Delta) Z|$ in combination with our estimate for $|Z|$.
\begin{proposition}\label{prop:intineq2}
If $Z$ is any smooth family of sections of $\Zc$ with compact support in $M\times(0, T)$,
then, for any $j$, $l\geq 0$, and $\cb > 0$,
\begin{align}
\begin{split}
\label{eq:intident2}
&\iint \tau^j|\nabla Z|^2e^{2\phi}\,\dX\\
&\qquad\leq \iint\tau^{j}\left(\Delta \phi + 2|\nabla\phi|^2 - \pd{\phi}{\tau} -\frac{R}{2} 
+ \frac{c\tau^{1-l}}{2\tau}\right)|Z|^2e^{2\phi}\,\dX\\
&\qquad\phantom{\leq}+ \iint\frac{\tau^{j+l}}{2\cb}|(D_{\tau}+\Delta)Z|^2e^{2\phi}\,\dX.
\end{split}
\end{align}
\begin{proof}
 Write $V = e^{\phi}Z$ as before and consider the identities
 \[
  \tau^j|\nabla V|^2 = \frac{1}{2}\left(\pdtau + \Delta\right)(\tau^j|V|^2) - \frac{j\tau^{j-1}}{2}|V|^2 -\tau^j\langle(D_{\tau} + \Delta)V, V\rangle
 \]
and
\begin{align*}
\begin{split}
 \tau^j\langle(D_{\tau}+\Delta)V, V\rangle &= \tau^{j-1}\langle\Lc V, V\rangle + \tau^j\left(\Delta\phi + \pd{\phi}{\tau} - |\nabla\phi|^2\right)|V|^2 + \tau^j\langle \nabla\phi, \nabla|V|^2\rangle.
\end{split}
\end{align*}
Combining these identities, integrating over $M\times (0, T)$, and integrating by parts, we obtain
\begin{align}
 \begin{split}\nonumber
   \iint \tau^j|\nabla V|^2\dX &= \iint\tau^{j}\left(|\nabla\phi|^2 -\pd{\phi}{\tau} -\frac{R}{2} - \frac{j}{2\tau}\right)|V|^2 \,\dX\\
   &\phantom{=} - \iint \tau^{j-1}\langle \Lc V, V\rangle\,\dX
 \end{split}\\
\begin{split}\label{eq:intineq2e1}
  &\leq\iint\tau^{j}\left(|\nabla\phi|^2 -\pd{\phi}{\tau} -\frac{R}{2} +\frac{\cb\tau^{1-l}-j}{2\tau}\right)|V|^2\,\dX\\
  &\phantom{\leq} + \iint\frac{\tau^{j+l}}{2\cb}|(D_{\tau} + \Delta) Z)|^2e^{2\phi}\,\dX
\end{split}
\end{align}
for any $\cb > 0$ and $l\geq 0$.
On the other hand, 
\[
  |\nabla V|^2 = e^{2\phi}(|\nabla Z|^2 + \langle \nabla \phi, \nabla|Z|^2\rangle + |\nabla \phi|^2|Z|^2),
\]
so
\begin{align}
\begin{split}\label{eq:intineq2e2}
  \iint \tau^j|\nabla Z|^2\,\dX &= \iint\tau^j|\nabla V|^2e^{2\phi}\,\dX + \iint\tau^j\left(\Delta\phi + |\nabla\phi|^2\right)|Z|^2e^{2\phi}\,\dX.
\end{split}
\end{align}
Combining \eqref{eq:intineq2e1} and \eqref{eq:intineq2e2}, we obtain \eqref{eq:intident2}.
\end{proof}

\end{proposition}

\subsection{Carleman estimates to imply exponential decay}
\label{ssec:carlemandecay}
For the rest of the section, we will specialize to the cylinder $M = \Cc$ with
\[
\Psi_{\tau}(\theta, z) = (\theta, z/\sqrt{\tau}), \quad g(\tau) = \tau\Psi_{\tau}^*g(1) = (2(k-1)\tau \gc)\oplus\bar{g}, 
\]
and
\[
 \quad f_{z_0}(\theta, z, \tau) = f_{z_0}(\Psi_{\tau}(\theta, z), 1) = \frac{|z-z_0|^2}{4\tau} + \frac{k}{2},
\]
for $\tau > 0$ and some $z_0\in \RR^{n-k}$ as before. 

\subsubsection{An estimate for the PDE component}
We start with the proof of Theorem \ref{thm:pdecarleman2}. Following \cite{EscauriazaSereginSverakHalfSpace}, \cite{WangCylindrical}, 
we define for $\alpha > 0$ and $z_0\in \RR^{n-k}$ the weight function
$\varphi = \varphi_{\alpha, z_0}: \Cc\times (0, \infty)\to \RR$ by
\begin{align}
\begin{split}
\label{eq:phi1def}
 \varphi(z, \theta, \tau)&= - \frac{|z-z_0|^2}{8\tau} - \alpha\log\sigma(\tau)\\
 &= -\frac{1}{2}f_{z_0}(z, \theta, \tau) - \alpha\log\sigma(\tau) + \frac{k}{4},	
 \end{split}
\end{align}
where $\sigma(\tau) = \tau e^{(T-\tau)/3}$. 

\begin{proof}[Proof of Theorem \ref{thm:pdecarleman2}]
Fix $0 < T \leq 2$, $\alpha \geq 1$, and $z_0\in \RR^{n-k}$. It suffices to prove the estimate for the case that $\Zc$ has a single summand (i.e., is a tensor bundle over $\Cc$). Let $Z$ be a smooth family
of sections of $\Zc$ with compact support in $\Cc\times (0, T)$ and write $V = e^{\varphi}Z$. Consider \eqref{eq:commutator} with the choice $\phi = \varphi$.
Since $\varphi$ differs from $-f_{z_0}/2$ 
by a function depending only on $\tau$, it follows from Remark \ref{rem:solitoncommutator} that the quantities $\Qc^{(i)}_{\varphi}$ in \eqref{eq:commutator} 
satisfy
\begin{align*}
  \begin{split}
\Qc^{(1)}_{\varphi} = 0, \quad   \Qc^{(2)}_{\varphi} = -\alpha\tau\left( \tau(\log \sigma)^{\prime\prime} + (\log \sigma)^{\prime}\right) = \frac{\alpha\tau}{3}, \quad \Qc^{(3)}_{\varphi} = 0.
  \end{split}
\end{align*}
According to \eqref{eq:basicdecomposition} and Proposition \ref{prop:commutator}, we then have
\begin{equation}\label{eq:cyldecay1}
 \frac{\alpha}{3}\iint\tau|Z|^2e^{2\varphi}\,\dX \leq \iint \tau^2|(D_{\tau} + \Delta)Z|^2e^{2\varphi}\,\dX.
\end{equation}

To incorporate the derivative of $Z$, we use Proposition \ref{prop:intineq2} with $\phi = \varphi$, $j=2$, $l=1$, and $\cb = 2\alpha$. Using the soliton identities (see Remark 8.2),
we can simplify the integrand of the first integral on the right of \eqref{eq:intident2} to find
\begin{align*}
 &\tau^2\left(\Delta \varphi + |\nabla\varphi|^2 - \pd{\varphi}{\tau} - \frac{R}{2} + \frac{\alpha}{\tau}\right)\\
 &\qquad=\tau^2\left(-\frac{\Delta f_{z_0}}{2} + \frac{|\nabla f_{z_0}|^2}{4} + \frac{1}{2}\pd{f_{z_0}}{\tau} + \alpha (\log\sigma)^{\prime} -\frac{R}{2} + \frac{\alpha}{\tau}\right)\\
 &\qquad= 2\alpha\tau -\tau\left(\frac{\alpha\tau}{3}+\frac{n}{4}\right),
\end{align*}
and hence that
\begin{equation*}
  \iint \tau^2|\nabla Z|^2e^{2\varphi}\dX \leq 2\alpha\iint \tau|Z|^2e^{2\varphi}\dX + \frac{T}{4\alpha}\iint\tau^2|(D_{\tau} + \Delta)Z|^2e^{2\varphi}\dX.
\end{equation*}
Combining this with \eqref{eq:cyldecay1} and using that $T \leq 2$, we arrive at
\[
 \iint (\alpha \tau|Z|^2 + \tau^2|\nabla Z|^2)e^{2\varphi}\,\dX \leq 10 \iint\tau^2|(D_{\tau}+\Delta)Z|^2e^{2\varphi}\,\dX,
\]
which implies \eqref{eq:pdecarleman2}.
\end{proof}

\subsubsection{Estimates for the ODE component}
Both of the Carleman-type estimates \eqref{eq:odecarleman2g} and \eqref{eq:odecarleman2ng} are consequences of the simple identity 
\begin{align}
 \begin{split}
&\label{eq:odeident}
 \pdtau\left(\tau^{j} |Z|^2\,e^{2\phi}\,d\mu\right) = \tau^{j}\left(\left(\frac{j}{\tau} + 2\pd{\phi}{\tau} + R\right)|Z|^2 + 2\langle D_{\tau} Z, Z\rangle\right)e^{2\phi}\,d\mu,
\end{split}
 \end{align}
where $Z$ is a smooth family of tensor fields over $\Cc$, $j\geq 0$ is a fixed number, and $\phi:\Cc\times (0, T)\to \RR$ is an arbitrary smooth function.
\begin{proof}[Proof of Theorem \ref{thm:odecarleman2}]
Again it suffices to consider the case that $\Zc$ is a tensor bundle over $\Cc$. Let $Z$ be a smooth family of sections of $\Zc$ with compact support in $U\times (0, T)$
for some open $U\subset \Cc$. Let $D\subset \Cc$ be any open set with $\overline{D}\subset U$ and fix $\alpha \geq 1$, $\lambda > 0$, and $z_0\in \RR^{n-k}$. (The support
of $Z(\cdot, \tau)$ need not be contained in $\overline{D}$.)

For the first inequality \eqref{eq:odecarleman2g}, we apply \eqref{eq:odeident} with $\phi = \varphi$ and $j = \lambda + 1$ at some fixed $p= (\theta, z)$, obtaining
\begin{align}
\nonumber
\begin{split}
 \pdtau\left(\tau^{\lambda+1}|Z|^2e^{2\varphi}d\mu\right)  &=\bigg(\tau^{\lambda}\left(\lambda+1 + \frac{|z-z_0|^2}{4\tau} + \frac{k}{2} -\frac{2\alpha}{3}(3-\tau)\right)|Z|^2\\
 &\phantom{=}+ 2\tau^{\lambda+1}\langle D_{\tau} Z, Z\rangle\bigg)e^{2\varphi}d\mu. 
\end{split}
 \end{align}
Since $Z$ vanishes identically near $\tau = 0$ and $\tau = T$, we may integrate the above identity over $D\times [0, T]$ to obtain 
\begin{align*}
 \begin{split}
  &\int_0^T\int_D\tau^{\lambda}\left(\frac{8\alpha }{3} - 4\lambda -2k- 4\right)|Z|^2e^{2\varphi}\,\dX \\
  &\qquad\leq \int_0^T\int_D \tau^{\lambda-1}\left(|z-z_0|^2|Z|^2 + 8\tau^2\langle D_{\tau}Z, Z\rangle\right)e^{2\varphi}\,\dX.
  \end{split}
\end{align*}
Estimating
\[
 8\tau^2\langle D_{\tau} Z, Z \rangle \leq \frac{\alpha\tau}{3}|Z|^2 + \frac{48\tau^3}{\alpha}|D_{\tau}Z|^2,
\]
we see that
\begin{align*}
 \begin{split}
  &2\alpha\int_0^T\int_D\tau^{\lambda}|Z|^2e^{2\varphi}\,\dX \leq \int_0^T\int_D \tau^{\lambda-1}\left(|z-z_0|^2|Z|^2 + \frac{48}{\alpha}\tau^2|D_{\tau}Z|^2 \right)e^{2\varphi}\,\dX
  \end{split}
\end{align*}
for $\alpha \geq \alpha^{\prime}(k, \lambda)$ sufficiently large. This implies \eqref{eq:odecarleman2g} for such $\alpha$ and $D$.

For \eqref{eq:odecarleman2ng}, we apply \eqref{eq:odeident} again with $\phi = -\alpha\log \sigma$ and $j = \lambda + 1$ at some fixed $p= (\theta, z)$,
obtaining
\begin{align}
\nonumber
\begin{split}
 &\pdtau\left(\tau^{\lambda+1}|Z|^2\sigma^{-2\alpha}d\mu\right)\\
 &\qquad=\bigg(\tau^{\lambda}\left(\lambda+1 + \frac{k}{2} -\frac{2\alpha}{3}(3-\tau)\right)|Z|^2 + 2\tau^{\lambda+1}\langle D_{\tau} Z, Z\rangle\bigg)\sigma^{-2\alpha}d\mu. 
\end{split}
 \end{align}
Integrating over $D\times [0, T]$, we obtain 
\begin{align*}
 \begin{split}
  &\int_0^T\int_D\tau^{\lambda}\left(\frac{2\alpha }{3} - \lambda - \frac{k}{2} - 1\right)|Z|^2\sigma^{-2\alpha}\,\dX \leq 2\int_0^T\int_D \tau^{\lambda+1}\langle D_{\tau}Z, Z\rangle\sigma^{-2\alpha}\,\dX.
  \end{split}
\end{align*}
Since
\[
 2\tau^{\lambda+1}\langle D_{\tau}Z, Z\rangle \leq \frac{\alpha\tau^{\lambda}}{8}|Z|^2 + \frac{8\tau^{\lambda+2}}{\alpha}|D_{\tau}Z|^2,
\]
we have
 \begin{align*}
 \begin{split}
  &\frac{\alpha}{2}\int_0^T\int_D\tau^{\lambda}|Z|^2\sigma^{-2\alpha}\,\dX 
  \leq \frac{8}{\alpha}\int_0^T\int_D \tau^{\lambda+2}|D_{\tau}Z|^2\sigma^{-2\alpha}\,\dX,
  \end{split}
\end{align*}
provided $\alpha \geq \alpha^{\prime\prime}(k, \lambda)$ is sufficiently large. This implies \eqref{eq:odecarleman2ng} for such $\alpha$. Putting $\alpha_0 = \max\{\alpha^{\prime}, \alpha^{\prime\prime}\}$ finishes the proof.
\end{proof}

\subsection{Carleman estimates to imply backward uniqueness}
Now we prove the second set of Carleman estimates from Section \ref{sec:backwarduniqueness}. Here, as in \cite{WangCylindrical}, we fix some  $0 < T \leq 1$
and construct our weight from the function $\phi_{\alpha} = \phi_{\alpha, \delta}:\Cc\times(0, T)\to \RR$ given by
\[
  \phi_{\alpha}(z, \theta, \tau) = \alpha\eta(\tau)\left(4\left(f_0(z, \theta, \tau) - \frac{k}{2}\right)\right)^{\delta} = 
  \alpha\eta(\tau)\left(\frac{|z|^2}{\tau}\right)^{\delta},
\]
as in \eqref{eq:phialphadef} with $\eta:[0, T]\to [0, 1]$ defined as in \eqref{eq:etadef}.
The function $\eta$ is piecewise-differentiable, twice weakly-differentiable, and satisfies the following inequalities.
\begin{lemma}[\cite{WangCylindrical}]\label{lem:etaprop}
The function $\eta$ is nonincreasing and satisfies
\begin{equation}\label{eq:etaprop}
 0 \leq \eta \leq 1, \quad \delta\eta - \tau\eta^{\prime} \geq \delta, \quad \tau^2\eta^{\prime\prime}\geq -\frac{1}{4}\delta(4\delta -3),
\end{equation}
for $\tau \in [0, T]$.
\end{lemma}
These inequalities are verified in Lemma 2.5 of \cite{WangCylindrical}
for the function $\tilde{\eta}(\tau) = \eta(\tau/T)$. They are invariant under rescaling of $\tau$ and are hence also valid in our situation.
\subsubsection{An estimate for the PDE component}

To apply the integral identities in the preceding section, we first need to collect formulas for the various derivative expressions that appear in the quantities $\Qc^{(i)}_{\phi_{\alpha}}$,
$i=1, 2, 3$, in \eqref{eq:commutator}. The necessary expressions have already been computed  in \cite{WangCylindrical}. (The computations there, made relative to the 
Euclidean metric are valid for the evolving cylindrical metric here
since $\phi_{\alpha}$ is independent of the spherical variables.)

\begin{lemma}[Lemma 2.4, \cite{WangCylindrical}]\label{lem:phialphaderiv}
  For any $\alpha > 0$, the derivatives of the function $\phi_{\alpha}$ satisfy the expressions
 \begin{align*}
   \nabla \phi_{\alpha} &= \frac{2\alpha\delta\eta}{\tau^{\delta}}|z|^{2\delta-2}z,\\
  |\nabla \phi_{\alpha}|^2 &= \frac{4\alpha^2\delta^2\eta^2}{\tau^{2\delta}}|z|^{4\delta -2},\\
    \nabla\nabla\phi_{\alpha} &= \frac{2\alpha\delta\eta}{\tau^{\delta}}|z|^{2\delta -4}\left(|z|^2 \Pb + 2(\delta- 1)z\otimes z\right),\\
    \Delta \phi_{\alpha} &= \frac{2\alpha\delta\left(2(\delta -1) + n-k\right)\eta}{\tau^{\delta}}|z|^{2\delta -2},\\
      \pd{\phi_{\alpha}}{\tau}&=\frac{\alpha(\tau\eta^{\prime} - \delta\eta)}{\tau^{\delta+1}}|z|^{2\delta},\\
      \frac{\partial^2\phi_{\alpha}}{\partial\tau^2} &= \frac{\alpha(\tau^2\eta^{\prime\prime} - 2\delta\tau\eta^{\prime}
+ \delta(\delta +1)\eta)}{\tau^{\delta+2}}|z|^{2\delta},\\
    \pdtau |\nabla\phi_{\alpha}|^2 &= \frac{8\alpha^2\delta^2\eta(\tau\eta^{\prime} -\delta\eta)}{\tau^{2\delta+1}}|z|^{4\delta -2},\\
  \Delta^2\phi_{\alpha} &= \frac{4\alpha\delta(\delta-1)(2(\delta -1) + n-k)(2(\delta-2) + n-k)\eta}{\tau^{\delta}}|z|^{2\delta-4},
\end{align*}
on $\Cc_r\times (0, T)$ for any $r > 0$.
\end{lemma}
Above, in the first and third equations, we identify $z$ with the differential of the function $(\theta, z)\mapsto |z|^2/2$ and, in the expression 
for $\nabla\nabla\phi_{\alpha}$, we identify the endomorphism $\Pb$ with the two-tensor $\Pb_{ij} = \Pb_i^kg_{kj}$. Now we prove Theorem \ref{thm:pdecarlemanbu}. 
\begin{proof}[Proof of Theorem \ref{thm:pdecarlemanbu}] 
Fix $\delta \in (7/8, 1)$ and $T\in (0, 1]$, and let $r\geq r_3$ where $r_3\geq 1$
is to be specified over the course of the proof. We will assume, as before, that $\Zc$ is a fixed tensor bundle over $\Cc$.  Let $Z$ be a smooth family of sections of $\Zc$ on $\Cc_{r}$ 
defined for $\tau\in (0, T)$ and let $V = e^{\phi_{\alpha}}Z$.

With an eye toward \eqref{eq:commutator}, let us define
\begin{align*}
 S_{\phi_{\alpha}} &=  \frac{g}{2\tau} - \Rc(g) + 2\nabla\nabla\phi_{\alpha}. 
\end{align*}
Then, using Lemma \ref{lem:phialphaderiv}, we have
\begin{align*}
S_{\phi_{\alpha}} &= \frac{\Pb}{2\tau} + 2\nabla\nabla\phi_{\alpha}
= \frac{\Pb}{2\tau} + \frac{4\alpha\delta\eta}{\tau^{\delta}}|z|^{2\delta -4}\left(|z|^2 \Pb + 2(\delta- 1)z\otimes z\right).
\end{align*}
Since $\delta > 1/2$, the second term, and hence the sum, is nonnegative-definite when considered as a two-tensor on $T\Cc$ over $\Cc_r$.  In particular, the quantity $\Qc^{(1)}_{\phi_{\alpha}}(\nabla V, \nabla V)$ from \eqref{eq:commutator} is nonnegative.  

For the quantity $\Qc^{(2)}_{\phi_{\alpha}}$, we have similarly that
\begin{align*}
\begin{split}
   &\Qc^{(2)}_{\phi_{\alpha}} \geq \tau^2\left(\frac{\partial^2\phi_{\alpha}}{\partial\tau^2} - \Delta^2\phi_{\alpha} -2\pdtau|\nabla \phi_{\alpha}|^2\right) 
  + \tau\left(\pd{\phi_{\alpha}}{\tau} -2|\nabla \phi_{\alpha}|^2\right),
   \end{split}
\end{align*}
where we have used that $\nabla R = 0$, $\Delta R = 0$, and $\pd{R}{\tau} + R/{\tau} = 0$. 
Now, two of the terms on the right are proportional to $\alpha^2$. Using Lemmas \ref{lem:etaprop} and \ref{lem:phialphaderiv},
we see that we may estimate them below by
\begin{align*}
-2\tau\left(\tau\pdtau|\nabla \phi_{\alpha}|^2 +  |\nabla \phi_{\alpha}|^2\right)
&= -\frac{8\alpha^2\delta^2\eta|z|^{4\delta-2}}{\tau^{2\delta-1}}\left(2(\tau\eta^{\prime} -\delta\eta) +\eta\right)\\
&\geq\frac{6\alpha^2\delta^2\eta|z|^{4\delta-2}}{\tau^{2\delta-1}}.
\end{align*}
The remaining terms are proportional to $\alpha$, and we may estimate them similarly:
\begin{align*}
\begin{split}
&\tau^2\left(\frac{\partial^2\phi_{\alpha}}{\partial\tau^2} - \Delta^2\phi_{\alpha}\right) + \tau\pd{\phi_{\alpha}}{\tau}\\
&\quad=\frac{\alpha|z|^{2\delta}}{\tau^{\delta}}\bigg(\tau^2\eta^{\prime\prime} - 2\delta\tau\eta^{\prime}
+ \delta(\delta +1)\eta + (\tau\eta^{\prime} - \delta\eta)
- \frac{C(\delta, k, n)\tau^2\eta}{|z|^{4}}\bigg)
\end{split}\\
&\quad\geq \frac{\alpha|z|^{2\delta}}{\tau^{\delta}}\left(\frac{3\delta\eta + (1-4\delta)\tau\eta^{\prime}}{4}  - \frac{C(\delta, k, n)\tau^2\eta}{|z|^4}\right).
\end{align*}
So, if $r_3 = r_3(\delta, n, k)$ is taken sufficiently large, we will have
\[
\Qc^{(2)}_{\phi_{\alpha}}\geq \frac{\alpha|z|^{2\delta}}{2\tau^{\delta}}(\delta\eta - \tau\eta^{\prime}) + \frac{6\alpha^2\delta^2\eta|z|^{4\delta-2}}{\tau^{2\delta-1}},
\]
on $\Cc_{r}\times (0, T)$.

Finally, $\Qc^{(3)}_{\phi_{\alpha}} = 0$ on the cylinder since $\nabla \Rc = 0$ and $\Rm(\cdot, \cdot, \cdot, \nabla \phi_{\alpha}) = 0$. Putting things together
and using \eqref{eq:basicdecomposition}
and \eqref{eq:commutator}, we thus see that
\begin{align}
\label{eq:bucarleman1}
\begin{split}
 &\int_0^T\int_{\Cc_r}\left(\frac{\alpha|z|^{2\delta}}{2\tau^{\delta}}(\delta\eta-\tau\eta^{\prime}) +\frac{6\alpha^2\delta^2\eta|z|^{4\delta-2}}{\tau^{2\delta-1}}\right)|Z|^2e^{2\phi_{\alpha}}\,\dX\\
 &\qquad\leq \int_0^T\int_{\Cc_r}\tau^2|D_\tau Z + \Delta Z|^2e^{2\phi_{\alpha}}\, \dX, 
\end{split}
 \end{align}
for all $\alpha > 0$ and $r\geq r_3$. 

Now we use Proposition \ref{prop:intineq2} to add in the derivative term. Taking $\phi = \phi_{\alpha}$ and $\cb = j= l =1$ in \eqref{eq:intident2} yields
\begin{align}
\label{eq:bucarleman2}
\begin{split}
&\int_0^T\int_{\Cc_r} \tau|\nabla Z|^2e^{2\phi_{\alpha}}\,\dX\\
&\qquad\leq \int_0^T\int_{\Cc_r}\tau\left(\Delta \phi_{\alpha} + 2|\nabla\phi_{\alpha}|^2 - \pd{\phi_{\alpha}}{\tau} -\frac{R}{2}\right)|Z|^2e^{2\phi_{\alpha}}\,\dX\\
&\phantom{\qquad\leq}+ \int_0^T\int_{\Cc_r}\frac{\tau^{2}}{2}|(D_{\tau}+\Delta)Z|^2e^{2\phi_{\alpha}}\,\dX.
\end{split}
\end{align}
Then, by Lemmas \ref{lem:etaprop} and \ref{lem:phialphaderiv},
\begin{align*}
&\tau\left(\Delta \phi_{\alpha} + 2|\nabla\phi_{\alpha}|^2 - \pd{\phi_{\alpha}}{\tau} -\frac{R}{2}\right) =\frac{2\alpha\delta\left(2(\delta -1) + n-k\right)\eta}{\tau^{\delta}}|z|^{2\delta -2}\\
&\qquad\phantom{=}
 + \frac{8\alpha^2\delta^2\eta^2}{\tau^{2\delta}}|z|^{4\delta -2}
 -\frac{\alpha(\tau\eta^{\prime} - \delta\eta)}{\tau^{\delta+1}}|z|^{2\delta} - \frac{k}{4\tau}\\
 &\qquad\leq \frac{\alpha|z|^{2\delta}}{\tau^{\delta}}\left( (\delta\eta - \tau\eta^{\prime}) + \frac{\tau(n-k)}{|z|^2} \right) + \frac{8\alpha^2\delta^2\eta^2|z|^{4\delta-2}}{\tau^{2\delta-1}}\\
  &\qquad\leq \frac{\alpha|z|^{2\delta}}{\tau^{\delta}}\left( (\delta\eta - \tau\eta^{\prime}) + \frac{\delta}{2}\right) + \frac{8\alpha^2\delta^2\eta^2|z|^{4\delta-2}}{\tau^{2\delta-1}},
\end{align*}
for $r_3$ sufficiently large.  Returning to \eqref{eq:bucarleman2} with this, multiplying both sides by $1/4$, and combining the result with \eqref{eq:bucarleman1}, we obtain
\begin{align*}
\begin{split}
 &\int_0^T\int_{\Cc_r}\left(\frac{\alpha\delta|z|^{2\delta}}{8\tau^{\delta}}|Z|^2 + \frac{\tau}{4}|\nabla Z|^2\right)e^{2\phi_{\alpha}}\,\dX \leq \int_0^T\int_{\Cc_r}\tau^2|D_\tau Z + \Delta Z|^2e^{2\phi_{\alpha}}\, \dX,
 \end{split}
\end{align*}
for $r\geq r_3$ and all $\alpha > 0$. The estimate \eqref{eq:pdecarlemanbu} follows.
\end{proof}

\subsubsection{An estimate for the ODE component}

For the proof of the matching estimate for the ODE component, we again use the identity \eqref{eq:odeident}. 
\begin{proof}[Proof of Theorem \ref{thm:odecarlemanbu}]  Fix $\alpha\geq 1$, $0 < T\leq 1$, and let $r\geq r_4$ for some $r_4$ to be specified later. Let $Z$ be a smooth family of sections
 the tensor bundle $\Zc$ with compact support in $\Cc_{r}\times (0, T)$.  Starting from \eqref{eq:odeident} with $j=1$ and $\phi = \phi_{\alpha}$, we have 
\begin{equation*}
 \pdtau\left(\tau|Z|^2e^{2\phi_{\alpha}}d\mu\right) = \tau\left(\left(\frac{1}{\tau} + 2\pd{\phi_{\alpha}}{\tau} + \frac{k}{2\tau}\right)|Z|^2 + 2\langle D_{\tau} Z, Z\rangle\right)e^{2\phi_{\alpha}}d\mu.
\end{equation*}
By Lemmas \ref{lem:etaprop} and \ref{lem:phialphaderiv}, 
\[
  \pd{\phi_{\alpha}}{\tau} = \alpha(\tau\eta^{\prime} - \delta\eta)\tau^{-\delta-1}|z|^{2\delta} \leq -\alpha\delta\tau^{-\delta-1}|z|^{2\delta},
\]
so, integrating over $\Cc_r\times (0, T)$ and using Cauchy-Schwarz, we see that
\begin{align*}
 \int_0^T\int_{\Cc_r}\tau^2|D_{\tau} Z|^2e^{2\phi_{\alpha}}\,\dX &\geq 
 -\int_0^T\int_{\Cc_r}\left(2\tau \pd{\phi_{\alpha}}{\tau} + \frac{k+4}{2}\right)|Z|^2e^{2\phi_{\alpha}}\,\dX\\
 &\geq \int_0^T\int_{\Cc_r}\left(\frac{2\alpha\delta|z|^2}{\tau^{\delta}} - \frac{k+4}{2}\right)|Z|^2e^{2\phi_{\alpha}}\,\dX.
\end{align*}
Thus, provided $r_4 = r_4(n, k, \delta)$ is sufficiently large, we will have
\begin{align*}
 \int_0^T\int_{\Cc_r}\tau^2|D_{\tau} Z|^2e^{2\phi_{\alpha}}\,\dX 
 &\geq \int_0^T\int_{\Cc_r}\frac{\alpha\delta|z|^2}{\tau^{\delta}}|Z|^2e^{2\phi_{\alpha}}\,\dX
\end{align*}
as claimed.
\end{proof}
\appendix

\section{Normalizing the soliton vector field.}
\label{app:normalization}

In this section, we prove Theorem \ref{thm:vfnormalization}, which provides the diffeomorphism $\Phi$ we use to identify the soliton vector field with that of the standard cylindrical soliton structure. 

\subsection{Preliminaries}
Let us first review the  prerequisites we need from ODE theory, following Chapter 9 of \cite{LeeSmoothManifolds}. Recall that a \emph{flow-domain} on a manifold $M$
is an open set $\mathcal{D}\subset M\times \RR$ satisfying that, for each $p\in M$, the set of $t$ for which $(p, t)$ belongs to $\mathcal{D}$ is an open interval containing $0$.
(Here, the order of the time and space variables is opposite to that in \cite{LeeSmoothManifolds}.)
A \emph{smooth flow} 
is a smooth map $\Theta:\mathcal{D}\to M$ from a flow domain $\mathcal{D}$ which satisfies the group laws
\[
  \Theta(p, 0) = p, \quad \Theta(\Theta(p, s), t) = \Theta(p, s+t),
\]
for all $p\in M$ and $s$, $t\in \RR$ for which $(p, s)$, $(\Theta(p, s), t)$, and $(p, s+t)$ belong to $\mathcal{D}$. 

The infinitesimal generator $\pd{\Theta}{t}(p, 0)$ of a smooth flow $\Theta:\mathcal{D}\to M$ is a smooth flow  is a smooth vector field on $M$, and it is a consequence of the local theory of ODE that, to each smooth vector
field $V$, there is a maximally defined smooth flow $\Theta$ whose infinitesimal generator is $V$. (See Theorem 9.12 in \cite{LeeSmoothManifolds}.)

The main tool we need is the \emph{Flowout Theorem} (Theorem 9.20 in \cite{LeeSmoothManifolds}), which asserts that if $S\subset M$ is a compact hypersurface and
$V$ is nowhere tangent to $S$, then the restriction of the flow $\Theta$ of $V$ to $\mathcal{O} = (S\times \RR)\cap \mathcal{D}$ is a smooth immersion
which pushes forward the coordinate vector field $\pd{}{t}$ along $\RR$ to $V$. When $S$ is a compact hypersurface, there is $\delta > 0$
such that $\Phi|_{S\times (-\delta, \delta)}$ is  a diffeomorphism onto its image.

\subsection{A sequence of maps identifying the vector fields}

Now we specialize to the setting Theorem \ref{thm:vfnormalization}.
We will assume below that $(\Cc_{r_0}, \gt, \delt\ft)$ is strongly asymptotic to $(\Cc, g, \nabla f)$ as a soliton structure
and write, as before, 
\[
h = \gt - g, \quad \Xt = \delt \ft, \quad X = \nabla f = \frac{r}{2}\pd{}{r}, \quad E = \Xt - X.
\]
By assumption, there
are constants $M_{l, m}$ such that
\begin{equation}\label{eq:decay4}
 \sup_{\Cc_{r_0}} |z|^l\left\{|\nabla^{(m)}h| + |\nabla^{(m)}E|\right\} \leq M_{l, m}
\end{equation}
for all $l$, $m\geq 0$. 

Using the notation and terminology of the previous section, let $\Theta: \Dc\subset \Cc_{a_0}\times \RR\to \Cc$ be the maximal smooth flow of $\Xt$. 
There are a variety of ways to use $\Theta$ to construct an injective local diffeomorphism $\Ss^k\times \Ss^{n-k}\times (0, \infty) \to \Cc$ by identifying
$\Ss^{k}\times \Ss^{n-k-1}$ with an appropriate hypersurface in $\Cc_{r_0}$ to which $\Xt$ is nowhere tangent. Each of these local diffeomorphisms
can be adjusted to pull $\Xt$ back to $X$. The trick is to choose an identification for which it is convenient to see that the pull-back of $\gt$ by the map this identification produces
is still strongly asymptotic to the cylindrical metric. We will construct a sequence of maps $\Phi^{(b)}$ from the identifications of $\Ss^k\times \Ss^{n-k-1}$ with $\Sc_b = \partial\mathcal{C}_b^k$ for values of $b$ tending to infinity. 
From this sequence, we will extract a limit map which, in a sense, agrees with the identity to infinite order at spatial infinity.

To begin, let us use the infinite order agreement of $\tilde{X}$ and $X$ to choose $a_0$ so large that $a_0 > 2r_0$ and
\begin{equation}\label{eq:xangle}
  \left.\left\langle \tilde{X}, \pd{}{r} \right\rangle \right|_{(\theta, \sigma, r)}  \geq \frac{r}{4}, \quad 
  \left.\left\langle \tilde{X}, \pd{}{r} \right\rangle_{\gt} \right|_{(\theta, \sigma, r)}  \geq \frac{r}{4},
\end{equation}
on $\Cc_{a_0}$.

\begin{proposition}
\label{prop:phiexist}
There exists a constant $a_1 \geq a_0$ with the property that, for each $b \geq a_1$, there is an injective local diffeomorphism $\Phi^{(b)}: \Cc_{a_1} \longrightarrow \Cc_{a_1/2}$ satisfying
\begin{equation}\label{eq:phibvf}
d\Phi^{(b)}_{(\theta, \sigma, s)}\left(\frac{s}{2}\pd{}{s}\right) = \tilde{X}(\Phi^{(b)}(\theta, \sigma, s)), \quad \left.\Phi^{(b)}\right|_{\Sc_{b}} = \operatorname{Id}_{\Sc_b},
\end{equation}
\begin{equation}\label{eq:phibimage}
 \Cc_{2a_1} \subset \Phi^{(b)}(\Cc_{a_1}),
\end{equation}
and
\begin{equation}\label{eq:rscomp}
 \frac{s}{2} \leq r\circ \Phi^{(b)}(\theta, \sigma, s) \leq 2s.
\end{equation}
Additionally, for each $l\geq 0$, there is a constant $C_l$ such that
\begin{align}\label{eq:rsc0}
      |r\circ \Phi^{(b)}(\theta, \sigma, s) - s| &\leq C_l\left|\frac{1}{s^{l}} - \frac{s}{b^{l+1}}\right|,\\
      \label{eq:transversec0}
      d_{\Sc_s}((\theta, \sigma, s), \pi\circ\Phi^{(b)}(\theta, \sigma, s)) &\leq C_l\left|\frac{1}{s^l} -\frac{1}{b^l}\right|,
\end{align}
for all $b$, $s\geq a_1$,
where $d_{\Sc_{s}}$ is the induced distance on $\Sc_{s}$ and $\pi = \pi_{s}: \Cc_{a_0}\to \Sc_s$ is the projection $\pi_{s}(\theta, \sigma, r) = (\theta, \sigma, s)$.
\end{proposition}

\begin{proof}
By \eqref{eq:xangle}, $\tilde{X}$ is nowhere tangent to $\Sc_a$ for $a \geq a_0$. We use this to construct a preliminary map $\tilde{\Phi}^{(b)}$ following Theorem 9.20 of \cite{LeeSmoothManifolds}.
Let 
$\Theta:\mathcal{D}\subset \Cc_{a_0}\times \RR\to \Cc$
be the maximal smooth flow of $\tilde{X}$, and let $\tilde{\Phi}^{(b)} = \Theta|_{\mathcal{O}_b}$ where $\mathcal{O}_b = \mathcal{D}\cap (\Sc_b\times \RR)$. By \eqref{eq:xangle}, $r$ is 
increasing along the integral curves of $\tilde{X}$, so the  
flow of $\tilde{X}$ preserves $\Cc_{a_0}$.  By \eqref{eq:decay4}, $|\tilde{X}|\leq M(r+1)$ for some $M$, so the integral curves of $\tilde{X}$ starting at any point in $\Cc_{a_0}$ exist for all positive $t$. 

Fix some $a > a_0$. By the compactness of $\Sc_a$, we will have
$\Sc_{a}\times (-\delta, \infty) \subset \mathcal{O}_{a}$ for some $\delta > 0$, and this implies that, for all $b\geq a$,
we will have $\Sc_{b}\times (-(\delta + \alpha(b)), \infty) \subset \mathcal{O}_{b}$ where
\[
  \alpha(b) \dfn \inf \left\{\,t\,|\ \Theta(\Sc_{a}\times\{t\})\cap \Sc_{b} \neq \emptyset\,\right\}
\]
is the minimum time needed to reach $\Sc_{b}$ via an integral curve of $\tilde{X}$ starting in $\Sc_{a}$.

Now, just as in \cite{LeeSmoothManifolds}, each 
$\tilde{\Phi}^{(b)}$ is a local diffeomorphism, and 
\[
 d\tilde{\Phi}^{(b)}_{(\theta, \sigma, t)}\left(\pdt\right) = \tilde{X}(\tilde{\Phi}^{(b)}(\theta, \sigma, t)), \quad \tilde{\Phi}^{(b)}(\theta, \sigma, 0) = (\theta, \sigma, b).
\]
Provided $\delta$ is small enough, the restriction of $\tilde{\Phi}^{(b)}$ to $\Sc_{b}\times (-\delta, \delta)$ will be injective and hence a diffeomorphism onto its image. But it is not hard
to see that $\tilde{\Phi}^{(b)}$ is actually injective on all of $\Cc_{b - (\alpha(b) + \delta)}$. Indeed, $\frac{d}{ds}r(\gamma(s)) \geq a_0/4 > 0$ along any integral curve $\gamma$ of $\tilde{X}$,
so each point in the image of $\tilde{\Phi}^{(b)}$ lies on an integral curve which intersects $\Sc_b$ in exactly one point. Following each point $(\theta_0, \sigma_0, s_0)$ in the image
along an integral curve of $\tilde{X}$ to $\Sc_b$ thus associates the point with a unique radial translation $t$ and a unique $(\theta, \sigma)$ such that $(\theta, \sigma, b) \in \Sc_b$
and $\tilde{\Phi}^{(b)}(\theta, \sigma, t) = (\theta_0, \sigma_0, s_0)$.

Now define
\[
 \Phi^{(b)}(\theta, \sigma, s) = \tilde{\Phi}^{(b)}(\theta, \sigma, 2\ln(s/b))
\]
for all $(\theta, \sigma, s)$ such that $(\theta, \sigma, 2\ln(s/b)) \in \mathcal{O}_b$. Then
\[
 d\Phi^{(b)}_{(\theta, \sigma, s)}\left(\frac{s}{2}\pd{}{s}\right) = \tilde{X}(\Phi^{(b)}(\theta, \sigma, s)), \quad \Phi^{(b)}|_{\Sc_b} = \operatorname{Id}|_{\Sc_b},
\]
and $\Phi^{(b)}$ is a diffeomorphism onto its image.

Now we consider the distortion of distance under $\Phi^{(b)}$. Fix $(\theta, \sigma)\in \Ss^k\times \Ss^{n-k-1}$. For all $s$ such that $\gamma^{(b)}(s) = \Phi^{(b)}(\theta, \sigma, s)$ is well-defined, we have from Proposition \ref{prop:xnorm}
that $r^{(b)}(s) = r(\gamma^{(b)}(s))$ satisfies
\begin{align}\label{eq:rderiv}
\begin{split}
 \frac{d}{ds}\left(\frac{r^{(b)}(s)}{s}\right) &= - \frac{r^{(b)}(s)}{s^2} + \frac{2}{s^2}\left.\left\langle \tilde{X}, \pd{}{r}\right\rangle\right|_{\gamma^{(b)}(s)}\\
 &= \frac{2}{s^2}\left.\left\langle E, \pd{}{r}\right\rangle\right|_{\gamma^{(b)}(s)}.
 \end{split}
\end{align}
Integrating from $s$ to $b$, we find that
\begin{equation*}
      \left|\frac{r^{(b)}(s)}{s} - 1\right| \leq c\left|\int_s^b\frac{1}{t^2}\,dt\right|,
\end{equation*}
for some  $c$ independent of $\theta$, $\sigma$, and, in particular, that
\begin{equation}\label{eq:rest0}
-c \leq r^{(b)}(s) - s \leq c,
\end{equation}
for all $s \leq b$ such that $\gamma^{(b)}(s)$ is defined.  But $\gamma^{(b)}(s)$ will be defined at least as long as $r^{(b)}(s) > a_0$, and, so, at least for all $s > a_0 + c$.  
Choose $a_1 = 2(a_0 + c)$. Then $\Phi^{(b)}$ will be defined on $\Cc_{a_1}$ and \eqref{eq:rest0} says that, for $b\geq a_1$,
\[
r^{(b)}(a_1) \geq a_1 - c = 2a_0 + c > \frac{a_1}{2}.
\]
Consequently, $\Phi^{(b)}(\Cc_{a_1}) \subset \Cc_{a_1/2}$. Similarly,
\[
 r^{(b)}(a_1) \leq a_1 + c \leq 2a_0 + 3c  < 2a_1,
\]
so $\Cc_{2a_1} \subset \Phi^{(b)}(\Cc_{a_1})$.
For $b$, $s\geq a_1$, we will also have 
\[
\frac{s}{2} \leq a_0 + \frac{s}{2} \leq s - c    \leq r^{(b)}(s) \leq  s + c \leq 2s, 
\]
which is \eqref{eq:rscomp}. We may then estimate $|E\circ\Phi^{(b)}| \leq C_lr^{-l} \leq C_l2^ls^{-l}$. Returning to \eqref{eq:rderiv} with this
bound and integrating again along arbitrary paths with fixed $\theta$, $\sigma$ we obtain \eqref{eq:rsc0}.

The estimate \eqref{eq:transversec0} is proven in the same way. Fix $(\theta, \sigma) \in \Ss^k\times \Ss^{n-k-1}$ and $s_0\geq a_1$ and let $p(s) = \pi_{s_0}\circ\Phi^{(b)}(\theta, \sigma, s)$.
For any $s$, we have
 \[
  p^{\prime}(s) = d\pi_{s_0}\circ d\Phi^{(b)}\left(\pd{}{s} \right) = \frac{2}{s}d\pi_{s_0}(\tilde{X}(p(s))) = \frac{2}{s}d\pi_{s_0}(E(p(s))),
 \]
while, by estimate \eqref{eq:rsc0} above, we have $|E(p(s))| \leq C_ls^{-l}$
for all $l\geq 0$ for some $C_l$ independent of $\theta$ and $\sigma$.
But this is enough, since
\[
 |d\pi_{s_0}(E(p(s)))|_{g_{\Sc_{s_0}}} \leq \frac{s_0}{s}|E(p(s))|,
\]
and so
\begin{align*}
&d_{\Sc_{s_0}}\left((\theta, \sigma, s_0), \pi_{s_0}\circ\Phi^{(b)}(\theta, \sigma, s_0)\right) = d_{\Sc_{s_0}}(p(b), p(s_0))\\
&\qquad\qquad\leq \left|\int_{s_0}^b |p^{\prime}(t)|_{g_{\Sc_{s_0}}}\,dt\right| \leq C\int_{s_0}^b \frac{1}{t^{l+1}}\,dt,
\end{align*}
and \eqref{eq:transversec0} follows.
\end{proof}

\subsection{Analysis of an associated system of ODE} Next we seek uniform derivative estimates on the family of maps $\Phi^{(b)}$ in order to extract a limit as $b\to \infty$.
The distance distortion estimates \eqref{eq:rsc0}-\eqref{eq:transversec0} guarantee that the image of a point under $\Phi^{(b)}$ will not wander too far from the point itself, and therefore that
we can obtain the derivative estimates we need from an analysis of the local coordinate representations
of $\Phi^{(b)}$ relative to a fixed finite atlas on $\Cc_{a_0}$. Each of these coordinate representations will satisfy a system of equations with a common structure
reflecting the infinite order agreement of $\Xt$ and $X$ at spatial infinity. We analyze a general version of this system now.

Consider solutions
\[
     \psi: U\times (s_0, \infty) \to W\subset \RR^{n-1}, \quad r: U\times (s_0, \infty) \to (s_1, \infty),
\]
to the system
\begin{align}
\begin{split}\label{eq:psirssys}
   \pd{\psi}{s} &= \frac{2}{s}E_{\psi}(\psi, r), \quad \psi(x, b) = x,\\
    \pd{r}{s} & = \frac{r}{s} + \frac{2}{s}E_{r}(\psi, r), \quad r(x, b) = b,
\end{split}
\end{align}
where $U\subset \RR^{n-1}$, $W\subset \RR^{n-1}$ are open sets and
$E = (E_{\psi}, E_r): W\times (r_0, \infty)\to \RR^n$ satisfies 
\begin{equation}\label{eq:edecay}
 \left|\frac{\partial^{|\mu| + p}E}{\partial y^{\mu} \partial r^p}\right|(y, r) \leq \frac{C(\mu, l,  p)}{r^l} 
\end{equation}
for all $l$, $p\geq 0$ and all multiindices $\mu = (\mu_1, \ldots, \mu_{n-1})$.

Here in this subsection (and only for this subsection) we will write 
\[
\Phi(x, s) = (\psi(x, s), r(x, s)),
\]
and use 
$\langle \cdot, \cdot \rangle$ and $|\cdot|$ to denote the standard Euclidean inner product and norm on $\RR^{n}$. 
The collision of notation is intentional: in our eventual application to the proof of Theorem \ref{thm:vfnormalization}, the neighborhoods $U$ and $W$ will correspond to the images of charts on coordinate neighborhoods of $\Ss^k\times \Ss^{n-k-1}$.
The maps $\Phi$ and $E$ will correspond to the coordinate 
representations of $\Phi^{(b)}$ (for fixed $b$) and $E$ relative to the associated charts on $\Cc$. 

Our goal is to derive
estimates on $\Phi$ from this system on compact subsets of $U\times (s_0, \infty)$ which are independent of $b$.
\begin{proposition}\label{prop:odeckestimate} Let $V$ be a precompact open set with $\overline{V}\subset U$. Then, 
for all $k$, $l \geq 0$, there is a constant $C = C(k, l)$ depending on $V$, but independent of $b$, such that
 \begin{equation}\label{eq:odeckestimate}
      \sup_{V\times (s_0, b]}s^l\left| \frac{\partial^{|\mu|+p}}{\partial x^{\mu}\partial s^p}(\Phi - \operatorname{Id})\right|(x, s) \leq C(k, l)
 \end{equation}
for all $\mu$ and $p \geq 0$ such that $|\mu| + p = k$.
\end{proposition}
\begin{proof} Let $V$ be a precompact open set with with $\overline{V}\subset U$. Fix $x\in V$. Then
 \begin{equation}\label{eq:c01}
   \pd{}{s}\left(\frac{r(x, s)}{s}\right) = \frac{2}{s^2} E_r(\psi(x, s), r(x, s)),
 \end{equation}
so, using the bound $|E(\psi, r)|\leq C$ we have
\[
 \left|1 - \frac{r(x, s)}{s}\right| \leq C\int_s^b \frac{1}{t^2}\,dt \leq  \frac{C}{s},
\]
and hence that $|r(x, s)- s| \leq C$ for any $x$ and any $s_0 < s \leq b$.  

For all $s$ sufficiently large, we will also have that $s/2 \leq r(x, s) \leq 2s$. 
Hence, for each $l$, there is $C_l$ such that $|E(\psi(x, s), r(x, s))| \leq C_ls^{-l}$.
Returning to \eqref{eq:c01}, then, we can estimate
\[
 \left|1 - \frac{r(x, s)}{s}\right| \leq \int_s^b \frac{2}{t^2}|E(\psi(x, t), r(x, t))|\,dt \leq  C_l\int_s^b \frac{1}{t^{l+2}}\,dt \leq  \frac{C_l}{s^{l+1}},
\]
and hence that $|r(x, s) - 1| \leq C_ls^{-l}$. Using now that $r$ and $s$ are comparable, we obtain similarly that
\[
 |\psi(x, s) - x| \leq \int_s^b\frac{2}{t}|E_{\psi}(\psi(x, t), r(x, t))|\,dt \leq \frac{C_l}{s^{l}}.
\]

Now we estimate the first derivatives of $\Phi$. Fix some $l \geq 0$. From what we have done above, we already know that
 \[
  \left|\pd{\psi}{s}\right| = \frac{2}{s}\left|E_{\psi}(\psi, r)\right| \leq \frac{C_l}{s^l}, \quad
 \left|\pd{r}{s} -1\right| \leq \left|\frac{r}{s} - 1\right| + \frac{2}{s}|E_{r}(\psi, r)| \leq \frac{C_l}{s^l}.
\]
For the $x$-derivatives, it will be convenient to introduce the map 
\[
F = \rho_{\frac{1}{s}}\circ \Phi:U\times (s_0, \infty) \to W\times (0, \infty),
\]
where $\rho_{\lambda}(x, r) = (x, \lambda r)$, i.e., $F(x, s) = (\psi(x, s), r(x, s)/s)$.  Fix $1 \leq i \leq n-1$. Then
\begin{align*}
 \pd{}{s} \pd{F}{x^i} &= \pd{}{x^i}\left(\frac{2}{s}E_{\psi}\circ \Phi, \frac{2}{s^2}E_{r}\circ \Phi\right)
		      = \frac{2}{s} (d\rho_{\frac{1}{s}}\circ dE)\pd{\Phi}{x^i}\\
		      &= \frac{2}{s} (d\rho_{\frac{1}{s}}\circ dE\circ d\rho_s) \pd{F}{x^i}.
\end{align*}
Now, the matrix-valued function
\[
 A =  \frac{2}{s} (d\rho_{\frac{1}{s}}\circ dE\circ d\rho_s) \circ \Phi = 2\left(\begin{array}{cc}
                                                                      \frac{1}{s}\pd{E_{\psi}^{\alpha}}{y^\beta} & \pd{E_{\psi}^{\alpha}}{r}\\
                                                                      \frac{1}{s^2}\pd{E_{r}}{y^\beta} & \frac{1}{s}\pd{E_{r}}{r}\\
                                                                     \end{array}\right)
\]
satisfies $|A| \leq C_ls^{-(l+1)}$ for all $l$, so the function $\phi = \left|\pd{F}{x^i} - e_i\right|^2$ satisfies
\begin{align*}
 \pd{\phi}{s} &= 2\left\langle A\left(\pd{F}{x^i} - e_i\right),\pd{F}{x^i} - e_i\right\rangle 
	    + 2\left\langle A e_i,\pd{F}{x^i} - e_i\right\rangle\\
	    &\geq -3|A|\phi - |A|. 
\end{align*}
Fix $s_0 < s_1 \leq b$. Then, there is $C$ depending only on $l$ such that 
\[
\pd{\phi}{s}  \geq -C_ls^{-2}(\phi + s_1^{-2l})
\]
for any $x$ and all $s \geq s_1$. Integrating from $s_1$ to $b$ yields
\[
 \ln\left(\frac{\phi(x, b) + s_1^{-2l}}{\phi(x, s_1) + s_1^{-2l}}\right) \geq  \frac{C_l}{b} - \frac{C_l}{s_1}
\]
which, since $\phi(x, b) = 0$, means that
\[
 \phi(x, s_1) \leq e^{\frac{C}{s_1}-\frac{C}{b}}{s_1^{-2l}} \leq C_ls_1^{-2l},
\]
where $C$ is independent of $s_1$. Since $s_1$ was arbitrary, it follows that
\[
 \left|\pd{\psi^{\alpha}}{x^{i}} - \delta^{\alpha}_{i}\right|
+  \frac{1}{s}\left|\pd{r}{x^{i}}\right| \leq \frac{C_l}{s^l} 
\]
for all $s$, and the desired estimate follows.

The higher derivatives may be estimated similarly. We will give here the details only for the case $k =2$. Fix again $l\geq 0$.
From above, we have already seen that 
\[
 \frac{\partial^2r}{\partial s^2} = \pd{}{s}\left(\frac{r}{s} + \frac{2}{s}E_r\right) = \frac{2}{s}dE_r\pd{\Phi}{s},
\]
and
\[
 \frac{\partial^2\psi}{\partial s^2} =\pd{}{s} \left(\frac{2}{s}E_{\psi}\right) = -\frac{2}{s^2}E_{\psi} +\frac{2}{s} dE_{\psi}\pd{\Phi}{s},
\]
so
\[
 \left|\frac{\partial^2\psi}{\partial s^2}\right| + \left|\frac{\partial^2r}{\partial s^2}\right| \leq \frac{C_l}{s^l}
\]
for some $C_l$.
Similarly, 
\[
 \left|\frac{\partial^2 r}{\partial x^i \partial s}\right| \leq \frac{1}{s}\left|\pd{r}{x^i}\right| + \frac{2}{s}\left|dE_r\pd{\Phi}{x^i}\right| \leq C_ls^{-l},
\] 
and
\[
 \left|\frac{\partial^2 \psi}{\partial x^i \partial s}\right| \leq \frac{2}{s}\left|dE_{\psi}\pd{\Phi}{x^i}\right| \leq C_ls^{-l}
\]
for any $i$.

For the pure $x$-derivatives, we again use the map $F$ and compute that
\begin{align*}
  \pd{}{s} \frac{\partial^{2} F}{\partial x^i \partial x^j} &= \frac{2}{s}(d\rho_{\frac{1}{s}}\circ dE\circ d\rho_{s})\frac{\partial^2 F}{\partial x^i \partial x^j}\\
  &\phantom{=}
  + \frac{2}{s}\left(d\rho_{\frac{1}{s}} \circ d^2E\right)\left(d\rho_s \pd{F}{x^i}, d\rho_s \pd{F}{x^j}\right)
\end{align*}
for any $i$ and $j$.
Fixing any $x$ and integrating from $s$ to $b$, we may estimate as in the previous lemma that
\[
 \left|\frac{\partial^2 F}{\partial x^i \partial x^j}(x, s)\right| \leq \frac{C}{s^{l}}
\]
using that
\[
 \frac{\partial^2 F}{\partial x^i \partial x^j}(x, b) = 0.
\]
The desired estimate on $\frac{\partial^2\Phi}{\partial x^i\partial x^j}$ follows immediately.
\end{proof}

\subsection{Convergence to a limit diffeomorphism}
Now we are ready to extract a limit as $b\to \infty$ from the family $\Phi^{(b)}$ of local diffeomorphisms constructed in Proposition \ref{prop:phiexist}. We first fix a finite coordinate atlas
in order to import the estimates from the previous section to the cylinder.

It follows from the distance estimates \eqref{eq:transversec0} that we can cover $\Ss^{k}\times \Ss^{n-k-1}$ by a finite collection
$\{U^i_{\delta}\}_{i=1}^N$ of products 
\[
U^i_{\delta} = \mathring{B}_{\delta}^k(p_i)\times\mathring{B}^{n-k-1}_{\delta}(q_i)
\]
of coordinate balls of radius $\delta$ less than one fourth the injectivity radii of $\Ss^k$ and $\Ss^{n-k-1}$ with the property that
\[
   \Phi^{(b)}(\overline{U}_{2\delta}^i\times (a_2, \infty)) \subset U_{4\delta}^i \times (a_2/2, \infty)
\]
for all $a_2 \geq a_1$ sufficiently large (depending on $\delta$) and all $b \geq a_2$. 
Write $\tilde{U}^i = U_{2\delta}^i\times (a_2, \infty)$ and $\tilde{W}^i = U_{4\delta}^i\times (a_2/2, \infty)$
and consider the corresponding atlases $\{(\tilde{U}^i, \tilde{\varphi}^i)\}_{i=1}^N$ and $\{(\tilde{W}^i, \tilde{\varphi}^i)\}_{i=1}^N$ of $\Cc_{a_2}$ and $\Cc_{a_2/2}$, respectively.
Here we use $\tilde{\varphi}^i$ to represent both the map $\exp_{p_i}^{-1}\times \exp_{q_i}^{-1}\times \operatorname{Id}$ on $\tilde{W}^i$ and its restriction to $\tilde{U}^i$. 

Passing to the coordinate representations $\tilde{\varphi}^i\circ\Phi^{(b)}\circ(\tilde{\varphi}^i)^{-1}$ and $d\tilde{\varphi}^{i}(E) \circ (\tilde{\varphi}^i)^{-1}$ of $\Phi^{(b)}$ and $E$ (which we will continue to denote
by the same symbols) we obtain a system of the 
form \eqref{eq:psirssys} on $\tilde{\varphi}^i(\tilde{U}^i)$ with the bounds \eqref{eq:edecay} for some $C$ depending on $U^i$;
these bounds follow from \eqref{eq:decay4} since the coordinate representation of $g$ on $\RR^{n}$ satisfies
\[
  C^{-1}\delta_{jk} \leq g_{jk}(y, s) \leq Cs^2\delta_{jk}
\]
on $\tilde{U}^i$ for some $C > 0$ depending only on $i$, and we have bounds of the form $\left|\partial^{(m)}\Gamma_{jk}^l\right| \leq C(i, m)$ on $\tilde{U}^i$ for all $m\geq 0$. Here $y = (\theta, \sigma)$.

From Proposition \ref{prop:odeckestimate}, we obtain that, for fixed $i$ and  $a_2 < s_1 < s_2$, the $C^k$-norms of the coordinate representation of  $\Phi^{(b)} - \operatorname{Id}$ 
are uniformly bounded on the compact set $K =  \overline{U^i_{\delta}}\times [s_1, s_2]\subset \tilde{U}^i$ for each $k\geq 0$. From the Ascoli-Arzela theorem, then, there is a sequence $b_j \to \infty$ such that $\Phi^{(b_j)}$ 
converges in every $C^k$-norm to a smooth map $\Phi_K^{(\infty)}$ on $K$. Covering the annular regions
$A_j = \Ac_{a_2+ 1/j, ja_2}$ by finitely many of the charts from this atlas, we can obtain a smooth limit $\Phi_j^{(\infty)}$ on $A_j$ for each $j$; taking a further subsequence, we obtain a smooth limit
$\Phi = \Phi^{(\infty)}$ defined on all of $\Cc_{a_2}$. We record this statement and some additional observations in the following proposition.

\begin{proposition}
\label{prop:limitphi} Let $a_2$ be as in the discussion above.
There exists $a_3 \geq a_2$ and a sequence $b_j \to \infty$ such that $\Phi^{(b_j)}$ converges locally smoothly as $j\to \infty$ to a smooth map $\Phi: \Cc_{a_3}\to \Cc_{a_3/2}$ satisfying
\begin{enumerate}
 \item[(a)] $d\Phi_{(\theta, \sigma, s)}\left(X(\theta, \sigma, s)\right) = \tilde{X}\circ\Phi(\theta, \sigma, s)$,
\item[(b)] $\Phi$ is a diffeomorphism onto its image and $\Cc_{2a_3}\subset \Phi(\Cc_{a_3})$,

\item[(c)] On each coordinate neighborhood $U = U^i_{\delta}$ defined above, and for each $k$, $l\geq 0$, there is $C = C(i, k, l)$
such that, for all $s > a_3$,
\begin{equation}\label{eq:coordck}
      s^l\left\{\|\Phi - \operatorname{Id}\|_{C^k(U\times [s, 2s])} + \|\Phi^*g - g\|_{C^k(U\times [s, 2s])} \right\} \leq C, 
\end{equation}
relative to the Euclidean norm and connection.
\end{enumerate} 
\end{proposition}
\begin{proof} For now, we will assume just that $a_3 \geq a_2$ and further restrict $a_3$ as we work through the argument.
 The identity in (a) follows from \eqref{eq:phibvf} and the $C^1$-convergence of $\Phi^{(b_j)}$. The second claim in (b) follows from \eqref{eq:phibimage}, and the estimate on the first term in \eqref{eq:coordck} follows from Proposition \ref{prop:odeckestimate} and the discussion preceding the statement of this proposition.  In particular, we can choose $a_3$ sufficiently large 
 so that $(1/2)\operatorname{Id} \leq d\Phi \leq 2\operatorname{Id}$ on $U^i\times [a_3, \infty)$
 for each $i$. Among other things, this ensures that $\Phi$ will be a local diffeomorphism on $\Cc_{a_3}$. 
 
 The argument that $\Phi$ is injective goes then
 just as the corresponding argument for $\Phi^{(b)}$ in Proposition \ref{prop:phiexist}.  Here, as there, $r(s)$ is strictly increasing along the radial lines $s\mapsto (\theta, \sigma, s)$,
 and $\Phi$ is a diffeomorphism
 when restricted to $\Sc_t\times (t-\epsilon, t+ \epsilon)$ for some sufficiently large $t$ and sufficiently small $\epsilon$. Following the radial lines forward and backward as 
 in the proof of Proposition \ref{prop:phiexist}, we see that $\Phi$ must be injective on $\Cc_{a_3}$, and hence a diffeomorphism onto its image. Using the $C^0$-comparison of $r\circ \Phi$ with $s$,
 we can also enlarge $a_3$ if necessary to ensure that $\Phi(\Cc_{a_3}) \subset \Cc_{a_3/2}$ and $\Cc_{2a_3}\subset \Phi(\Cc_{a_3})$.
 
 Finally, the $C^k$ estimates on $\Phi^*g - g$ in \eqref{eq:coordck} follow from the uniform estimates we have on the derivatives of the coordinate representations of $\Phi-\operatorname{Id}$ and the metric $g$ on the neighborhoods 
 $\tilde{U}^i$.
 \end{proof}
 
\subsection{Proof of Theorem \ref{thm:vfnormalization}}  

Now we assemble the proof of Theorem \ref{thm:vfnormalization}. Taking $r_1 = a_3$ (and recalling that $a_3 \geq a_0 \geq 2r_0$), Proposition \ref{prop:limitphi} gives us the existence of a map
$\Phi: \Cc_{r_1}\to \Cc_{r_1/2} \subset \Cc_{r_0}$ satisfying that $\Phi_*X = X\circ\Phi$ and $\Cc_{2r_1}\subset \Phi(\Cc_{r_1})$.  Moreover (patching together estimates using the local bounds on
the Christoffel symbols), part (c) of that proposition ensures that
\[
  \sup_{\Cc_s}s^l|\nabla^{(m)}(\Phi^*g - g)| < \infty
\]
for all $l\geq 0$.  Writing $\hat{g} = \Phi^*g$ and $\hat{\nabla}$ for the connection of $\hat{g}$, we thus have
\begin{equation}
 \label{eq:conndiff}
 \sup_{\Cc_s}s^l|\nabla^{(m)}(\hat{\Gamma} - \Gamma)| < \infty,
\end{equation}
and, consequently,
\begin{equation}
\label{eq:hatest}
 \sup_{\Cc_s}s^l|\hat{\nabla}^{(m)}(\hat{g} - g)| < \infty,
\end{equation}
for all $l$ and $m$. 

But then, for all $l$, we have
\begin{align*}
   |\Phi^*\gt - g| &\leq |\Phi^*\gt - \hat{g}| + |\hat{g} - g| \leq C|\Phi^*\gt - \hat{g}|_{\hat{g}} + |\hat{g} - g|\\
   &=C|\gt - g|\circ \Phi + |\hat{g} - g| \leq C_l s^{-l} 
\end{align*}
for some $C_l$, using that both $\hat{g}$ and $\gt$ are strongly asymptotic to $g$ and that $r$ and $s$ are comparable.
We can then proceed inductively, using \eqref{eq:conndiff} and \eqref{eq:hatest} to estimate the covariant derivatives of $\gt- g$. For example,
since
\begin{align*}
  |\nabla(\Phi^*\gt - g)| &\leq C|\hat{\Gamma} - \Gamma||\Phi^*\gt - g| + |\hat{\nabla}(\Phi^*\gt - g)|\\
    &\leq  C|\hat{\Gamma} - \Gamma||\Phi^*\gt - g| + C|\hat{\nabla}(\Phi^*\gt - \hat{g})|_{\hat{g}} + |\hat{\nabla}(\hat{g} - g)|\\
    &= C|\hat{\Gamma} - \Gamma||\Phi^*\gt - g| + C|\nabla(\gt - g)|\circ \Phi + |\hat{\nabla}(\hat{g} - g)|,
\end{align*}
we see that we have a bound of the form $|\nabla(\Phi^*\gt - g)| \leq C_ls^{-l}$ for all $l$. We can argue similarly for the higher derivatives. This completes the proof.


\begin{thebibliography}{99}

\bibitem{BairdDanielo} P. Baird and L. Danielo,
\textit{Three-dimensional Ricci solitons which project to surfaces},
J. Reine Angew. Math. {\bf 608} (2007), 65--91, MR2339469, Zbl 1128.53020.

 \bibitem{BohmWilking} C. B\"ohm and B. Wilking,
  \textit{Manifolds with positive curvature operators are space forms},
 Ann. of Math. (2) {\bf 167} (2008), no. 3, 1079--1097,  MR2415394, Zbl 1185.53073. 

\bibitem{BrendleGenRFConvergence} S. Brendle, 
 \textit{A general convergence result for the Ricci flow in higher dimensions},
 Duke Math. J. {\bf 145} (2008), no. 3, 585--601, MR2462114,  Zbl 1161.53052.

 \bibitem{BrendleSchoen} S. Brendle and R. Schoen, 
 \textit{Manifolds with 1/4-pinched curvature are space forms}, 
 J. Amer. Math. Soc. {\bf 22} (2009), 287--307, MR2449060,  Zbl 1251.53021.

 \bibitem{CaoChen} H.-D. Cao and Q. Chen,
 {\it On Bach-flat gradient shrinking Ricci solitons},
  Duke Math. J., {\bf 162} (2013), no. 6, 1149--1169, MR3053567,  Zbl 1277.53036. 
 
  \bibitem{CaoChenZhu} H.-D. Cao, B.-L. Chen, and X.-P. Zhu,
  {\it Recent developments on Hamilton's Ricci flow }, 
 Surveys in differential geometry, Vol. XII, 47–112, Surv. Differ. Geom., XII, Int. Press,
  Somerville, MA, 2008, MR2488948, Zbl 1157.53002.
 
 \bibitem{CaoZhou} H.-D. Cao and D. Zhou,
 {\it On complete gradient shrinking Ricci solitons},
 J. Diff. Geom. {\bf 85} (2010), no. 2, 175--185, MR2732975, Zbl 1246.53051. 
 
 \bibitem{CaoWangZhang} X. Cao,  B. Wang, and Z. Zhang,
 {\it On locally conformally flat gradient shrinking Ricci solitons},
 Commun. Contemp. Math. {\bf 13} (2011), no. 2, 269--282, MR2794486, Zbl 1215.53061. 
 
 \bibitem{CarrilloNi} J. Carrillo and L. Ni,
 {\it Sharp logarithmic Sobolev inequalities on gradient solitons and applications},
 Comm. Anal. Geom. {\bf 17} (2009), no. 4, 721--753, MR3010626, Zbl 1197.53083.

 \bibitem{ChenStrongUniqueness} B.-L. Chen,
   {\it Strong uniqueness of the Ricci flow}, 
   J. Differential Geom. {\bf 82} (2009), no. 2, 363--382, MR2520796,  Zbl 1177.53036.
 
 \bibitem{ChenLuTian} X.-X. Chen, P. Lu, and G. Tian, 
 {\it A note on uniformization of Riemann surfaces by Ricci flow},
    Proc. Amer. Math. Soc. {\bf 134} (2006), no. 11, 3391--3393, MR2231924,  Zbl 1113.53042.

\bibitem{ChenWang} X.-X. Chen and Y. Wang, {\it On four-dimensional anti-self-dual gradient Ricci solitons}, J. Geom. Anal. {\bf 25} (2015), no. 2, 1335--1343, 
MR3319974, Zbl 1322.53041.   

\bibitem{Chow2D} B. Chow, {\it The Ricci flow on the 2-sphere},
J. Diff. Geom. {\bf 33} (1991), no. 2, 325--334, MR1094458, Zbl 0734.53033.
   
 \bibitem{ChowKnopf} B. Chow and D. Knopf,
   {\it The Ricci flow: an introduction},
    Mathematical Surveys and Monographs, 110, American Mathematical Society, Providence, RI, (2004), xii+325 pp,   
 MR2061425, Zbl 1086.53085.
 



\bibitem{RFV2P2} B. Chow, S.-C. Chu, D. Glickenstein, C. Guenther, 
   J. Isenberg, T. Ivey, D. Knopf, P. Lu, F. Luo, and L. Ni, 
   \textit{The Ricci flow: techniques and applications. Part II. Analytic aspects},
    Mathematical Surveys and Monographs, {\bf 144}, American Mathematical Society, Providence, RI, 2008, xxvi+458 pp., MR2604955, Zbl 1157.53035.
 

\bibitem{ChowLu}
B. Chow and P. Lu,
{\it On $\kappa$-noncollapsed complete noncompact shrinking gradient Ricci solitons which split at infinity},
Math. Ann. {\bf 366} (2016), no. 3-4, 1195--1206, MR3563235, Zbl 1355.53040.


\bibitem{ChowLuYang} B. Chow, P. Lu, and B. Yang, Bo,
{\it Lower bounds for the scalar curvatures of noncompact gradient Ricci solitons},
C. R. Math. Acad. Sci. Paris {\bf 349} (2011), no. 23-24, 1265--1267, MR2861997, Zbl 1230.53036. 



 \bibitem{DancerWang} A. Dancer and M. Wang,
 {\it On Ricci solitons of cohomogeneity one},
 Ann. Global Anal. Geom. {\bf 39} (2011), no. 3, 259--292, MR2769300, Zbl 1215.53040.

 \bibitem{EminentiLaNaveMantegazza} M. Eminenti, G. La Nave, and C. Mantegazza, 
 {\it Ricci solitons: the equation point of view},
 Manuscripta Math. {\bf 127} (2008), no. 3, 345--367, MR2448435,  Zbl 1160.53031. 

\bibitem{EndersMuellerTopping}  J. Enders, R. M\"{u}ller, and P. Topping, 
{\it On Type-I singularities in Ricci flow}, 
Comm. Anal. Geom. {\bf 19} (2011), no. 5, 905--922, MR2886712,  Zbl 1244.53074.


\bibitem{EscauriazaSereginSverakHalfSpace} L. Escauriaza, G. Seregin, and V. \v{S}ver\'{a}k,
  \textit{Backwards uniqueness for the heat operator in a half-space},
 St. Petersburg Math. J. {\bf 15} (2004), no. 1, 139--148, MR1979722, Zbl 1053.35052.

 \bibitem{FeldmanIlmanenKnopf} M. Feldman, T. Ilmanen, and D. Knopf,
 {\it Rotationally symmetric shrinking and expanding gradient K\"ahler-Ricci solitons},
 J. Diff. Geom. {\bf 65} (2003), no. 2, 169--209, MR2058261,  Zbl 1069.53036.

 \bibitem{FernandezLopezGarciaRio}  M. Fern\'andez-L\'opez and E. Garc\'ia-R\'io, 
 {\it Some gap theorems for gradient Ricci solitons}, Internat. J. Math. {\bf 23} (2012), no. 7, 
 1250072, 9 pp, MR2945651, Zbl 1247.53051.
 
 \bibitem{HamiltonSurfaces} R. Hamilton, {\it The Ricci flow on surfaces},
  Mathematics and general relativity (Santa Cruz, CA, 1986), 237--262, 
 Contemp. Math. {\bf 71}, Amer. Math. Soc., Providence, RI, 1988, MR0954419,  Zbl 0663.53031.    
 
 \bibitem{HamiltonSingularities} R. Hamilton,  
   {\it The formation of singularities in Ricci flow},
   Surveys in differential geometry, Vol. II (Cambridge, MA, 1993), 7--136, Int. Press, Cambridge, MA, 1995, MR1375255,  Zbl 0867.53030.
 

\bibitem{HaslhoferMueller1} R. Haslhofer and R. M\"uller,
{\it A compactness theorem for complete Ricci shrinkers},
Geom. Funct. Anal. {\bf 21} (2011), no. 5, 1091--1116, MR2846384, Zbl 1239.53056. 

\bibitem{HaslhoferMueller2} R. Haslhofer and R. M\"uller,
{\it A note on the compactness theorem for 4-$d$ Ricci shrinkers},
Proc. Amer. Math. Soc. {\bf 143} (2015), no. 10, 4433--4437, MR3373942, Zbl 1323.53046.


 \bibitem{Ivey3DSolitons} T. Ivey,
 {\it Ricci solitons on compact three-manifolds},
 Diff. Geom. Appl. {\bf 3} (1993), no. 4, 301--307, MR1249376,  Zbl 0788.53034.
 
 \bibitem{IveyLocalSoliton} T. Ivey,
 {\it Local existence of Ricci solitons}, Manuscripta Math. {\bf 91} (1996), no. 2, 151--162, MR1411650, Zbl 0870.53039.  
 
\bibitem{KotschwarBackwardsUniqueness} B. Kotschwar, 
   {\it Backwards uniqueness for the Ricci flow}, 
   Int. Math. Res. Not. (2010) no. 21, 4064--4097, MR2738351, Zbl 1211.53086.
   

\bibitem{KotschwarKaehlerShrinker} B. Kotschwar,
{\it K\"ahlerity of shrinking gradient Ricci solitons asymptotic to K\"ahler cones},
J. Geom. Anal. {\bf 28} (2018), no. 3, 2609--2623, MR3833809, Zbl 1407.53065.


\bibitem{KotschwarWangConical} B. Kotschwar and L. Wang,
\textit{Rigidity of asymptotically conical shrinking gradient Ricci solitons},
J. Diff. Geom. {\bf 100} (2015), no. 1, 55--108, MR3326574,  Zbl 06438782.  
 

\bibitem{LeeSmoothManifolds} J. Lee, 
{\it Introduction to smooth manifolds, 2nd ed.}, Graduate Texts in Mathematics, {\bf 218}, Springer, New York, 2013. xvi+708 pp. ISBN: 978-1-4419-9981-8,
MR2954043,  Zbl 1258.53002.

\bibitem{LiNiWang} X. Li,  L. Ni, and K. Wang,
{\it Four-dimensional gradient shrinking solitons with positive isotropic curvature},
Int. Math. Res. Not. (2018), no. 3, 949--959, MR3801452, Zbl 1405.53071.

\bibitem{LiWang} Y. Li and B. Wang,
{\it The rigidity of Ricci shrinkers of dimension four},
 Trans. Amer. Math. Soc. {\bf 371} (2019), no. 10, 6949--6972, MR3939566,  Zbl 1415.53050. 

\bibitem{Lott} J. Lott,
\textit{On the long-time behavior of type-III Ricci flow solutions},
Math. Ann. {\bf 339} (2007), no. 3, 627--666. 

\bibitem{LottWilson} J. Lott and P. Wilson,
Note on asymptotically conical expanding Ricci solitons,
Proc. Amer. Math. Soc. {\bf 145} (2017), no. 8, 3525--3529, MR3652804, Zbl 1366.53032. 

 \bibitem{MunteanuSesum} O. Munteanu and N. \v{S}e\v{s}um,
 {\it On gradient Ricci solitons}, J. Geom. Anal. {\bf 23} (2013), no. 2, 539--561, 
 MR3023848, Zbl 1275.53061.

\bibitem{MunteanuWang1} O. Munteanu and J. Wang,
{\it Analysis of weighted Laplacian and applications to Ricci solitons}, 
Comm. Anal. Geom. {\bf 20} (2012), no. 1, 55--94  MR2903101,  Zbl 1245.53039. 

\bibitem{MunteanuWangKaehler}  O. Munteanu and J. Wang,
{\it Topology of K\"ahler Ricci solitons}, J. Differential Geom. {\bf 100} (2015), no. 1, 109--128, 
MR3326575,  Zbl 1321.53083.

\bibitem{MunteanuWang2} O. Munteanu and J. Wang,
{\it Geometry of shrinking Ricci solitons}, Compos. Math. {\bf 151} (2015), no. 12, 2273--2300, 
MR3433887,  Zbl 1339.5303.
 
\bibitem{MunteanuWang3} O. Munteanu and J. Wang,
{\it Conical structure for shrinking Ricci solitons},  J. Eur. Math. Soc. {\bf 19} (2017), no. 11, 3377--3390, 
MR3713043, Zbl 06802928. 

\bibitem{MunteanuWang4} O. Munteanu and J. Wang,
{\it Structure at infinity for shrinking Ricci solitons}, Ann. Sci. \`{E}c. Norm. Sup\'{e}r. (4) {\bf 52} (2019), no. 4, 891--925, MR4038455, Zbl 07144475.
 
\bibitem{MunteanuWangPositiveCurvature} O. Munteanu and J. Wang, {\it Positively curved shrinking Ricci solitons are compact},
J. Diff. Geom. {\bf 106} (2017), no. 3, 499--505, MR3680555, Zbl 1405.53072. 

\bibitem{Naber4D} A. Naber,
{\it Noncompact shrinking four solitons with nonnegative curvature}, 
J. Reine Angew. Math. {\bf 645} (2010), 125--153, MR2673425,  Zbl 1196.53041. 

\bibitem{NiWallach} L. Ni and N. Wallach,
{\it On a classification of gradient shrinking solitons},
Math. Res. Lett. {\bf 15} (2010) no. 5, 941--955, MR2443993, Zbl 1158.53052.

\bibitem{Perelman1} G. Perelman,  
{\it The entropy formula for the Ricci flow and its geometric applications}, 
{\tt arXiv:math/0211159 [math.DG]}.

 
\bibitem{Perelman2} G. Perelman,  
{\it Ricci flow with surgery on three-manifolds}, 
{\tt arXiv:math/0303109 [math.DG]}.
 
 
\bibitem{PetersenWylie} P. Petersen and W. Wylie,
{\it On the classification of gradient Ricci solitons},
Geom. Topol. {\bf 14} (2010), no. 4, 2277--2300, MR2740647, Zbl 1202.53049. 

\bibitem{SesumTypeI} N. \v{S}e\v{s}um, 
{\it Limiting behavior of Ricci flows}, Ph.D. Thesis, Massachusetts Institute of Technology, (2004).

\bibitem{WangConical} L. Wang,
{\it Uniqueness of self-similar shrinkers with asymptotically conical ends}, 
J. Amer. Math. Soc. {\bf 27} (2014), no. 3, 613--638, MR3194490, Zbl 1298.53069. 
 
\bibitem{WangCylindrical} L. Wang,
{\it Uniqueness of self-similar shrinkers with asymptotically cylindrical ends},
J. Reine Angew. Math. {\bf 715} (2016), 207--230, MR3507924,  Zbl 1343.53069.

\bibitem{Wylie} W. Wylie, \emph{Complete shrinking Ricci solitons have finite fundamental group},
Proc. Amer. Math. Soc. {\bf 136} (2008), 1803--1806, MR2373611, Zbl 1152.53057. 


\bibitem{Yang} B. Yang,
{\it A characterization of noncompact Koiso-type solitons},
Internat. J. Math. {\bf 23} (2012), no. 5, 1250054, 13 pp, MR2914656, Zbl 1241.53059. 

\bibitem{ZhangCompleteness} Z.-H. Zhang,
{\it On the completeness of gradient Ricci solitons},
Proc. Amer. Math. Soc. {\bf 137} (2009), no. 8, 2755--2759, MR2497489,  Zbl 1176.53046. 
  
\bibitem{ZhangWeyl} Z.-H. Zhang,
{\it Gradient shrinking solitons with vanishing Weyl tensor},
Pacific J. Math. {\bf 242} (2009), no. 1, 189--200, MR2525510, Zbl 1171.53332. 

\end{thebibliography}
\end{document}